\documentclass[12pt]{amsart}
\usepackage{amssymb}
\usepackage[all]{xy}
\newtheorem{theorem}{Theorem}[section]
\newtheorem{proposition}[theorem]{Proposition}
\newtheorem{coro}[theorem]{Corollary}
\newtheorem{lemma}[theorem]{Lemma}
\newtheorem{rem}[theorem]{Remark}

\renewcommand{\epsilon}{\varepsilon}

\newcommand\C{\mathbb{C}}
\newcommand\N{\mathbb{N}}
\newcommand\R{\mathbb{R}}

\newcommand\lp{L^p(M,\mu)}
\newcommand\lqq{L^q(M,\mu)}

\newcommand\restr[2]{{
  \left.\kern-\nulldelimiterspace 
  #1 
  \vphantom{\big|} 
  \right|_{#2} 
  }}

\numberwithin{theorem}{subsection}
\numberwithin{equation}{section}

\title[A new approach to heat kernel upper bounds]{A new approach to pointwise  heat kernel upper bounds on doubling metric measure spaces}
\author{S. Boutayeb}
\address{Salahaddine Boutayeb, Facult\'e polydisciplinaire de Safi, Universit\'e 
Cadi Ayyad,  46000 Safi,
Maroc}
\email{sboutaye@gmail.com}

\author{T. Coulhon}
\address{Thierry Coulhon, Mathematical Sciences Institute, The Australian National University, Canberra ACT 0200, Australia}
\email{thierry.coulhon@anu.edu.au}

\author{A. Sikora}
\address{Adam Sikora, Department of Mathematics, Macquarie
University, NSW 2109, Australia}
\email{adam.sikora@mq.edu.au}

\thanks{TC's and AS's research was  supported by an 
Australian Research Council (ARC) grant DP  130101302}

\date{\today}

\begin{document}
\begin{abstract} On doubling metric measure spaces  endowed with a strongly local regular Dirichlet form, we show some characterisations of pointwise upper bounds
of the heat kernel in terms of   global scale-invariant inequalities that
correspond respectively to the Nash inequality and to a
Gagliardo-Nirenberg type inequality when the volume growth is polynomial. This yields a new  proof and a generalisation of the well-known equivalence between classical heat kernel upper bounds and  relative Faber-Krahn inequalities or localized Sobolev or Nash inequalities. We are  able to treat more general pointwise estimates, where the heat kernel rate of decay is not necessarily governed by the volume growth. A crucial role is played by the finite  propagation speed property for the associated wave equation, and
our main result holds for an abstract semigroup of operators satisfying the Davies-Gaffney estimates.
\end{abstract}

\maketitle

\tableofcontents
\section{Introduction}

\subsection{Background and motivation}

Let $M$ be a complete non-compact  connected Riemannian manifold,  $d$ is the geodesic distance and $\mu$ the Riemannian measure on $M$.  Denote by
$V(x,r):=\mu(B(x,r))$  the volume  of the ball $B(x,r)$ of center $x\in M$ and radius $r>0$ with respect to $d$.

Let $\Delta$ be
the non-negative Laplace-Beltrami operator and  $p_{t}$ be the heat kernel of $M$, that is by definition the smallest
positive fundamental solution of the heat equation:
\begin{equation*}
\frac{\partial u}{\partial t}+\Delta u=0,
\end{equation*}
or the kernel of the heat semigroup $e^{-t\Delta}$, i.e.
$$e^{-t\Delta}f(x)=\int_{M}p_{t}(x,y)f(y)d\mu(y),~~ f\in L^{2}(M,\mu),~\mu-\mbox{a.e. }~x\in
M.$$
It is well-known that in this situation, contrary to more general ones, $p_t(x,y)$ is smooth in $t>0$, $x,y\in M$ and everywhere positive (see for instance \cite{G4}).

\bigskip
In the Euclidean space $ \mathbb{R}^{n}$,  $p_{t}$ is given by the
classical Gauss-Weierstrass kernel:
\begin{equation*}
p_{t}(x,y)=\frac{1}{(4\pi
t)^{n/2}}\exp\left(-\frac{|x-y|^{2}}{4t}\right), \, t>0,\,x,y\in \mathbb{R}^{n}.
\end{equation*}
On general  manifolds, where of course no such formula is available, the subject of upper estimates of the heat kernel has led to an intense activity in the last three decades (see for instance  \cite{Cou5, G3, G4, SA, SF} for references and background).

One says that $M$ satisfies the volume doubling property if there exists $C>0$ such that:
    \begin{equation*}\tag{$V\!D$}
     V(x,2r)\le C V(x,r),\quad \forall~x \in M,~r > 0.
    \end{equation*}

    For such manifolds, a typical upper estimate for $p_t(x,x)$ (so-called diagonal upper estimate)
     is
  \begin{equation}\tag{$DU\!E$}
p_{t}(x,x)\leq
\frac{C}{V(x,\sqrt{t})},\forall~t>0,\,x\in
 M.\label{due}
\end{equation}
This estimate holds in particular on manifolds with non-negative Ricci curvature, see \cite{LY}.

Under  $(V\!D)$,  $(DU\!E)$ is equivalent to the  apparently stronger Gaussian upper estimate
\begin{equation}\tag{$U\!E$}
p_{t}(x,y)\leq
\frac{C}{V(x,\sqrt{t})}\exp
\left(-\frac{d^{2}(x,y)}{Ct}\right),\forall~t>0,\,x,y\in
 M,\label{UE}
\end{equation}
 see \cite[Theorem 1.1]{Gr1},  also \cite[Corollary 4.6]{CS} and the references therein.

A fundamental characterisation of $(U\!E)$ or $(DU\!E)$ was found by Grigor'yan. One says that $M$ admits the
relative Faber-Krahn inequality if there exists $c>0$ such that, for any ball $B(x,r)$ in $M$
and any open set $\Omega\subset B(x,r)$:
\begin{equation*}\label{FKalpha}
\tag{$FK$} \lambda_1(\Omega)\geq
\frac{c}{r^{2}}\left(\frac{V(x,r)}{\mu(\Omega)}\right)^{\alpha},
\end{equation*}
where $c$ and $\alpha$ are some positive constants and $\lambda_1(\Omega)$ is the smallest Dirichlet eigenvalue of
$\Delta$ in $\Omega$. It was  proved in \cite{G2} that $(FK)$ is equivalent to the
upper bound (\ref{due}) together with $(V\!D)$. The proof that  $(FK)$ implies (\ref{due})  is difficult. It goes through a mean value inequality for solutions of the heat equation which is proved via a non-trivial Moser type iteration. One then deduces  (\ref{due}) from this mean value inequality by using either   the integrated maximum principle (see \cite[chapter 15]{G4}) or   the Davies-Gaffney estimates which will play an important role in the present article (see  \cite[Section 5]{CG3} and \cite[Theorem 4.4]{AB}).


It turns out that the relative Faber-Krahn inequality is equivalent to the following family of localised Sobolev inequalities introduced by  Saloff-Coste (see \cite{S} and also \cite[section 2.3]{SF}): there exists $C>0$ and $q>2$ such that,
for every ball $B=B(x,r)$ in $M$ and for every $f\in \mathcal{C}_0^{\infty}(B)$,
\begin{equation}\label{ISo}\tag{$LS_q$}\left(\int_B|f|^{q}\,d\mu\right)^\frac{2}{q}\le \frac{Cr^2}{V^{1-\frac{2}{q}}(x,r)}\int_B\left(|\nabla f|^2+r^{-2}|f|^2\right)\,d\mu.
\end{equation}
This family of inequalities  implies in turn by H\"older's inequality a  family of localised Nash inequalities
\begin{equation}\label{INa}\tag{$LN_\alpha$}\left(\int_B|f|^2\,d\mu\right)^{1+\alpha}\le \frac{Cr^2}{V^{\alpha}(x,r)}\left(\int_B|f|\,d\mu\right)^{2\alpha}\int_B\left(|\nabla f|^2+r^{-2}|f|^2\right)\,d\mu,
\end{equation}
where $\alpha=1-\frac{2}{q}>0$, and in fact the methods of \cite{BCLS} show that they are equivalent.
One can  prove directly  that  (\ref{ISo}) implies (\ref{due}) (see \cite[Section 5.2]{SA} and the references therein; this has been extended in \cite{ST} to a  more general Dirichlet form setting), but the proof  again relies on the Moser iteration process.    A direct proof of the implication from   (\ref{due}) to (\ref{ISo}) is implicit in \cite[Theorem 10.3]{SU} and can be found also in 
\cite[Theorem 2.6, proof of Lemma 4.4]{HSC}, but it is not straightforward either.
\bigskip

In the case where there exists $C,n>0$ such that
$$C^{-1}r^n\le V(x,r)\leq Cr^{n},$$ one says that the volume growth is polynomial with exponent $n$.
This is a much more restrictive and less natural condition than $(V\!D)$, but  in that situation the characterisation of heat kernels upper bounds turns out to be much easier.
Indeed,  the upper bound \eqref{due} then reads
\begin{equation}\label{UB1}
p_{t}(x,x)\leq Ct^{-n/2},~~\forall~t>0,~x\in
 M,
\end{equation}
and the fact that this estimate is uniform (meaning that the RHS is independent of $x\in M$) allows one to make use of purely functional analytic methods, which yield  many characterisations of \eqref{UB1} in terms of Sobolev type inequalities.
First, a necessary and sufficient condition
for this upper bound is the Sobolev inequality:
\begin{equation*}
\|f\|_{\frac{np}{n-\alpha p}}\leq C\|\Delta^{\alpha/2}f\|_{p},\quad
\forall f\in \mathcal{C}^{\infty}_{0}(M),
\end{equation*}
for  $p>1$ and $0<\alpha p<n$, see \cite{Var} and \cite{Cou1}.
Note especially, when $n>2$, the particular case
$\alpha=1$, $p=2$ of this inequality:
\begin{equation}\label{sobo}
\|f\|_{\frac{2n}{n-2}}^2\leq C\mathcal{E}(f),\quad
\forall f\in \mathcal{C}^{\infty}_{0}(M),
\end{equation}
where $$\mathcal{E}(f):=\|\Delta^{1/2}f\|_{2}^2=<\Delta f,f>=\||\nabla f|\|_{2}^2$$
is the Dirichlet form associated with the Laplace-Beltrami operator.
Also
 equivalent to (\ref{UB1}) is the
Nash inequality:
\begin{equation}\label{N1}
\|f\|_{2}^{2+\frac{4}{n}}\leq C\|f\|_{1}^{4/n}\mathcal{E}(f),\quad
\forall f\in \mathcal{C}^{\infty}_{0}(M)
\end{equation}
(for this equivalence, see \cite{CKS}, and for generalisations see \cite{C-N}). Yet another characterisation of
(\ref{UB1}) is given by the Gagliardo-Nirenberg type inequalities:
\begin{equation}\label{GN1}
\|f \|_{q}^2\leq
C\|f\|_{2}^{2-\frac{q-2}{q}n}\mathcal{E}(f)^{\frac{q-2}{2q}n},\quad
\forall f\in \mathcal{C}^{\infty}_{0}(M),
\end{equation}
for   $q>2$ such that $\frac{q-2}{q}n<2$, see \cite{Cou2} for such
inequalities and other extensions.
For instance,  if one takes $q=2+\frac{4}{n}$ (in which case  the above conditions on $q$ are satisfied), then
\eqref{GN1} is the well-known Moser inequality
$$
\|f \|_{2+\frac{4}{n}}^2\leq
C\|f\|_{2}^{\frac{4}{n+2}}\mathcal{E}(f)^{\frac{n}{n+2}}.
$$
Note also that in \eqref{GN1} the limit case $\frac{q-2}{q}n=2$, that is $q=\frac{2n}{n-2}$, is nothing but \eqref{sobo}.
Let us insist on the fact that the proofs of the above equivalences  work in the general setting of a  measure space endowed with a Dirichlet form.
For more on this, see \cite{Cou5}, and for a summary  in book form, see \cite[Section 6.1]{O}.

\bigskip

The equivalences between \eqref{UB1} on the one hand,  \eqref{sobo}, \eqref{N1}, and \eqref{GN1} on the other hand do not use the fact that the majorizing function $Ct^{-n/2}$ is tied to the volume growth via $V(x,r)\simeq r^{n}$. This raises the question 
of characterising  estimates of the type
\begin{equation*}
p_{t}(x,x)\leq m(x,t),\forall~t>0,\,x\in
 M,
\end{equation*}
where $m$ is not necessarily linked to the volume function.

\bigskip

The aim of the present paper is, assuming the volume doubling property $(V\!D)$ instead of the more restrictive polynomial volume growth property, to establish new
characterisations of the upper bound  $(DU\!E)$ in terms of two types of
one-parameter weighted inequalities, which coincide respectively with the Nash inequality (\ref{N1})
and with the Gagliardo-Nirenberg type inequalities (\ref{GN1}) when the volume growth happens to be polynomial of exponent $n$. We will provide a proof of these characterisations that does not rely on Grigor'yan's  theorem on the equivalence between relative Faber-Krahn inequalities $(FK)$ and the heat kernel upper bound $(DU\!E)$.  As a matter of fact,  we will obtain as a by-product a new proof of this equivalence, also of the one with families of localised Sobolev inequalities. All this will rely (as in the uniform case) on functional analytic methods as opposed to PDE methods such as the Moser iteration process. 

Further interesting features of our approach are the following: instead of considering
a family of inequalities indexed by all balls, we deal with  global inequalities (with a scale parameter though); the fact that $(DU\!E)$ implies such inequalities is rather straightforward; the converse relies on a new argument with respect to  the preceding proofs, namely the finite  propagation speed of the wave equation (note that the latter, in its equivalent form of Davies-Gaffney estimate, is also an underlying principle of the equivalence between $(DU\!E)$ and $(U\!E)$, see \cite{CS}); more importantly, we shall consider a more general form $(DU\!E^v)$ of $(DU\!E)$, where the volume function $V$ is replaced by a more general doubling function $v$, and we shall prove  the equivalence between  $(DU\!E^v)$ in the one hand,  and matching versions $(N^v)$ and $(GN_q^v)$ of Nash and Gagliardo-Nirenberg  inequalities on the other hand.  We shall also show that the latter inequalities are equivalent to their  localised Sobolev and Nash counterparts and also to a more general version of the relative Faber-Krahn inequality. Finally, we shall work in a much more general setting than the Riemannian one.

\bigskip


\subsection{Framework and main results}\label{FMR}

Let 
$(M,\mu)$ be a $\sigma$-finite measure space.  Denote by $L$ a
non-negative self-adjoint operator on $L^2(M,\mu)$ with  dense domain $\mathcal{D}$.
The  quadratic form $\mathcal{E}$ associated with $L$ is defined, for $f,g\in\mathcal{D}$, by
$$\mathcal{E}(f,g):=<Lf,g>,$$ where $<\cdot,\cdot>$ is the scalar product in
$L^{2}(M,\mu)$.
We shall abbreviate $$\mathcal{E}(f):=\mathcal{E}(f,f)=\|L^{1/2}f\|_2^2.$$ Let $\mathcal{F}$ the domain of $\mathcal{E}$, which is usually larger than $\mathcal{D}$.  The form $\mathcal{E}$ is closed, symmetric, non-negative.
By spectral theory, the operator $-L$
generates an analytic contraction semigroup $(e^{-tL})_{t>0}$ on
$L^{2}(M,\mu)$. For $1\leq p\leq +\infty $ we denote the norm of a function $f$
in $L^{p}(M,\mu)$ by $\|f\|_{p}$  and, if $T$ is a  bounded linear operator from
$L^{p}(M,\mu)$ to $L^{q}(M,\mu)$, $1\le p\le q\le +\infty$, we denote by
$\|T\|_{p\to q}$ the operator norm of $T$. If $A$ is an unbounded operator acting on $L^p(M,\mu)$, 
$\mathcal{D}_p(A)$ will denote its domain.

Let $v(x,r)$ be a function of $x\in M$ and $r>0$, measurable in $x$, a.e. finite and positive,  and non-decreasing in $r$ for a.e. fixed $x$.
These will be standing assumptions that we will call $(A)$.

We shall often have to assume also that  $v$ is doubling, in the sense that  there exists $C>0$ such that
\begin{equation*}\tag{$D_v$}
 v(x,2r)\le C v(x,r), \forall\,r>0, \,\mbox{ for  }\mu-\mbox{a.e. }\, x\in M.
\end{equation*}
As a consequence of $(D_v)$,  there exist
positive constants $C$ and $\kappa_v$ such that
\begin{equation}\label{dv}\tag{$D_v^{\kappa_v}$}\
     v(x,r)\le C\left(\frac{r}{s}\right)^{\kappa_v} v(x,s),\quad \forall~
r \ge s>0,\,\mbox{ for  }\mu-\mbox{a.e. }\,  x \in M.
  \end{equation}

From Section \ref{LG0} on, we shall consider the case where $M$ is endowed with a metric $d$ (and $\mu$ is Borel). 
In that situation, we may need  to assume  in addition to $(D_v)$:
\begin{equation*}\label{D2}
\tag{$D'_v$}\  v(y,r)\le C v(x,r), \forall\,x,y\in M,\,r>0,\,d(x,y)\le r.
\end{equation*}

One should compare the above definitions with the notion of doubling gauge in  \cite{kig}. Note that we do not need to assume $\inf_{x\in M}v(x,r)>0$ for some $r>0$.

Let again $B(x,r):=\{y\in M,\,d(x,y)<r\}$ be the open ball in $M$ for the distance $d$, of center $x\in M$ and radius $r>0$.
Assume that $V(x,r):=\mu(B(x,r))$ is finite and positive for all $x\in M, r>0$. Exactly as in the case of Riemannian manifolds, define property $(V\!D)$, which may or may not be satisfied by $(M,d,\mu)$:
  \begin{equation*}
     V(x,2r)\le C V(x,r),\quad \forall~x \in M,~r > 0.
    \end{equation*}
    and, if it is the case, let $\kappa>0$ be such that :
     \begin{equation*}\label{d}\tag{$V\!D_\kappa$}
      V(x,r)\le C\left(\frac{r}{s}\right)^{\kappa} V(x,s),\quad \forall~
r \ge s>0,~ x \in M.
    \end{equation*}
In other terms, $(V\!D)$ is nothing but $(D_V)$, $\kappa=\kappa_V$, and in that case, it is easy to check that $(D'_V)$ always holds.
We shall sometimes say in short that  $(M,d,\mu)$ is a doubling metric measure space  meaning that it satisfies $(V\!D)$.

When $(V\!D)$ is satisfied, a typical example of  function $v$ satisfying $(D_v)$ is 
 $v(x,r):=V^\alpha(x,r^\beta)$, $\alpha,\beta>0$; if $\beta=1$ then
 \eqref{D2} is satisfied in addition.
  Another interesting example is $v(x,r):=V(x,\min (r,r_0))$, which satisfies  $(D_v)$  and  \eqref{D2}  as soon as $(M,d,\mu)$ satisfies $(V\!D_{loc})$, that is $(V\!D)$ for $0<r\le r_0$.
 As a consequence,  the family of general pointwise heat kernel upper estimates $(DU\!E^v)$ defined below encompasses $(DU\!E_{loc})$, that is  $(DU\!E)$  for $0<t\le t_0$.
 
 Note by the way that, except in Section \ref{KLR}, we  treat finite as well as infinite measure spaces, and compact as well as non-compact metric measure spaces (recall that according to an observation by Martell,  under $(V\!D)$ compactness is equivalent to finiteness of the measure if  balls in $M$ are precompact, see \cite[Corollary 5.3]{GH}).
 
    Let us finally record a consequence of $(V\!D)$, which we will use several times in the sequel, even in the presence of a function $v$ instead of $V$, namely the bounded covering principle $(BC\!P)$:
  for every $r>0$,   there exists a sequence $x_i \in M$ such that
$d(x_i,x_j)> r$ for $i\neq j$ and $\sup_{x\in M}\inf_i d(x,x_i)
\le r$. The balls $B(x_i,r/2)$ are pairwise disjoint and, for all $\theta> 1/2$, there exists $K_0$  only depending on $\theta$ and on the constant in $(V\!D)$ such that
 $$K(x):=\#\{i\in I,\,x\in B(x_i,\theta r)\}\le K_0,\ \forall\,x\in M.$$

Let us turn now towards the heat kernel pointwise upper estimates. Even in the case $v=V$, the definition of the heat kernel upper bound $(DU\!E^v)$ requires some adaptation from the Riemannian setting to the metric measure space setting.
First, in our general setting the existence of a measurable heat kernel is not granted, and it will be part of the definition of
$(DU\!E^v)$ that
$(e^{-tL})_{t>0}$  has a measurable kernel
$p_{t}$, that is
$$e^{-tL}f(x)=\int_{M}p_{t}(x,y)f(y)d\mu(y),~t>0,~ f\in L^{2}(M,\mu),~\mbox{ for  }\mu-\mbox{a.e. }~x\in
M.$$ 
This being said, even if $p_t$ exists as a measurable function of $(x,y)\in M\times M$,
the expression $p_t(x,x)$ is not properly defined in general,  since the diagonal is a set of measure zero in $M\times M$.
The following definition overcomes this difficulty:
we shall say that $(M,\mu,L,v)$ satisfies $(DU\!E^v)$ if $(e^{-tL})_{t>0}$ has a measurable kernel $p_t$ and there exist $C,c>0$ such that
\begin{equation*}
|p_{t}(x,y)|\leq
\frac{C}{\sqrt{v(x,c\sqrt{t})v(y,c\sqrt{t})}},\mbox{ for all }t>0,~\mbox{ for  }\mu-\mbox{a.e. }\,x,y \in
M.
\end{equation*}
If $v$ satisfies $(D_v)$, one can obviously take $c=1$ in the above estimate.

If $p_t$ happens to be continuous in $x,y$, the above inequality holds for all $x,y$.  Taking $x=y$ and observing that $p_t(x,x)\ge 0$
yields the original form of the estimate
\begin{equation*}
p_{t}(x,x)\leq
\frac{C}{v(x,c\sqrt{t})},\mbox{ for all }t>0,\,x \in
M,
\end{equation*}
with $c=1$ if $v$ satisfies $(D_v)$.
Conversely, using the symmetry of the kernel  $p_t(x,y)$ and the semigroup
property, it is easy to prove that $$|p_t(x,y)|\leq \sqrt{p_t(x,x)p_t(y,y)},$$
for all $t>0$, $x,y\in M$, hence the two forms of $(DU\!E^v)$ are equivalent as soon as they both have a meaning.

The reader may wonder why, since we make almost no assumptions on $v$, we do not write the upper estimate under consideration in a more compact form like
$$p_t(x,x)\le m(x,t)$$
when $p_t$ is continuous, or
$$|p_{t}(x,y)|\leq \sqrt{m(x,t)m(y,t)}$$
in general. The advantage of our choice is that it makes  the comparison with the classical case $v=V$ easier.
Our notation is adapted to a classical time-space scaling $t=r^2$.  One can of course easily change this notation in  order to treat the sub-Gaussian case, but this raises other questions, to which we will come back in a future work.
Finally,  we keep the apparently useless constant $C$ in the definition of $(DU\!E^v)$ in order to stress the fact that the equivalences
we will show between $(DU\!E^v)$ and the  functional inequalities we are going to consider are up to a multiplicative constant.

Note that there are a posteriori some limitations on $v$: our results in Section \ref{KLR} will rely on the assumption that $v$ is bounded from below by $V$. In the opposite direction, in many situations, $v$ cannot be substantially larger than $V$.
This follows from the fact which we are about to explain that $(DU\!E^v)$ implies a Gaussian upper bound $(U\!E^v)$; then, if $v$ is too large, $\int_Mp_t(x,y)\,d\mu(y)$ cannot be uniformly bounded from below, therefore $(e^{-tL})_{t>0}$ cannot preserve the constant function 
$1$ (the so-called conservativeness property). More precisely, it follows from \cite[Theorem 6.1, see also beginning of Section 7]{CG1} that, at least in a Riemannian situation, $(DU\!E^v)$ can only hold if $v(x,r)\le CV(x,r\log r)$, $r\ge 2$; but this relies strongly on the conservativeness property (also called stochastic completeness), and there are many interesting situations where this property does not hold, for instance Dirichlet boundary conditions. In any case, there is no reason to tie a priori $v$ to $V$.

If $(M,\mu)$ is endowed with a metric $d$, the full Gaussian estimate $(U\!E^v)$ can be formulated essentially in the same way as in the Riemannian setting:
\begin{equation}\tag{$U\!E^v$}
\left|p_{t}(x,y)\right|\leq
\frac{C}{\sqrt{v(x,\sqrt{t})v(y,\sqrt{t})}}\exp
\left(-\frac{d^{2}(x,y)}{Ct}\right),\label{UEv}
\end{equation}
$\forall~t>0$,  for $\mu-\mbox{a.e. }\,x,y\in
 M$.
 If $v$ satisfies $(D_v)$ and \eqref{D2}, \eqref{UEv} can be rewritten in the simpler form
 \begin{equation*}
\left|p_{t}(x,y)\right|\leq
\frac{C}{v(x,\sqrt{t})}\exp
\left(-\frac{d^{2}(x,y)}{Ct}\right),\forall~t>0,~\mbox{ for  }\mu-\mbox{a.e. }\,x,y\in
 M.
\end{equation*}

\bigskip

Let us now briefly introduce a major but very general assumption, namely the Davies-Gaffney estimate.  Let $(M,d,\mu)$  be a metric measure space and $L$  a non-negative self-adjoint operator on $L^2(M,\mu)$ with dense domain. For $U_1, U_ 2$ open subsets of $M$, let
$d(U_1,U_2)=\inf\limits_{x\in U_1, y\in U_2}{d}(x,y)$. One  says that $(M,d,\mu,L)$, or in short $L$, satisfies
 the
Davies-Gaffney estimate if
\begin{equation}  \label{DG2}\tag{$DG$}
| \langle e^{-tL}f_1,f_2\rangle | \le
\exp\left(-\frac{r^2}{4t}\right) \|f_1\|_{2} \|f_2\|_{2},
\end{equation}
for all $t >0$, $U_i \subset M$, \ $f_i\in L^2(U_i,\mu)$, $i=1,2$ and $%
r=d(U_1,U_2)$. Davies-Gaffney
estimate holds for essentially all self-adjoint, elliptic or subelliptic
second-order differential operators including Laplace-Beltrami operators on
complete Riemannian manifolds, Schr\"odinger operators with real-valued
potentials and electromagnetic fields; as we already said, it is equivalent to the finite propagation speed of the wave equation. For more information, see for instance \cite{CS} and the beginning of Section \ref{DGE}. 
Davies-Gaffney estimate also holds for measure spaces endowed
 with  a strongly local regular Dirichlet form and the associated intrinsic metric, see \cite{ST} and also \cite{HR}, \cite{AH} for the same estimate with an optimal metric. Note however that this intrinsic metric can degenerate, or be discontinuous,
and in such instances our approach cannot  be used.
For example in the case of fractals, such metrics degenerate. Indeed
in this case $d(x,y)=0$ for all $x,y$ from the ambient space and of course
it is impossible to use this  intrinsic metric in a meaningful way. Recall finally that $(DU\!E^v)$ and $(U\!E^v)$ are equivalent under $(DG)$ and $(D_v)$ (see  \cite[Section 4.2]{CS}).

 \bigskip

Let us finally introduce the functional inequalities that are going to generalise \eqref{N1} and \eqref{GN1}. Denote  
$$
v_r(x):=v(x,r), \,r>0,\,x\in M.
$$
For $v$ as above (not necessarily satisfying $(D_v)$ or \eqref{D2}), consider the inequality 
\begin{equation*}\label{NB}
\tag{$N^v$}\|f \|_{2}^2\leq
C(\|fv_{r}^{-1/2}\|_{1}^2+r^2\mathcal{E}(f)), \quad
\forall\, r>0, \,f\in \mathcal{F}.
\end{equation*}
Of course, our understanding is that if  the RHS is infinite (here because  $\|fv_{r}^{-1/2}\|_{1}$ is infinite) then \eqref{NB} holds. The same applies to the inequalities we will consider in the sequel.
When $v(x,r)\simeq r^n$ for some $n>0$,  for instance when $M$ is endowed with a metric $d$, $v=V$, and $(M,d,\mu)$ has  polynomial volume growth of exponent $n$, \eqref{NB} yields
\begin{equation*}
 \|f \|_{2}^2\leq
C'(r^{-n}\|f\|_{1}^2+r^2\mathcal{E}(f)),\quad
\forall\, r>0, \, f\in \mathcal{F},
\end{equation*}
which, as one sees by minimising in $r$, has exactly the same form as \eqref{N1}. This why we shall call \eqref{NB} a $v$-Nash inequality, or in short a Nash inequality.

The following  inequality  was introduced in \cite{kig}:
\begin{equation*}\label{KNB}
\tag{$K\!N^v$}\|f \|_{2}^2\leq
C\left(\frac{\|f\|_{1}^2}{\inf\limits_{z\in \text{supp}(f)}v_r(z)}+r^2\mathcal{E}(f)\right), \quad
\forall\, r>0, \,f\in \mathcal{F}.
\end{equation*}
Obviously, $(N^v)$ implies $(K\!N^v)$. Kigami shows that  $(DU\!E^v)$ implies $(K\!N^v)$ and that, under an exit time estimate,  $(K\!N^v)$ implies $(DU\!E^v)$. We shall see that  if one replaces the exit time assumption by a Davies-Gaffney estimate which holds  in great generality,  $(K\!N^v)$, $(N^v)$, and $(DU\!E^v)$ are in fact equivalent.  Kigami also considered a  version of $(K\!N^v)$ that is adapted to the case of so-called sub-Gaussian estimates. We will not pursue this direction in the present article (see however the remarks at the very end). The unpublished note \cite{CK}  contains further generalisations of  $(K\!N^v)$ type inequalities. 


In the case where $(M,d,\mu)$ is a metric space, we can also introduce the family of localised $v$-Nash inequalities: there exist $\alpha,C>0$  such that for every ball $B=B(x,r)$, for every $f\in \mathcal{F}_c(B)$,
\begin{equation*}\label{LNa}
\tag{$LN^v_\alpha$}
\|f \|_{2}^{2\left(1+\alpha\right)}\le \frac{C}{v_r^\alpha(x)}\|f \|_{1}^{2\alpha}\left(\|f\|^2_2+r^2\mathcal{E}(f)\right).
\end{equation*}
Here $\mathcal{F}_c(\Omega)$ is the set of functions in $ \mathcal{F}$ that are compactly supported in the open set $\Omega$. The positive parameter $\alpha$ plays a minor role here.
For instance, if $v=V$ one can check easily that  $(LN^v_\alpha)\Rightarrow (LN^v_{\alpha'})$ for all $0<\alpha'<\alpha$. We will often drop $\alpha$ in $(LN_\alpha^v)$, and then $(LN^v)$ will mean: $(LN_\alpha^v)$ for some $\alpha>0$ . 
We shall see in Proposition \ref{KL} below that $(K\!N^v)\Rightarrow(LN^v)$   if $v$ satisfies $(D_v)$ and \eqref{D2}, and
in Proposition \ref{LF}  that $(LN^v)$ is equivalent with some form of relative Faber-Krahn inequality, which coincides with $(FK)$ if $v=V$ and $(M,d,\mu)$ satisfies the doubling and reverse doubling volume conditions.
\bigskip

Introduce now, for $2<q\le+\infty$,  the  inequality:
\begin{equation*}\label{GNB}
 \tag{$GN_{q}^v$}\|fv_{r}^{\frac{1}{2}-\frac{1}{q}} \|_{q}^2\leq
C(\|f\|_{2}^2+r^2\mathcal{E}(f)), \quad \forall\, r>0,  \,f\in \mathcal{F}.
\end{equation*}
When $v(x,r)\simeq r^n$ for some $n>0$,  for instance when $M$ is endowed with a metric space, $v=V$, and $(M,d,\mu)$ has  polynomial volume growth of exponent $n$, \eqref{GNB} yields
\begin{equation*}
 \|f \|_{q}^2\leq
Cr^{-\frac{q-2}{q}n}(\|f\|_{2}^2+r^2\mathcal{E}(f)),\quad
\forall\, r>0,\, f\in \mathcal{F},
\end{equation*}
which, as one sees again by minimising in $r$, is equivalent to \eqref{GN1}  if $\frac{q-2}{q}n<2$ and to \eqref{sobo} if $\frac{q-2}{q}n=2$.
This is why we shall call \eqref{GNB} a $v$-Gagliardo-Nirenberg inequality or in short a Gagliardo-Nirenberg inequality.
Note at once that \eqref{GNB} is nothing but a resolvent estimate:
 $$ \sup_{r>0}\|{v_{r}^{\frac{1}{2}-\frac{1}{q}}}(I+r^2L)^{-1/2}\|_{2\to q}<+\infty.$$

Similarly as for $(N^v)$,  one can  formally weaken $(GN_q^v)$ in the spirit of \cite{kig}, by introducing a $v$-Kigami-Gagliardo-Nirenberg inequality 
\begin{equation*}\label{KGNB}
 \tag{$K\!G\!N_{q}^v$}\left(\inf\limits_{z\in\text{supp}(f)}v_{r}^{1-\frac{2}{q}}(z)\right)\|f \|_{q}^2\leq
C(\|f\|_{2}^2+r^2\mathcal{E}(f)), \quad \forall\, r>0,  \,f\in \mathcal{F}.
\end{equation*}

In the case where $(M,d,\mu)$ is a metric space and  if $v$ satisfies  \eqref{D2}, by restricting oneself to functions supported in  $B(x,r)$, one sees that \eqref{KGNB} implies the following version of the family of localised Sobolev inequalities \eqref{ISo}:
there exists $C>0$  such that for every ball $B=B(x,r)$, for every $f\in\mathcal{F}_c(B)$,
\begin{equation*}\label{Sq}
\tag{$LS_q^v$}
\|f \|_{q}^2\le \frac{C}{v_r^{1-\frac{2}{q}}(x)}\left(\|f\|^2_2+r^2\mathcal{E}(f)\right).
\end{equation*}

Note that in
\cite{S},  \cite{SA}, and \cite{SF} such inequalities are considered (in the case $v=V$) and called localised or scale-invariant Sobolev inequalities. For the sake of coherence with the notation $(N^v)$, $(K\!N^v)$, $(LN^v)$ on the one hand, $(GN_q^v)$, $(KGN_q^v)$ on the other hand, we could also have denoted this inequality by $(LGN_q^v)$, but we rather chose the name $(LS_q^v)$ to emphasise the connection with
Sobolev type inequalities.

It is an easy exercise to check that $(G\!N_q^v)$, $(K\!G\!N_q^v)$, $(LS_q^v)$ respectively imply
$(G\!N_{q_1}^v)$, $(K\!G\!N_{q_1}^v)$, $(LS_{q_1}^v)$ for $2<q_1<q$. We leave this to the reader.  Inequalities $(G\!N_2^v)$, $(K\!G\!N_2^v)$, $(LS_2^v)$  are all trivial.

It will be understood that if $M$ is endowed with a metric and $v=V$ we omit the superscript $v$ in all the inequalities considered above.

\bigskip
There are several reasons why we consider two families of inequalities, namely  $(N^v)$ and its variants on the one hand, and  $(GN_q^v)$ and its variants on the other hand, instead of one, even though the $(GN_q^v)$ family is easily seen to imply the $(N^v)$ one (Proposition \ref{gn} below)  and a more involved converse happens to hold under additional assumptions (Proposition \ref{ngn}). First, we want to bridge as much as possible the polynomial theory and the existing doubling theory,
and they both involve analogues of these two families. Second, and more importantly, the two families display different features.  We shall see that $(N^v)$ admits a variant which is adapted to the case where $v$ is not doubling,
whereas  $(GN_q^v)$, being in essence a resolvent estimate, is more directly related to  the matching pointwise heat kernel upper estimate $(DU\!E^v)$ to be defined below. As a matter of fact, we shall have to go through $(GN_q^v)$ in order to show, under additional assumptions, that $(N^v)$ implies $(DU\!E^v)$.
\bigskip

We are now ready to state our main result.

\begin{theorem}\label{mainDG} Let $(M,d,\mu)$ be a doubling metric measure space, $v:M\times \R_+\to \R_+$  satisfy $(A)$, $(D_v)$, and \eqref{D2},  and $L$ a  non-negative self-adjoint operator on $L^2(M,\mu)$.  Assume that 
 $(M,d,\mu,L)$ satisfies 
the Davies-Gaffney
estimate \eqref{DG2}  
and   that the semigroup $(e^{-tL})_{t>0}$ is  uniformly bounded on $L^{1}(M,\mu)$. Then the upper bound
$(DU\!E^v)$
is equivalent to
$(GN_{q}^v)$, for any $q$ such that $2<q\le +\infty$ and $\frac{q-2}{q}\kappa_v<2$, where $\kappa_v$ is as in \eqref{dv}. If in addition $(e^{-tL})_{t>0}$ is  positivity preserving, $(DU\!E^v)$ is also equivalent to  each of the following conditions: $(N^v)$, $(K\!N^v)$,  $(LN^v)$, as well as $(K\!GN_{q}^v)$, $(LS_q^v)$, under the same condition on $q$.
\end{theorem}

Note that $\frac{q-2}{q}\kappa<2$ together with  $2<q\le +\infty$ means that either $\kappa< 2$ and $q\in(2,+\infty]$,  or $\kappa\ge 2$ and $q\in (2,\frac{2\kappa}{\kappa-2})$. In the latter case, for $q>\frac{2\kappa}{\kappa-2}$, $(GN_q^v)$ may happen to be trivially false, as the polynomial case shows, even though
 $(DU\!E^v)$ is true.

We obtain also a characterization of  $(DU\!E^v)$ in terms of Faber-Krahn inequalities, but in that case there are additional subtleties, so that we refer the reader directly to Section \ref{R}.

Recall that, already 15 years ago, Carron showed in \cite{Ca} that   $(DU\!E^v)$, with $v$ not necessarily doubling,  implies a non-uniform Sobolev-Orlicz inequality involving $\mathcal{E}(f)$; interestingly enough, he also claimed that a converse should rely on the finite propagation speed of the wave equation (that is, on the Davies-Gaffney estimate). The present paper proves that he was right. It would be interesting to check directly that, under our assumptions, $(N^v)$ and $(GN_q^v)$ imply this Sobolev-Orlicz inequality, and  to investigate whether
 they are equivalent or not.

\bigskip

Here is the plan of this article:

\medskip

In Section \ref{FA}, we will use purely functional analytic techniques. We will work with a quadruple $(M,\mu,L,v)$, with $(M,\mu)$ a measure space, $L$ a non-negative self-adjoint operator on $L^2(M,\mu)$, and $v$ a function on $M\times \R_+$ satisfying $(A)$.  Sometimes, but not always, we will also assume that $v$ satisfies some doubling properties.  In Section \ref{prel}, we introduce the  weighted $L^p-L^q$ estimates that will be our main  technical tool,   in Section \ref{HN}  we prove that  $(DU\!E^v)$ implies $(N^v)$,   in Section \ref{HGN} that $(DU\!E^v)$ implies $(GN_q^v)$ for  $q>2$ small enough, in Section \ref{uni} that  $(GN_q^v)$ implies $(DU\!E^v)$ if one assumes that $v$ does not depend on $x\in M$; the fact that $(N^v)$ implies $(DU\!E^v)$ is already known in that case, see \cite{C-N}, but we elaborate on that result. 

\medskip

In Section \ref{LG0}, we introduce  a metric $d$ on $M$,  and from  Section \ref{LG} on we assume that the quadratic form $\mathcal{E}$ associated with $L$ is a strongly local and regular Dirichlet form. In Section \ref{GNN} we observe that $(GN_q^v)$ implies $(N^v)$,     that  $(K\!G\!N_q^v)$ implies $(K\!N^v)$, and that  $(LS_q^v)$ implies $(LN^v)$, in Section \ref{LG} that  $(GN_q^v)$,     $(K\!G\!N_q^v)$ and $(LS_q^v)$ (resp. $(N^v)$, $(K\!N^v)$ and $(LN^v)$) are equivalent, in
Section \ref{R} we  establish the connection with Faber-Krahn inequalities, in Section \ref{KLR} we study the effect of the so-called reverse doubling condition on Faber-Krahn  inequalities and on localised Nash inequalities.

\medskip

In Section \ref{DGE},  we assume that $(M,d,\mu)$ is a doubling metric measure space and that $L$ satisfies   the Davies-Gaffney estimate. In Section \ref{DG} we prove that  $(GN_q^v)$ implies $(DU\!E^v)$ and   in Section \ref{NG}  that $(N^v)$ implies $(GN_q^v)$ under an $L^1$ assumption on $(e^{-tL})_{t>0}$.   This section finishes the  proof of Theorem    \ref{mainDG}.

\bigskip

As we have just seen, the setting in which we work may vary from section to section. All our results are however valid in the setting of  a doubling metric measure space endowed with a strongly local regular Dirichlet form compatible with the distance (see Section \ref{LG} for details). Therefore they  not only cover the Laplace-Beltrami operators on Riemannian manifolds,  but a significantly larger class of self-adjoint differential operators acting on more general spaces. Such Dirichlet forms include restrictions of the Laplace operator to  open subsets with Dirichlet or Neumann boundary conditions, see for example  \cite{GS}. This setting also includes degenerate elliptic operators, a class which was studied for instance in \cite{ERSZ1, FKS, FL, Tru2} or subelliptic operators as in \cite{FP}. In some instances we will also consider Schr\"odinger type operators with positive or  negative potentials.

Note that Theorem \ref{mainDG} itself holds in a even more general setting, that is, in principle, beyond differential operators or even Dirichlet forms: the only assumption for the first assertion is Davies-Gaffney and uniform boundedness on $L^1$ (in particular, one could treat semigroups acting on vector bundles). In the second assertion, some positivity is required.

\bigskip


Let us finally point out that a  very preliminary version of parts of  the present work, with the same authors, appeared in \cite{Bou}.

\bigskip

{\bf Remark on notation:} In the sequel, letters $c$, $C$ and $C'$ will denote positive constants, whose
value may change at each occurrence.
\bigskip

\section{Functional analysis}\label{FA}
\subsection{Weighted $L^p-L^q$ estimates}\label{prel}

In this section,  $(M,\mu)$ will be a $\sigma$-finite measure space, $L$ a non-negative self-adjoint operator on $L^2(M,\mu)$, and  $v$ a function from $M\times \R_+$ to $\R_+$ satisfying $(A)$.

For a function $W:M\rightarrow \mathbb{R}$,  let $M_{W}$ the
operator of multiplication by $W$, that is
$$(M_{W}f)(x)=W(x)f(x).$$ In the sequel, we shall identify the function $W$ and the operator $M_W$. That is, if $T$ is a linear operator, we shall denote by $W_1TW_2$ the operator  $M_{W_1}TM_{W_2}$.
In other words $$W_1TW_2f(x):=W_1(x)T(W_2f)(x).$$

Let $1\le p\le q\le +\infty$. Let $\gamma$, $\delta$ be  real numbers such that $\gamma+\delta=\frac{1}{p}-\frac{1}{q}$.
We shall denote
 \begin{equation*}\label{vEvapq}
 \tag{$vEv_{p,q,\gamma}$}
\sup_{t>0}  \| {v_{\sqrt{t}}^\gamma}\,  e^{-tL}\, {v_{\sqrt{t}}^\delta} \|_{p \to q} < +\infty.
\end{equation*}
Of course, this condition may or may not hold, and requires in the first place that the operator ${v_{\sqrt{t}}^\gamma}\,  e^{-tL}\, {v_{\sqrt{t}}^\delta} $ is bounded from $L^p$ to $L^q$ for all $t>0$,
which is certainly not always true.
When $\gamma=\frac{1}{p}-\frac{1}{q}$ (that is $\delta=0$), we shall abbreviate $(vEv_{p,q, \gamma})$ by
\begin{equation*}\label{vEpq}
\tag{$vE_{p,q}$}\sup_{t>0}\|{v_{\sqrt{t}}^{\frac{1}{p}-\frac{1}{q}}}e^{-tL}\|_{p\rightarrow
q}<+\infty,
\end{equation*}
and when $\gamma=0$,
 by
\begin{equation*}\label{Evpq}
\tag{$Ev_{p,q}$}
\sup_{t>0}\|e^{-tL}\,{v_{\sqrt{t}}^{\frac{1}{p}-\frac{1}{q}}}\|_{p\rightarrow
q}< +\infty.
\end{equation*}
Finally, we shall abbreviate  $(vEv_{1,\infty, 1/2})$, that is
\begin{equation*}\label{ultracontractivity}
\sup_{t>0}\|{v_{\sqrt{t}}^{1/2}}\,e^{-tL}\,{v_{\sqrt{t}}^{1/2}}\|_{1\to
\infty}< +\infty,
\end{equation*}
by $(vEv)$. Another noteworthy particular case is $(vEv_{p,p,0})$, which is nothing but the uniform boundedness of $(e^{-tL})_{t>0}$ on $L^p(M,\mu)$.

In the case where we take $v=V$, we shall of course use the notation $(V\!EV_{p,q,\gamma})$, $(V\!E_{p,q})$, 
$(EV_{p,q})$, $(V\!EV)$.

Note that the above estimates are on-diagonal versions of the generalised Gaussian estimates introduced by Blunck and Kunstmann (see for instance \cite{BK}).

\bigskip

Observe that, by duality, $(vEv_{p,q,\gamma})$ is equivalent to
$(vEv_{q^{\prime},p^{\prime}, \delta})$, where $p^{\prime},q^{\prime}$ are the conjugate exponents to $p,q$ and $\gamma+\delta=\frac{1}{p}-\frac{1}{q}$. In particular, $(vE_{p,q})$ and $(Ev_{q',p'})$ are equivalent.
The following statement does not use any doubling assumption on $v$. To this purpose, we introduce slightly modified versions of  $(Ev_{1,2})$ and $(vE_{2,\infty})$, which  under $(D_v)$ are equivalent to their counterparts.

\begin{proposition}\label{implication0}  The estimates  $(DU\!E^v)$,
$(vEv)$,
\begin{equation*}\label{tilde}
\tag{$\tilde{E}v_{1,2}$}\sup_{t>0}\|e^{-(t/2)L}\,{v_{\sqrt{t}}^{1/2}}\|_{1\to 2}<+\infty
\end{equation*}
and
\begin{equation*}\label{tilde1}
\tag{$v\tilde{E}_{2,\infty}$} \sup_{t>0}\|{v_{\sqrt{t}}^{1/2}}e^{-(t/2)L}\|_{2\to
 \infty}<+\infty
\end{equation*}
 are equivalent.
\end{proposition}
\begin{proof} 
According to Dunford-Pettis theorem (for  a nice account see  \cite[Proposition 3.1]{GH}), $(vEv)$ implies that the operator ${v_{\sqrt{t}}^{1/2}}\,e^{-tL}\,{v_{\sqrt{t}}^{1/2}}$ has a bounded measurable kernel. It follows that 
$(e^{-tL})_{t>0}$ also has a measurable kernel $p_t(x,y)$ and that
\begin{equation*}
\mbox{\rm ess sup}_{x,y\in
M}{v_{\sqrt{t}}^{1/2}(x)|p_{t}(x,y)|v_{\sqrt{t}}^{1/2}(y)}=\|{v_{\sqrt{t}}^{1/2}}\,e^{-tL}\,{v_{\sqrt{t}}^{1/2}}\|_{1\to \infty}<+\infty,
\end{equation*}
hence 
$(DU\!E^v)$ holds. The converse from $(DU\!E^v)$ to $(vEv)$ also follows   from the above equality.

Similarly to what we already observed, \eqref{tilde} and \eqref{tilde1} are equivalent by duality. Furthermore, it is well-known that, for an operator $T$ mapping $L^1$ to $L^2$,  $$\|T^{*}T\|_{1\to \infty}=\|T^{*}\|_{2\to
\infty}^{2}=\|T\|_{1\to 2}^{2}.$$ By taking
	$T=e^{-(t/2)L}{v_{\sqrt{t}}^{1/2}}$, we obtain
\begin{equation}\label{link}
 \|{v_{\sqrt{t}}^{1/2}}\,e^{-tL}\,{v_{\sqrt{t}}^{1/2}}\|_{1\to \infty}=\|{v_{\sqrt{t}}^{1/2}}e^{-(t/2)L}\|_{2\to
 \infty}^{2}=\|e^{-(t/2)L}\,{v_{\sqrt{t}}^{1/2}}\|_{1\to 2}^{2},
\end{equation}
which shows the equivalence of $(DU\!E^v)$ with the two other conditions.
\end{proof}

The consequence   if one does assume doubling is now obvious.
\begin{coro}\label{implication} Assume that $v$ satisfies $(D_v)$. The estimates  $(DU\!E^v)$,
$(vEv)$,
$(Ev_{1,2})$,
and $(vE_{2,\infty})$ are equivalent.
\end{coro}

\begin{rem} The above so-called $T^*T$-argument yields in the same way
$$(Ev_{p,2})\Leftrightarrow(vE_{2,p'})\Leftrightarrow(vEv_{p,p',\gamma})$$
for all $1\le p\le 2$, $\gamma=\frac{1}{p}-\frac{1}{2}$.
\end{rem}

\begin{rem} The equivalence between $(vEv)$ and
$(Ev_{1,2})$ means that  an equivalent definition for  $(DU\!E^v)$ is the following:
\begin{equation}
\|p_{t}(x,.)\|_2^2\leq
\frac{C'}{v(x,\sqrt{t})},\mbox{ for all }t>0, \mbox{ for }~\mu-\mbox{a.e. }\,x \in
M\label{duesgg}
\end{equation} 
(this also holds with a slight modification if $v$ is not doubling).  Also, it is worth emphasising the difference between $(Ev_{1,2})$ and $(vE_{1,2})$:
$(Ev_{1,2})$ (or $(vE_{2,\infty})$)  is equivalent to
$$\mbox{\rm ess sup}_{x\in
M,\,t>0}\,v_{\sqrt{t}}(x)\int_M p_t^2(x,y)\,d\mu(y)<+\infty.
$$ 
whereas  $(vE_{1,2})$ (or $(Ev_{2,\infty})$) is equivalent to
$$\mbox{\rm ess sup}_{x\in
M,\,t>0}\int_M p_t^2(x,y)v_{\sqrt{t}}(y)\,d\mu(y)<+\infty.
$$ 
The cornerstone of  our main results, Proposition  $\ref{cieplo1}$ below, yields in particular that, under the so-called Davies-Gaffney estimate and additional assumptions on $(M,\mu)$ and $v$, $(Ev_{1,2})$ and $(vE_{1,2})$ are actually equivalent.

\end{rem}


By applying complex interpolation to the family of operators $$T_z:=v_{\sqrt{t}}^{\gamma_1z+(1-z)\gamma_2}\,  e^{-tL}\, v_{\sqrt{t}}^{\delta_1z+(1-z)\delta_2},\ 0\le\mbox{Re}\,z\le 1,$$
one obtains easily:
\begin{proposition}[Interpolation]\label{interpolation} Let $1\leq p_1\le q_1\leq +\infty$,   $1\leq p_2\le q_2\leq +\infty$, $\gamma_1,\gamma_2\in\R$. Then  $(vEv_{p_1,q_1,\gamma_1})$ and
 $(vEv_{p_2,q_2,\gamma_2})$ imply $(vEv_{p,q,\gamma})$, where $\frac{1}{p}=\frac{\theta}{p_1}+\frac{1-\theta}{p_2}$, $\frac{1}{q}=\frac{\theta}{q_1}+\frac{1-\theta}{q_2}$, $\gamma=\theta\gamma_1+(1-\theta)\gamma_2$.
\end{proposition}

Of particular interest will be the case $p_2=q_2$, $\gamma_2=0$, which amounts to the uniform boundedness of $(e^{-tL})_{t>0}$ on $L^{p_2}(M,\mu)$.
Note that the main technical point of the present paper will be an extrapolation counterpart  to this consequence of Proposition \ref{interpolation}, namely   Proposition \ref{extrapolation} below. 

\bigskip

Next
    for any pair $(p,q)$ such that $1 \le p \le q \le \infty$ we define exponents $$0\le \gamma_-(p,q)\le \gamma_+(p,q)$$ by the formulae 
$$
\gamma_- \left( p,q\right)  =\max\left\{ 
\frac{1}{2p}-\frac{1}{q},0\right\} 
$$
and
$$
\gamma_+\left(p,q\right)  =\min\left\{\frac{1}{p}-\frac{1}{q},
\frac{1}{2}-\frac{1}{2q}\right\}.
$$
Using this notation, we can  state a  consequence of Corollary  \ref{implication} and Proposition \ref{interpolation} that only relies on duality and interpolation.

\begin{coro} \label{coro} Assume that $v$ satisfies $(D_v)$. The pointwise heat kernel upper bound $(DU\!E^v)$ implies $(Ev_{p,2})$ for all $p$ such that $1\leq p\leq 2$,  $(vE_{2,q})$
for all $q$ such that $2\leq q\leq \infty$, and $(vEv_{p,q,\gamma})$ for all $p,q$ such that $1\leq p\leq 2\le q\leq +\infty$ and    $\gamma=\frac{1}{2}-\frac{1}{q}$. 
If in addition $(e^{-tL})_{t>0}$ is uniformly bounded on $L^1(M,\mu)$,
one obtains also  $(vEv_{p,q,\gamma})$ for all $p,q$ such that $1\leq p\leq  q\leq +\infty$ and  all $  \gamma_-(p,q) \le  \gamma    \le \gamma_+(p,q)$.
\end{coro}
\begin{proof} Corollary  \ref{implication} says that, under $(D_v)$, $(DU\!E^v)$ implies $(vEv_{1,\infty,1/2})$, $(vEv_{1,2,0})$ and $(vEv_{2,\infty,1/2})$.
Since the semigroup $(e^{-tL})_{t>0}$ is uniformly bounded on $L^2(M,\mu)$, we have in addition $(vEv_{2,2,0})$.  Proposition \ref{interpolation} applied to $(vEv_{1,2,0})$ and   $(vEv_{2,2,0})$ yields $(vEv_{p,2,0})$, that is $(Ev_{p,2})$,  for all $1\le p\le2$, and
applied to $(vEv_{1,\infty,1/2})$  and  $(vEv_{2,\infty,1/2})$ it yields $(vEv_{p,\infty,1/2})$ for all $1\le p\le 2$. Now interpolating again between $(vEv_{p,2,0})$ and $(vEv_{p,\infty,1/2})$ yields the first part of the statement.
Next if  $(e^{-tL})_{t>0}$ is uniformly bounded on $L^1(M,\mu)$ then by duality and 
 interpolation $(vEv_{p,p,0})$ holds for all $ 1 \le p \le \infty$. One checks easily that interpolation between $(vEv_{1,1,0})$,
$(vEv_{\infty,\infty,0})$, and 
$(vEv_{1,2,0})$ yields $(vEv_{p,q,0})$ for all $1 \le p \le q \le \infty$ such that $1/p\leq 2/q$, that is $(vEv_{p,q,\gamma_-(p,q)})$ for this range of $p,q$. 
To obtain $(vEv_{p,q,\gamma_-(p,q)})$  for  $1/p>  2/q$,  one interpolates between $(vEv_{\infty,\infty,0})$,
$(vEv_{1,2,0})$, and $(vEv_{1,\infty,}\frac{1}{2})$.
One then obtains  $(vEv_{p,q,\gamma_+(p,q)})$ by duality and interpolating between $(vEv_{p,q,\gamma_-(p,q)})$
and $(vEv_{p,q,\gamma_+(p,q)})$ yields $(vEv_{p,q,\gamma})$ for all $  \gamma_-(p,q) \le  \gamma    \le \gamma_+(p,q)$.
\end{proof}

Note that one can drop assumption $(D_v)$ in Corollary \ref{coro}, at the expense of modifying  conditions $(vEv_{p,q,\gamma})$ in the spirit of Proposition \ref{implication0}. If one is prepared to assume the full Gaussian upper estimate $(U\!E^v)$ instead of $(DU\!E^v)$,
 one can obtain 
the same conclusion as in Corollary \ref{coro} in a more straightforward way and without any restriction on the exponent $\gamma$. 


 \begin{coro}\label{lem} Let  $(M,d,\mu)$  be a metric measure space, let $v:M\times \R_+\to \R_+$ satisfy $(A)$, $(D_v)$, and \eqref{D2}, and let $L$ be a  non-negative self-adjoint operator on $L^2(M,\mu)$.
 Assume that $(e^{-tL})_{t>0}$ is uniformly bounded on $L^{p_0}(M,\mu)$ for some $p_0\in [1,2)$.
 The full Gaussian upper bound
$(U\!E^v)$
 implies  $(vEv_{p,q,\gamma})$, for all $p,q$ such that $p_0\leq p< q\leq p'_0$ and   all $ \gamma\in \R$.
 \end{coro}
  \begin{proof} It is easy to see using $(D_v)$ and \eqref{D2}  that $(U\!E^v)$ implies
  $(vEv_{1,\infty, \gamma})$ for all $ \gamma\in \R$. On the other hand, $(e^{-tL})_{t>0}$ is uniformly bounded on $L^r(M,\mu)$ for all $p_0\le r\le p'_0$ by duality and interpolation.
Applying Proposition \ref{interpolation} with $p_1=1$, $q_1=\infty$,  every $\gamma_1\in\R$, $p_2=q_2=r$, for all $r$ such that $p\le r\le q$,  and $\gamma_2=0$    yields the claim.
\end{proof}

Note  that if $(M,d,\mu)$ is a doubling metric measure space and if $v\ge \varepsilon V$ for some $\varepsilon>0$, one need not assume the uniform boundedness of $(e^{-tL})_{t>0}$ on $L^{p_0}(M,\mu)$ in Corollary \ref{lem}, since it then   follows for $p_0=1$ from $(U\!E^v)$. 

Now recall that  $(U\!E^v)$ and $(DU\!E^v)$ coincide if the Davies-Gaffney estimate holds (see section \cite[Section 4.2]{CS}). We can therefore state the following.

\begin{theorem}\label{lemdg} Let  $(M,d,\mu)$  be a metric measure space,  $v:M\times \R_+\to \R_+$ satisfy $(A)$, $(D_v)$, and \eqref{D2}, and  $L$ be a  non-negative self-adjoint operator on $L^2(M,\mu)$.
 Assume that 
 $(M,d,\mu,L)$ satisfies 
 \eqref{DG2} and that $(e^{-tL})_{t>0}$ is uniformly bounded on $L^{p_0}(M,\mu)$ for some $p_0\in [1,2)$.
 Then the pointwise heat kernel upper bound
$(DU\!E^v)$
 implies  $(vEv_{p,q,\gamma})$, for all $p,q$ such that $p_0\leq p< q\leq p'_0$ and   all $ \gamma\in \R$.
 \end{theorem}

Theorem \ref{lemdg} indicates that a sensible step towards    $(DU\!E^v)$, at least if  the Davies-Gaffney estimate holds,  is to show that all estimates 
$(vEv_{p,q,\gamma})$ are equivalent. Indeed, this will be proved for $q=2$   in Proposition \ref{cieplo1} below.


\bigskip

\subsection{The heat kernel upper bound implies Nash}\label{HN}

Our main result here  is the following.

 \begin{proposition}\label{Nash22} 
Let $(M,\mu)$ be a measure space,  $L$ a  non-negative self-adjoint operator on $L^2(M,\mu)$, and  $v$ a function from $M\times \R_+$ to $\R_+$ satisfying $(A)$.  The heat kernel upper bound $(DU\!E^v)$
implies the inequality  $(N^v)$.
\end{proposition}

In the case where $v$ satisfies $(D_v)$, this result will also follow from  Propositions \ref{dg} and \ref{gn} below, but even in that  case, it is nice   to have the following simple and direct proof. According to Corollary  \ref{implication}, it is enough to prove that  \eqref{tilde} implies $(N^v)$. For the sake of simplicity and for future record, we prefer to state the implication from \eqref{EV12} to $(N^v)$, but the proof is similar.

\begin{proposition}\label{Nash1}  The estimate:
\begin{equation*}\label{EV12}
\tag{$Ev_{1,2}$} \sup_{t>0}\|e^{-tL}\,{v_{\sqrt{t}}^{1/2}}\|_{1\rightarrow
2}<+\infty
\end{equation*}
implies the   inequality $(N^v)$.
\end{proposition}
\begin{proof}
One can write,  for $f\in\mathcal{D}$ and $t>0$,
\begin{equation}
f=e^{-tL}f+\int_{0}^{t}L e^{-sL}f\,ds.\label{int}
\end{equation}
This formula is also valid in $L^2(M,\mu)$ because $(e^{-sL})_{s>0}$ is analytic on $L^2(M,\mu)$.
Thus, for $f\in\mathcal{F}=\mathcal{D}_2(L^{1/2})$,
\begin{eqnarray*}
   \|f\|_{2} &\leq & \|e^{-tL}f\|_{2}+\int_{0}^{t}\|L
e^{-sL}f\|_{2}\,ds\\
  &\leq & \|e^{-tL}v_{\sqrt{t}}^{1/2}v_{\sqrt{t}}^{-1/2}f\|_{2}+\int_{0}^{t}\|L^{1/2}
e^{-sL}L^{1/2}f\|_{2}\,ds\\
  &\leq & \|e^{-tL}{v_{\sqrt{t}}^{1/2}}\|_{1\to
  2}\|fv_{\sqrt{t}}^{-1/2}\|_{1}+C\int_{0}^{t}s^{-1/2}\|L^{1/2}f\|_{2}\,ds.
\end{eqnarray*}
In the last inequality, we have again used the analyticity   of $(e^{-sL})_{s>0}$ on $L^2(M,\mu)$, which yields
$$\|L^{1/2}e^{-sL}f\|_{2}\le C s^{-1/2}\|f\|_{2},\ \forall \,f\in L^2(M,\mu), \,s>0.$$
Hence
$$ \|f\|_{2} \le \|e^{-tL}{v_{\sqrt{t}}^{1/2}}\|_{1\to
  2}\|fv_{\sqrt{t}}^{-1/2}\|_{1}+C\sqrt{2t}\|L^{1/2}f\|_{2}, \quad
\forall\,f\in\mathcal{F},\, t>0,$$
therefore   using \eqref{EV12},
\begin{equation*}
 \|f \|_{2}\leq
C'(\|fv_{\sqrt{t}}^{-1/2}\|_{1}+\sqrt{t}\|L^{1/2}f\|_{2}), \quad
\forall\,f\in\mathcal{F},\, t>0,
\end{equation*}
that is, setting $r=\sqrt{t}$,  $(N^v)$. 
\end{proof}

Note that one could also adapt the proof of Kigami in \cite[pp.528-529]{kig} to get a proof of Proposition $\ref{Nash22}$, at least if one assumes a priori the existence of measurable $p_t$. We leave this to the interested reader. 
Our more functional analytic approach enables one to treat other $L^p$ spaces than $L^2$.
\begin{rem} \label{generalised Nash} One can prove in a similar way that if
\eqref{Evpq}
holds for some $p,q$  such that $1\le p< q<+\infty$ and, if  $(e^{-tL})_{t>0}$ is  bounded analytic on $L^q(M,\mu)$,
\begin{equation*}
\tag{$N_{p,q}^v$}\label{Npq} \|f \|_{q}\leq
C(\|fv_{r}^{\frac{1}{q}-\frac{1}{p}}\|_{p}+r\|L^{1/2}f\|_{q}),
\quad \forall\, r>0,\,f\in\mathcal{D}_q(L^{1/2}),
\end{equation*}
follows. More generally, for $\beta>0$,
\begin{equation*}
\tag{$N_{p,q,\beta}^v$}\label{Npqb}  \|f \|_{q}\leq
C_\beta(\|fv_{r}^{\frac{1}{q}-\frac{1}{p}}\|_{p}+r^\beta\|L^{\beta/2}f\|_{q}),
\quad \forall\, r>0,\,f\in\mathcal{D}_q(L^{\beta/2}).
\end{equation*}
For $\beta\ge 2$, one uses a higher order Taylor formula instead of \eqref{int}.
\end{rem}
Theorem $\ref{lemdg}$ and Remark \ref{generalised Nash}  yield the following.

 \begin{theorem}\label{Nash23} 
Let $(M,d,\mu)$ be a metric measure space,  $L$ a  non-negative self-adjoint operator on $L^2(M,\mu)$, and  $v$ a function from $M\times \R_+$ to $\R_+$ satisfying $(A)$, $(D_v)$ and \eqref{D2}. Assume that 
 $(M,d,\mu,L)$ satisfies 
 \eqref{DG2} and that  $(e^{-tL})_{t>0}$  is uniformly bounded  on $L^{{p_0}}(M,\mu)$ for some $p_0\in[1,2)$.  Then the heat kernel upper bound $(DU\!E^v)$
implies   \eqref{Npqb} for all $p_0\le p< q< p'_0$ and $\beta>0$.
\end{theorem}

In the proof of the above statement, we make use of the fact that, since $(e^{-tL})_{t>0}$ is  bounded analytic on $L^2(M,\mu)$, if in addition  it is uniformly bounded  on $L^{p}(M,\mu)$, for some $p\in [1,+\infty]$, then by interpolation it is   bounded analytic on $L^q(M,\mu)$, for all $q$ strictly between $2$ and $p$.

Although it will come under some additional assumptions as a by-product  of our results in Section \ref{DGE}, we do not know how to prove in a direct way, that is,  without going through $(GN_q^v)$,
that conversely $(N^v)$ implies \eqref{EV12} under these assumptions.   More generally, one may wonder whether  \eqref{Npq} and \eqref{Evpq}   coincide or not
(we will be able to answer positively a similar question for $(GN_q^v)$ and its variants, see Proposition \ref{resolvent2}). The following observation together with Proposition \ref{interpolation}  shows that, if it were to be the case, all $(N_{p,q}^v)$ (and  \eqref{Evpq})
would  be equivalent for fixed $q$ and all $p\in [1,q)$ as long as $(e^{-tL})_{t>0}$ is uniformly bounded on $L^1(M,\mu)$. This is unlikely to hold without further assumptions.

\begin{proposition} \label{pl} $(N_{p,q}^v)$ implies $(N_{{p_0},q}^v)$ for  $1\le {p_0}<p<q<+\infty$. In particular,
$(N_{p,2}^v)$ for  $1<p<2$ implies $(N^v)$.
\end{proposition}

\begin{proof}
Assume $(N_{p,q}^v)$ and write
$$\|fv_{r}^{\frac{1}{q}-\frac{1}{p}}\|_{p}\le \|f\|_{q}^{\theta}\|fv_{r}^{\frac{1}{q}-\frac{1}{{p_0}}}\|_{{p_0}}^{1-\theta},$$ where $\theta$ is such that $\frac{1}{p}=\frac{\theta}{q}+\frac{1-\theta}{{p_0}}$.
It follows that
\begin{eqnarray*}
 \|f \|_{q}&\leq&
C(\|f\|_{q}^{\theta}\|fv_{r}^{\frac{1}{q}-\frac{1}{{p_0}}}\|_{{p_0}}^{1-\theta}+r\|L^{1/2}f\|_{q})\\
&\le& C(\epsilon\|f\|_{q}+\epsilon^{-\frac{\theta}{1-\theta}}\|fv_{r}^{\frac{1}{q}-\frac{1}{{p_0}}}\|_{{p_0}}+r\|L^{1/2}f\|_{q}),
\end{eqnarray*}
for all $r,\epsilon>0$, $f\in\mathcal{D}_q(L^{1/2})$.
One obtains  $(N_{{p_0},q}^v)$  by choosing $\epsilon=\frac{1}{2C}$.
\end{proof}

It is known, already in the case where $v$ does not depend on $x$ but decays more quickly than a negative power (in particular $v$ is not doubling), that Proposition \ref{Nash1} is not optimal: in that case,  under mild conditions on $v$, $(DU\!E^v)$ is equivalent to a so-called generalised Nash inequality, which is strictly stronger than $(N^v)$ (see \cite[Theorem II.5]{C-N}). Using the technique introduced in \cite{C-N}, one can indeed obtain in a simple way a stronger version of Proposition \ref{Nash1}.
\begin{proposition}\label{Nash2}  The estimate
\eqref{tilde}
implies the   inequality 
\begin{equation*}\label{NnD}
\tag{$\widetilde{N}^v$}\|f \|_{2}^2\log\frac{c\|f \|_{2}^2}{\|fv_{r}^{-1/2}\|_{1}^2}\leq
r^2\mathcal{E}(f), \quad
\forall\, r>0, \,f\in \mathcal{F},
\end{equation*}
where $c$ is the inverse of the supremum in  \eqref{tilde} squared.
\end{proposition}
\begin{proof}
Start with the inequality from \cite[Proposition II.2]{C-N}, 
$$\exp\left(-\frac{\mathcal{E}(f)}{ \|f\|^2_2}t\right)\leq\frac{\|e^{-(t/2)L}f\|_2^2}{ \|f\|^2_2},$$
which follows from Jensen's inequality applied to the spectral resolution of $L$.

Then \eqref{tilde} with constant $C$ yields
\begin{equation*}
\exp\left(-\frac{\mathcal{E}(f)}{ \|f\|^2_2}t\right)\leq \frac{\|e^{-(t/2)L}\,{v_{\sqrt{t}}^{1/2}}\|_{1\to
  2}^2\|fv_{\sqrt{t}}^{-1/2}\|_{1}^2}{ \|f\|^2_2}\leq  \frac{C^2\|fv_{\sqrt{t}}^{-1/2}\|_{1}^2}{\|f\|^2_2},\end{equation*}
which is obviously equivalent to \eqref{NnD} by changing $t$ to $r^2$ and taking logarithms.
\end{proof}

Propositions \ref{Nash2} and  \ref{implication0}  yield the following corollary, where one still does not assume $v$ to be doubling.  Note that \cite[Theorem 3.9]{BBGM} 
corresponds to the particular case where $v(x,r)$ is a product of a function of $x$ and a function of $r$.

\begin{coro}\label{Nash3}  The heat kernel upper bound
$(DU\!E^v)$
implies the   inequality 
\begin{equation*}
\tag{$\widetilde{N}^v$}\label{bacra}\|f \|_{2}^2\log\frac{c\|f \|_{2}^2}{\|fv_{r}^{-1/2}\|_{1}^2}\leq
r^2\mathcal{E}(f), \quad
\forall\, r>0, \,f\in \mathcal{F},
\end{equation*}
for some $c>0$.
\end{coro}

Rewriting \eqref{NB}  as
$$0<\frac{1}{C}\leq
\frac{\|fv_{r}^{-1/2}\|_{1}^2}{\|f \|_{2}^2}+\frac{r^2\mathcal{E}(f)}{\|f \|_{2}^2}$$
for all $f\in \mathcal{F}\setminus \{0\}$ and all $r>0$
and using the elementary fact that
$$A,B,c>0, \,\log\frac{c}{A}\le B\Rightarrow A+B\ge \min\left(\frac{c}{2},\log 2\right),$$
one sees that \eqref{NnD}  implies  \eqref{NB} with $C=\frac{1}{\min\left(\frac{c}{2},\log 2\right)}$, and as we  already said the converse is false even in the case where $v$ does not depend on  $x$ (see \cite{C-N}).

A posteriori, if $v$ is doubling and if $(M,d,\mu,L)$ satisfies the additional assumptions of Proposition \ref{Nana} below, one can see that \eqref{NB} does imply \eqref{NnD}; indeed,  Proposition \ref{Nana} states that in that situation  \eqref{NB}  implies $(DU\!E^v)$, which implies  back \eqref{NnD}
by Corollary \ref{Nash3}. One can see this directly.

\begin{proposition} If $v$   satisfies $(D_v)$, \eqref{NB} and   \eqref{NnD} are equivalent.
\end{proposition}

\begin{proof} We have already seen  that \eqref{NnD} always  implies  \eqref{NB}. Now for the converse.
For $f\in \mathcal{F}\setminus \{0\}$ and $r>0$, denote
$$A(r,f):=\frac{\|fv_{r}^{-1/2}\|_{1}^2}{\|f \|_{2}^2}$$
and 
$$B(f):=\frac{\mathcal{E}(f)}{\|f \|_{2}^2}.$$
Since  $v$ is assumed to be non-decreasing in $r$, the function $r\to A(r,f)$ is non-increasing, and since $v$ satisfies $(D_v)$,
\begin{equation}
\label{blabla}
\frac{A(s,f)}{A(r,f)}\le C\left(\frac{r}{s}\right)^{\kappa_v}, \mbox{ for } r\ge s>0,
\end{equation}
where  $C>0$ being the doubling constant is  independent of $f$.

The validity of  \eqref{NB}
means that
\begin{equation}
\label{vali}\inf_{r>0,f\in \mathcal{F}\setminus \{0\}}A(r,f)+r^2B(f)= c>0.
\end{equation}

Assume first that $B(f) \neq 0$. 

Define $$r_0=r_0(f):=\inf\{r>0; r^2B(f)\ge A(r,f)\}.$$
Note that $r_0=0$ would not be compatible with \eqref{vali}, hence $r_0>0$.
Then $(r_0/2)^2B(f)< A(r_0/2,f)$ since $r_0/2<r_0$.
Also, there exists $\varepsilon>0$ arbitrarily small
such that 
$(r_0+\varepsilon)^2B(f)\ge A(r_0+\varepsilon,f)$, hence, for 
$\varepsilon\le r_0$,
$(2r_0)^2B(f)\ge A(2r_0,f)$ .
But by doubling $A(2r_0,f)\ge C^{-1}A(r_0,f)$ and $A(r_0/2,f)\le CA(r_0,f)$,
hence $$\frac{1}{4C}A(r_0,f)\le r_0^2B(f)\ge 4CA(r_0,f).$$
It follows that 
$$\min\left\{A(r_0,f),r_0^2B(f)\right\}\ge \frac{c}{4C+1}.$$
If $r\le r_0$ then 
$$A(r,f)e^{r^2B(f)}\ge A(r,f)\ge A(r_0,f)\ge \frac{c}{4C+1}:=c'>0.$$
Now for $r\ge r_0$.
Using \eqref{blabla},
\begin{eqnarray*}
A(r,f)e^{r^2B(f)}&\ge& C^{-1}A(r_0,f)\left(\frac{r_0}{r}\right)^{\kappa_v}e^{r^2B(f)}\\
&=&C^{-1} A(r_0,f)\left(\frac{r_0}{r}\right)^{\kappa_v}e^{(r/r_0)^2r_0^2B(f)}\\
&\ge&C^{-1}c'\left(\frac{r_0}{r}\right)^{\kappa_v}e^{c'(r/r_0)^2}
\end{eqnarray*}
Set $ b=\inf_{x\ge 1} x^{-{\kappa_v}}e^{c'x^2}$. Note that $b>0$   only depends   on ${\kappa_v},c'$.  Finally
$$A(r,f)e^{r^2B(f)}\ge C^{-1}c'b>0,$$
which is nothing but \eqref{NnD}.

The argument for the case $B(f)=0$ is straightforward so we skip it.

\end{proof}

\begin{rem}\label{cau} Similarly as Carron in \cite{Ca}, one can observe that the best upper bound for $p_t$ is... $p_t$ itself,
and  obtain a universal Nash inequality $(\widetilde{N})$ by taking $v(x,r):=\frac{1}{p_{r^2}(x,x)}$, or more generally $v(x,r):=\frac{1}{\|p_{r^2}(x,.)\|_2^2}$ if $p_t$ is not assumed or known to be continuous.
\end{rem}

\begin{rem} One may conjecture that  \cite[Theorem II.5]{C-N}, see also Section $\ref{uni}$ below, generalises to the case where $v$ does depend on $x$, that is \eqref{bacra} implies back $(DU\!E^v)$. The difficulty is related to the fact that
we do not know so far how to prove directly, even when $v$ is doubling, that $(N^v)$ implies $(DU\!E^v)$. We have to go through $(GN_q^v)$, hence the next section. The article \cite{BBGM} does contain a partial converse to Corollary $\ref{Nash3}$, in the case where 
$v(x,r)$ is a product of a function of $x$ and a function of $r$, and the function of $x$ satisfies a Lyapunov type condition.
\end{rem}


\bigskip

\subsection{The heat kernel upper bound  implies Gagliardo-Nirenberg}\label{HGN}

In this section,   we will prove:

\begin{proposition}\label{dg} Let  $(M,\mu)$  be a measure space, $L$ a  non-negative self-adjoint operator on $L^2(M,\mu)$ and let $v:M\times \R_+\to \R_+$ satisfy $(A)$ and  $(D_v)$.  Then the heat kernel upper bound
$(DU\!E^v)$ implies the inequality $(GN_q^v)$
for all $q$ such that $2<q\le +\infty$ and
$\frac{q-2}{q}\kappa_v<2$, where $\kappa_v$ is as in \eqref{dv}.
\end{proposition}

The above statement does not cover the limit case $\frac{q-2}{q}\kappa_v=2$; we suspect the latter might be obtained by using the self-improvement of $(DU\!E^v)$ into $(U\!E^v)$. 

Recall that the constraint on $q$ can be reformulated in the following way: either $\kappa_v< 2$ and
$q\in(2,+\infty]$,  or $\kappa_v\ge 2$ and $q\in (2,\frac{2\kappa_v}{\kappa_v-2})$. For specific considerations on the  case $q=+\infty$, see Corollary \ref{55} below.

A remark similar to Remark \ref{cau} is in order, except that the universal  inequality $(GN_q)$ one obtains in this way only holds under the condition
that $$p_t(x,x)\le Cp_{2t}(x,x),\,\forall\,t>0,\,x\in M$$ (similar formulation if needed with $\|p_{t}(x,.)\|_2^2$).

According to Corollary \ref{coro}, it is enough to prove that  $(vE_{2,q})$ implies $(GN_q^v)$.
In fact, these two conditions happen to be equivalent. For the converse, we shall use the fact that
$(GN_q^v)$ can be reformulated as a resolvent estimate:
$$ \sup_{t>0}\|{v_{\sqrt{t}}^{\frac{1}{2}-\frac{1}{q}}}(I+tL)^{-1/2}\|_{2\to q}<+\infty.$$
    We shall develop further this point of view in Proposition \ref{resolvent2} below, and it will also be instrumental in Section \ref{DG}. Let us start by adopting a point of view more similar to the one in Proposition \ref{Nash1}.

\begin{proposition} \label{equivalence} Assume $v$ satisfies $(D_v)$. Let $q$ be such that $2<q\le +\infty$ and $\frac{q-2}{q}\kappa_v<2$, where $\kappa_v$ is as in \eqref{dv}.
Then the estimate:
\begin{equation*}\label{VE2q}
\tag{$vE_{2,q}$}\sup_{t>0}\|{v_{\sqrt{t}}^{\frac{1}{2}-\frac{1}{q}}}e^{-tL}\|_{2\rightarrow
q}<+\infty
\end{equation*}
implies
$(GN_{q}^v)$. Conversely, $(GN_{q}^v)$ implies $(vE_{2,q})$ for all $q$ such that $2<q\le +\infty$. The latter  assertion does not require $v$ to be doubling.
\end{proposition}
\begin{proof} Assume $(vE_{2,q})$  and  $\frac{q-2}{q}\kappa_v<2$. Set $\alpha=\frac{1}{2}-\frac{1}{q}$. Again, for  $f\in L^2(M,\mu)$,  write
\begin{equation*}f=e^{-tL}f+\int_{0}^{t}L e^{-sL}f\,ds,\quad \forall\, t>0,\end{equation*}
hence
$${v_{\sqrt{t}}^{\alpha}}f={v_{\sqrt{t}}^{\alpha}}e^{-tL}f+\int_{0}^{t}{v_{\sqrt{t}}^{\alpha}}e^{-(s/2)L}L e^{-(s/2)L}f\,ds.$$
Then
\begin{eqnarray*}
   &&\|{v_{\sqrt{t}}^{\alpha}}f\|_{q} \leq 
  \|{v_{\sqrt{t}}^{\alpha}}e^{-tL}f\|_{q}+\int_{0}^{t}\|{v_{\sqrt{t}}^{\alpha}}e^{-(s/2)L}\|_{2\to q}
   \|L e^{-(s/2)L }f\|_{2}\,ds\\
  &\leq & C\|{v_{\sqrt{t}}^{\alpha}}e^{-tL}\|_{2\to q}\|f\|_{2}+\int_{0}^{t} \|v_{\sqrt{t}}/v_{\sqrt{s/2}}\|_{\infty}^{\alpha}\|{v_{\sqrt{s/2}}^{\alpha}}e^{-(s/2)L}\|_{2\to q}\|L
   e^{-(s/2)L}f\|_{2}\,ds.
   \end{eqnarray*}
Using   $(vE_{2,q})$ and (\ref{dv}), we obtain, for $f\in\mathcal{F}$,
   \begin{eqnarray*}
   \|{v_{\sqrt{t}}^{\alpha}}f\|_{q} &\leq & C\|f\|_{2}+C'\int_{0}^{t} \left(\frac{t}{s}\right)^{
  \alpha\kappa_v/2}\|L^{1/2}e^{-(s/2)L} (L^{1/2}f)\|_{2}\,ds\\
  &\leq & C\|f\|_{2}+C't^{ \alpha\kappa_v/2}\left(\int_{0}^{t}s^{-\frac{\alpha\kappa_v}{2}-\frac{1}{2}}\,ds\right)\|L^{1/2}f\|_{2}\quad\\
&\leq & C(\|f\|_{2}+\sqrt{t}\|L^{1/2}f\|_{2}),
\end{eqnarray*}
that is, setting $r=\sqrt{t}$, $(GN_{q}^v)$.
In the second inequality, we have used the analyticity of $(e^{-tL})_{t>0}$ on $L^2(M,\mu)$, and in the last one the fact that $\alpha\kappa_v<1$.

\bigskip

Now for the converse. Assume that
\begin{equation*}
\tag{$GN_{q}^v$} \|fv_{\sqrt{t}}^{\frac{1}{2}-\frac{1}{q}} \|_{q}\leq
C(\|f\|_{2}+\sqrt{t}\|L^{1/2}f\|_{2}), \quad \forall\, t>0,  \ \forall\,f\in\mathcal{D}.
\end{equation*}
This can be rewritten as
\begin{eqnarray*}
\|{v_{\sqrt{t}}^{\frac{1}{2}-\frac{1}{q}}}f \|_{q}^{2}&\leq&
C'(\|f\|_{2}^{2}+t<L f,f>)\\
&=& C'<(I+tL)f,f>\\
&=& C'\|(I+tL)^{1/2}f\|_2^2,
\end{eqnarray*}
thus, replacing $f$ by $e^{-tL}f$,
\begin{eqnarray*}
\|{v_{\sqrt{t}}^{\frac{1}{2}-\frac{1}{q}}}e^{-tL}f\|_{q}&\leq& C'\|(I+tL)^{1/2}e^{-tL}f\|_{2}\\
&\leq&  C'\|(I+tL)^{1/2}e^{-tL}\|_{2\to 2}\|f\|_{2}\\
&=&  C'\left(\sup_{\lambda>0}\,(1+t\lambda)^{1/2}\,e^{-t\lambda}\right)\|f\|_{2}\\
&=& C'' \|f\|_{2}.
\end{eqnarray*}
\end{proof}

As we already said, Proposition  \ref{dg} yields Proposition \ref{Nash22} as a by-product in the case where $v$ satisfies $(D_v)$, because, according to Proposition \ref{gn} below,  its conclusion is stronger. On the other hand, the converse part of Proposition
\ref{equivalence} can be used to give a (rather indirect) proof
of  the  implication from $(GN_{q}^v)$ to $(N^v)$ which will see in a straighforward way  in Proposition \ref{gn} below.

\begin{proposition} \label{nq}
For any $q>2$, $(GN_{q}^v)$ implies $(N_{p,2}^v)$ for all $p$, $1\le p\le q'$, and in particular $(N^v)$.
\end{proposition}
\begin{proof}
By Proposition \ref{equivalence}, $(GN_q^v)$ implies
(\ref{VE2q}). By duality, (\ref{VE2q}) is
equivalent to $(Ev_{q',2})$. On the other hand,  as noticed in Remark
\ref{generalised Nash},  $(Ev_{q',2})$ implies $(N_{q',2}^v)$. Now, according to Proposition  \ref{pl},  $(N_{q',2}^v)$ implies $(N_{p,2}^v)$ for all $p$, $1\le p\le q'$. The case $p=1$ yields  $(N^v)$.
\end{proof}

\bigskip


Next we show a variation on (and generalisation of) Proposition \ref{equivalence}, which yields  characterisations of (\ref{vEpq}) in terms of some
resolvent type conditions and of some generalised forms of $(GN_q^v)$.  
\begin{proposition}\label{resolvent2}
Let $1\le p< q\le +\infty$ and $\beta>(\frac{1}{p}-\frac{1}{q})\kappa_v$,  where $\kappa_v$ is as in \eqref{dv}.  Assume that $(e^{-tL})_{t>0}$ is  bounded analytic on $L^p(M,\mu)$. Then the following
conditions are equivalent:
\begin{equation}\label{1}
  \tag{$vE_{p,q}$}  \sup_{t>0}\|{v_{\sqrt{t}}^{\frac{1}{p}-\frac{1}{q}}}e^{-tL}\|_{p\to q}<+\infty
\end{equation}
\begin{equation}\label{vRpqbeta}
  \tag{$vR_{p,q,\beta}$}  \sup_{t>0}\|{v_{\sqrt{t}}^{\frac{1}{p}-\frac{1}{q}}}(I+tL)^{-\beta/2}\|_{p\to q}<+\infty,
\end{equation}
 \begin{equation*}\label{GNpqbeta}
\tag{$GN_{p,q,\beta}^v$} \|fv_{r}^{\frac{1}{p}-\frac{1}{q}}
\|_{q}\leq C(\|f\|_{p}+r^{\beta}\|L^{\beta/2}f\|_{p}), \quad
\forall\, r>0, \,f\in\mathcal{D}_p(L^{\beta/2}).
\end{equation*}
\end{proposition}
Note that the condition  $\beta > (\frac{1}{p}-\frac{1}{q})\kappa_v$, together with $p<q\leq +\infty$,  means that either $\beta> \frac{\kappa_v}{p}$ and $q\in (p, +\infty]$,  or
$\beta\le  \frac{\kappa_v}{p}$ and $q\in (p, \frac{p\kappa_v}{\kappa_v-p\beta})$.
Note also that   $(GN_{2,q,1}^v)$ is nothing but $(GN_{q}^v)$. In particular, taking $p=2$ and $\beta=1$
in the proof below yields an interesting alternative proof to the implication from $(vE_{2,q})$ to $(GN_{q}^v)$ in Proposition \ref{equivalence}.

\begin{proof}

Note first that  (\ref{vRpqbeta}) and (\ref{GNpqbeta}) can be rewritten respectively as
\begin{equation*}
    \|fv_{\sqrt{t}}^{\frac{1}{p}-\frac{1}{q}}\|_{q}\leq C
    \|(I+tL)^{\beta/2}f\|_{p},
\end{equation*}
uniformly in $t>0$ and $f\in\mathcal{D}_p(L^{\beta/2})$
and
\begin{equation*}
    \|fv_{\sqrt{t}}^{\frac{1}{p}-\frac{1}{q}}\|_{q}\leq C
    (\|f\|_{p}+t^{\beta/2}\|L^{\beta/2}f\|_{p}),
\end{equation*}
uniformly in $t>0$  and $f\in\mathcal{D}_p(L^{\beta/2})$.
The equivalence between (\ref{GNpqbeta}) and (\ref{vRpqbeta}) follows therefore from 
\begin{equation}\label{sum}
  \|(I+tL)^{\beta/2}f\|_{p} \simeq  \|(I+(tL)^{\beta/2})f\|_{p}
  \simeq  \|f\|_{p}+t^{\beta/2}\|L^{\beta/2}f\|_{p},
\end{equation}
 uniformly in $t>0$ and $f\in\mathcal{D}_p(L^{\beta/2})$.
 
 The norm equivalence \eqref{sum} is classical (see for instance \cite[Proposition 3.1]{BBR}).  Let us  sketch a proof for the sake of completeness. 
In order to prove \eqref{sum}, it is clearly enough to prove   \begin{equation}\label{ms}
    \sup_{t>0}\|(I+tL)^{\beta/2}(I+(tL)^{\beta/2})^{-1}\|_{p\to
    p}<+\infty,
\end{equation}
\begin{equation}\label{msa}
    \sup_{t>0}\|(I+(tL)^{\beta/2})(I+tL)^{-\beta/2}\|_{p\to
    p}<+\infty,
\end{equation}
\begin{equation}\label{ma}
    \sup_{t>0}\|(I+tL)^{-\beta/2}\|_{p\to
    p}<+\infty
\end{equation}
and
\begin{equation}\label{mb}
    \sup_{t>0}\|(tL)^{\beta/2}(I+tL)^{-\beta/2}\|_{p\to
    p}<+\infty.
\end{equation}
Note that \eqref{msa} obviously follows from \eqref{ma} and \eqref{mb}. An equivalent formulation of  \eqref{ms} is
$$
 \sup_{t>0}\|(I+t L)^{\beta/2}(I+(t L)^{\beta/2})^{-1}-I\|_{p \to p} <+\infty.
$$
Set $F(z)= (1+z)^{\beta/2}(1+z^{\beta/2})^{-1}-1$, which can be defined as an analytic function on $\C\setminus (-\infty,0]$.
One checks easily that
$$
|F(z)|\le  C\min(|z|^b,|z|^{-b}),
$$
where $b=\min(1,\beta/2) $.
On the other hand, by Hille-Yosida,
\begin{equation}\label{res}
\|(tL+zI)^ {-1}\|_{p \to p}\le C(-\mbox{Re} z)^{-1},
\end{equation}
for all $t>0$ and $z\in\C$ such that Re$z <0$.

Hence
$$
F(tL)=\int_\Gamma F(z) (tL+zI)^ {-1}\,dz,
$$
where the curve  $\Gamma$ consists of two half-lines $re^{i\theta_i}$,  $r>0$, and $\theta_1,\theta_2$ chosen so that
$\pi/2 < \theta_1 < \pi$ and $\pi < \theta_2 < 3\pi/2$.
Finally, using \eqref{res},
\begin{eqnarray*}
\|F(t L)\|
&\le& C\sum_{i=1,2}  \int_0^\infty \min(r^b,r^{-b}) \|(tL+
re^{i\theta_i}I)^ {-1}\|_{p\to p} \,dr\\
&\le& 2 C \int_0^\infty\frac{\min(r^b,r^{-b})}{r} dr = C'.
\end{eqnarray*}
This proves \eqref{ms}, and \eqref{ma},  \eqref{mb} can be proved in the same way.

\bigskip

Now for the equivalence between (\ref{vEpq}) and (\ref{vRpqbeta}).

\bigskip

(\ref{vEpq})$\Rightarrow$ (\ref{vRpqbeta}). Note that
\begin{equation}
   (I+tL)^{-\beta/2}=\frac{1}{\Gamma(\beta/2)}\int_{0}^{+\infty}e^{-s}s^{\beta/2-1}e^{-s(tL)}\,
    \,ds,
\end{equation}
so that
\begin{equation*}
{v_{\sqrt{t}}^{\frac{1}{p}-\frac{1}{q}}}(I+tL)^{-\beta/2}=\frac{1}{\Gamma(\beta/2)}\int_{0}^{+\infty}e^{-s}s^{\beta/2-1}{v_{\sqrt{t}}^{\frac{1}{p}-\frac{1}{q}}}e^{-stL}\,ds.
\end{equation*}
Hence
\begin{eqnarray*}
  && \|{v_{\sqrt{t}}^{\frac{1}{p}-\frac{1}{q}}}(I+tL)^{-\beta/2}\|_{p\to
q}\le
\frac{1}{\Gamma(\beta/2)}\int_{0}^{+\infty}e^{-s}s^{\beta/2-1}\|{v_{\sqrt{t}}^{\frac{1}{p}-\frac{1}{q}}}e^{-stL}\|_{p\to
q}\,\,ds\\
 &\le& \frac{1}{\Gamma(\beta/2)}\int_{0}^{+\infty}e^{-s}s^{\beta/2-1}      \|v_{\sqrt{t}}/v_{\sqrt{st}}\|_{\infty}^{\frac{1}{p}-\frac{1}{q}}     \|{v_{\sqrt{st}}^{\frac{1}{p}-\frac{1}{q}}}e^{-stL}\|_{p\to
q}\,\,ds.
 \end{eqnarray*}
Using \eqref{dv} and assumption (\ref{vEpq}), we obtain
  $$\|{v_{\sqrt{t}}^{\frac{1}{p}-\frac{1}{q}}}(I+tL)^{-\beta/2}\|_{p\to
q} \le \frac{C}{\Gamma(\beta/2)}\int_{0}^{+\infty}e^{-s}s^{\beta/2-1}\max\left(1,\frac{1}{s}\right)^{\frac{\kappa_v}{2}(\frac{1}{p}-\frac{1}{q})}\,ds,$$
which is finite since $\beta>\kappa_v(\frac{1}{p}-\frac{1}{q})$.

\bigskip

(\ref{vRpqbeta})$\Rightarrow$ (\ref{vEpq}). Observe that
$$
  \|{v_{\sqrt{t}}^{\frac{1}  {p}-\frac{1}{q}}}e^{-tL}\|_{p\to
q}\le \|{v_{\sqrt{t}}^{\frac{1}{p}-\frac{1}{q}}}(I+tL)^{-\beta/2}\|_{p\to q}  \|(I+tL)^{\beta/2}e^{-tL}\|_{p\to
p}.$$
Now, according to \eqref{sum}, 
$$\|(I+tL)^{\beta/2}e^{-tL}\|_{p\to
p}\le C\left(\|e^{-tL}\|_{p\to
p}+\|(tL)^{\beta/2}e^{-tL}\|_{p\to
p}\right),$$  and the RHS is  bounded uniformly in $t>0$ by bounded analyticity on $L^p(M,\mu)$ of $(e^{-tL})_{t>0}$. This yields the claim.
\end{proof}

Theorem  $\ref{lemdg}$ and   Proposition $\ref{resolvent2}$ yield the following.
\begin{theorem}\label{GN23} 
Let $(M,d,\mu)$ be a metric measure space,  $L$ a  non-negative self-adjoint operator on $L^2(M,\mu)$, and  $v$ a function from $M\times \R_+$ to $\R_+$ satisfying $(A)$, $(D_v)$ and \eqref{D2}. Assume that 
 $(M,d,\mu,L)$ satisfies 
 \eqref{DG2} and that $(e^{-tL})_{t>0}$  is  bounded analytic  on $L^{{p_0}}(M,\mu)$ for some $p_0\in[1,2)$. Then $(DU\!E^v)$ implies $(GN^v_{{p},q,\beta})$ for all $p,q$ such that $p_0\le p< q\le p'_0$ and $\beta$ such that $\beta>(\frac{1}{{p}}-\frac{1}{q})\kappa_v$, where $\kappa_v$ is as in \eqref{dv}.  
\end{theorem}


Let us now emphasise a consequence of the particular case  $p=2$, $q=+\infty$, and $\beta>\kappa_v/2$ of Proposition \ref{resolvent2}, where we take advantage of the fact that, according to Corollary \ref{implication}, $(vE_{2,\infty})$ is equivalent to $(DU\!E^v)$.
This yields a direct characterisation of $(DU\!E^v)$ in terms of a Gagliardo-Nirenberg inequality (compare with \cite{Cou2} where, in the case where  $v=V$  does not depend on $x$ and is polynomial  in $r$, no extrapolation is needed  for the case $q=+\infty$). This result is much easier to obtain than Theorem \ref{mainDG}.

\begin{coro}\label{55}
Let $\beta>\frac{\kappa_v}{2}$,  where $\kappa_v$ is as in \eqref{dv}.  Then $(DU\!E^v)$ is
 equivalent to
 \begin{equation*}\label{GNpqinfty}
\tag{$GN_{2,\infty,\beta}^v$} \|f\sqrt{v_{r}}
\|_{\infty}\leq C(\|f\|_{2}+r^{\beta}\|L^{\beta/2}f\|_{2}), \quad
\forall\, r>0, \,f\in\mathcal{D}_2(L^{\beta/2}).
\end{equation*}
\end{coro}

 The drawback of the above result is that it  involves a high power of $L$ in the expression $\|L^{\beta/2}f\|_{2}$, instead of the Dirichlet form $\mathcal{E}$, which is much easier to handle in applications. For instance, unless $\kappa_v<2$, in which case one can choose $\beta=1$, it is not clear how to see from the Corollary \ref{55}  that $(DU\!E^v)$ is invariant under quasi-isometry. If one insists, as one should, on taking $\beta=1$, one cannot in general take $p=2,q=+\infty$. This is why the implication 
 from $(GN_q)$  to $(DU\!E^v)$ will require an extrapolation argument.

\subsection{Converses in the uniform  case}\label{uni} Again, let $(M,\mu)$   be a measure space and $L$ a  non-negative self-adjoint operator on $L^2(M,\mu)$. In this section, we shall study the case where $v$ does not depend on $x\in M$, but only on $r>0$. We shall see that in this particular case, if in addition $(e^{-tL})_{t>0}$ is uniformly bounded on $L^{1}(M,\mu)$, one can  prove the converse of  Propositions \ref{Nash22} and \ref{dg} by using existing arguments, and conclude that $(DU\!E^v)$, $(N^v)$ and 
$(GN_q^v)$ for $q>2$ small enough are equivalent.  The general case will require new arguments and more assumptions. It will be treated in Sections \ref{DG} and \ref{NG}.  

Let us start with  Nash type inequalities. Since they are $L^1-L^2$ inequalities, one can derive $(DU\!E^v)$ from them without any interpolation argument.   Unfortunately, we do not see so far how to implement the argument of  Lemma \ref{nacou} below in a non-uniform situation.
We will consider  $(N^v)$, but also  $(\widetilde{N}^v)$ introduced in Section  \ref{HN}, which will enable us to go beyond condition $(D_v)$.  The following statement elaborates on \cite[Proposition II.1]{C-N}. Assume for simplicity that $v$ is one-to-one from $\R_+$ onto itself and $\mathcal{C}^1$. This excludes for instance the case where $v=V$ and $M$ has finite measure, which can  probably be also treated with similar methods; we leave the details to the reader.
Say that $v$ satisfies $(*_v)$ if  $U(r)=\log v(r)$ is such that 
$$U'(s)\ge \sigma U'(r), \ \forall\,r>0, \ \forall\,s\in[r,2r],$$
for some $\sigma>0$. Functions $v(r)= \exp(r^\alpha), r^\alpha$, $\alpha>0$, and many others satisfy $(*_v)$.

\begin{proposition}\label{extrapolationuna}  Assume that  $(e^{-tL})_{t>0}$ is uniformly bounded on $L^{1}(M,\mu)$ and that $v$ satisfies $(A)$ but does not depend on $x\in M$.    Then,
if $v$ satisfies $(D_v)$,  $(N^v)$ implies $(DU\!E^v)$ and if $v$ satisfies $(*_v)$,  $(\widetilde{N}^v)$ implies $(DU\!E^v)$.
\end{proposition}

When $v$ does not depend on $x$, the $v$-Nash inequality $(N^v)$ reads
\begin{equation*}
 \|f \|_{2}^2\leq
C\left(\frac{\|f\|_1^2}{v(r)}+r^2\mathcal{E}(f)\right),\ \forall\,f\in\mathcal{F},\ \forall\,r>0.
\end{equation*}
Choosing $\frac{1}{v(r)}=\frac{\|f\|_{2}^{2}}{2C\|f\|_{1}^2}$, that is $r=v^{-1}\left(\frac{2C\|f\|_{1}^{2}}{\|f\|_{2}^2}\right)$ yields
\begin{equation}\label{nc}
\|f\|_{1}^2\,\theta_1\left(\frac{\|f\|_{2}^{2}}{\|f\|_{1}^2}\right)\le \mathcal{E}(f), \ \forall\,f\in\mathcal{F}\setminus\{0\},
\end{equation}
 where
$\theta_1(\tau)= \frac{\tau}{2C\left[v^{-1}\left(\frac{2C}{\tau}\right)\right]^2}$. Note that it follows from our assumptions on $v$ that $\tau\to\frac{\theta_1(\tau)}{\tau}$ is non-decreasing and continuous.

Similarly, when $v$ does not depend on $x$, $(\widetilde{N}^v)$ reads
\begin{equation*}\label{NnDu}
\|f \|_{2}^2\log\frac{cv(r)\|f \|_{2}^2}{\|f\|_{1}^2}\leq
r^2\mathcal{E}(f), \quad
\forall\, r>0, \,f\in \mathcal{F}\setminus\{0\},
\end{equation*}
and can be rewritten as
\begin{equation}\label{ncti}
\|f\|_{1}^2\,\theta_2\left(\frac{\|f\|_{2}^{2}}{\|f\|_{1}^2}\right)\le \mathcal{E}(f), \ \forall\,f\in\mathcal{F}\setminus\{0\},
\end{equation}
 where
\begin{equation*}
\theta_2(\tau)=\tau\sup_{r>0}\frac{\log\left(cv(r)\tau\right)}{r^2}.
\end{equation*}
If $(\widetilde{N}^v)$ holds, this supremum has to be finite (this is certainly the case if $v$ has at most exponential growth in addition to the above assumptions),
and under our assumptions on $v$ it is always positive.
Again, note  that $\tau\to\frac{\theta_2(\tau)}{\tau}$ is non-decreasing and continuous.
To show continuity at $\tau_1>0$,   let $r_1$ be such that $ cv(r_1)\tau_1=1/2$.
Then $\frac{\theta_2(\tau_1)}{\tau_1}=\sup_{r>r_1}\frac{\log\left(cv(r)\tau_1\right)}{r^2}$ and, for $\tau_2\le 2\tau_1$, $\frac{\theta_2(\tau_2)}{\tau_2}=\sup_{r>r_1}\frac{\log\left(cv(r)\tau_2\right)}{r^2}$.
Now
 \begin{eqnarray*}
\left|\frac{\theta_2(\tau_1)}{\tau_1}-\frac{\theta_2(\tau_2)}{\tau_2}\right| &\le& \sup_{r>r_1}\left|\frac{\log\left(cv(r)\tau_1\right)}{r^2}-\frac{\log\left(cv(r)\tau_2\right)}{r^2}\right|\\
&\le&r_1^{-2}|\log(\tau_1/\tau_2)|.\end{eqnarray*}

According to \cite[p.414]{G2}, see also \cite[Section 3.3]{BPS}, if $v$ satisfies $(*_v)$, then
$$\theta_2(\tau)\ge \widetilde{\theta}_2(\tau)=\frac{\tilde{c}\,\tau^2v'\left(v^{-1}(\frac{1}{\tau})\right)}{v^{-1}(\frac{1}{\tau})},$$
where $\tilde{c}$ depends on $c$ in $(\widetilde{N}^v)$ and on $\sigma$ in $(*_v)$.

Then Nash's argument as adapted in \cite{C-N} shows that if $(e^{-tL})_{t>0}$ is uniformly bounded on $L^1(M,\mu)$, then
inequalities like \eqref{nc} or \eqref{ncti} imply a heat kernel upper bound:

\begin{lemma}\label{nacou}  Assume that $(e^{-tL})_{t>0}$ is uniformly bounded on $L^1(M,\mu)$ and that
\begin{equation}\label{nc1}
\|f\|_{1}^2\,\theta\left(\frac{\|f\|_{2}^{2}}{\|f\|_{1}^2}\right)\le \mathcal{E}(f), \ \forall\,f\in\mathcal{F}\setminus\{0\},
\end{equation}
where $\theta:\R_+\to\R_+$ is continuous and $\tau\to\frac{\theta(\tau)}{\tau}$ is non-decreasing.
Assume that $$\int^{+\infty}\frac{d\tau}{\theta(\tau)}<+\infty.$$
Then $(DU\!E^w)$ holds, for $w$ defined by
\begin{equation}\label{wm}
w(r)=\frac{1}{A^2m(r^2/2)}
\end{equation}
 and \begin{equation}\label{tetam}
\int_{m(t)}^{+\infty}\frac{d\tau}{\theta(\tau)}=2t,
\end{equation}
where $A=\sup_{t>0}\|e^{-tL}\|_{1\to 1}<+\infty$.
\end{lemma}
\begin{proof}
Substitute $e^{-tL}f$ to $f$ in \eqref{nc1}.
Use the fact that $$\frac{\|e^{-tL}f\|_2^2}{\|e^{-tL}f\|_1^2}\ge\frac{\|e^{-tL}f\|_2^2}{A^2\|f\|_1^2}$$
and that  the function $\tau\to\frac{\theta(\tau)}{\tau}$ is non-decreasing.
If follows that
\begin{equation}\label{hyde}
A^2\|f\|_{1}^2\,\theta\left(\frac{\|e^{-tL}f\|_{2}^{2}}{A^2\|f\|_{1}^2}\right)\le \mathcal{E}(e^{-tL}f), \ \forall\,f\in\mathcal{F}\setminus\{0\},\,t>0.
\end{equation}
Set $u(t)=\frac{\|e^{-tL}f\|_2^2}{A^2\|f\|_{1}^2}$. 
Since  $\frac{d}{dt}\|e^{-tL}f\|_2^2=-2\mathcal{E}(e^{-tL}f)$, \eqref{hyde} becomes
$$
\theta\left(u(t)\right)\le -\frac{u'(t)}{2}, \ t>0,
$$
hence
\begin{equation}\label{intineq}
\int_{u(t)}^{u(0)}\frac{d\tau}{\theta(\tau)}\ge 2t, \ t>0.
\end{equation}

Define $m(t)$ by  \begin{equation}\label{intineqm}
 \int_{m(t)}^{+\infty}\frac{d\tau}{\theta(\tau)}=2t.
 \end{equation}

It follows from \eqref{intineq} and \eqref{intineqm} that $u(t)\le m(t)$,
that is
$$\|e^{-tL}f\|_2^2\le A^2m(t)\|f\|_1^2,$$
in other words
$$\|e^{-tL}\|_{1\to 2}\le A\sqrt{m(t)}.$$
By duality $\|e^{-tL}\|_{2\to \infty}\le A\sqrt{m(t)},$
hence by writing
$$\|e^{-tL}f\|_{1\to \infty}\le \|e^{-(t/2)L}\|_{2\to \infty}\|e^{-(t/2)L}\|_{1\to 2}$$
it follows that $$\|e^{-tL}\|_{1\to \infty}\le A^2\,m(t/2)=\frac{1}{w(\sqrt{t})}.$$

\end{proof}

\begin{lemma}\label{rod} Functions $\theta_1$ and $\theta_2$ satisfy the assumptions of Lemma  $\ref{nacou}$. If $w_1$ and $w_2$ are the associated functions defined via
\eqref{wm} and \eqref{tetam}, then, if $v$ is doubling, there exists $C>0$ such that $w_1(r)\ge Cv(r), \,\forall\,r>0,$ and, if $v$ satisfies $(*_v)$, there exist $C,c>0$ such that $w_2(r)\ge Cv(cr), \,\forall\,r>0$.
In particular,  under these assumptions, $(DU\!E^{w_1})$, resp. $(DU\!E^{w_2})$, implies $(DU\!E^v)$.
\end{lemma}

\begin{proof}
We have already observed that the functions $\tau\to\frac{\theta_i(\tau)}{\tau}$, $i=1,2$, are non-decreasing and continuous. 
The computations below will prove that $\int^{+\infty}\frac{d\tau}{\theta_i(\tau)}<+\infty$, $i=1,2$. This gives the first assertion of the lemma.

As for the second assertion, let us start with $\theta_2$ which is simpler to treat. 
Changing variables  $\tau=\frac{1}{v(r)}$ in the expression of $\widetilde{\theta}_2$ yields
\begin{equation}\label{chacha}
\int_{\frac{1}{v(\sqrt{t})}}^{+\infty}\frac{d\tau}{\widetilde{\theta}_2(\tau)}=\frac{1}{\tilde{c}}\int_0^{\sqrt{t}}r\,dr=\frac{t}{2\tilde{c}}
\end{equation}
therefore 
$$ \int_{\frac{1}{v(\sqrt{4\tilde{c}t})}}^{+\infty}\frac{d\tau}{\theta_2(\tau)}\le \int_{\frac{1}{v(\sqrt{4\tilde{c}t})}}^{+\infty}\frac{d\tau}{\widetilde{\theta}_2(\tau)}\le 2t= \int_{m_2(t)}^{+\infty}\frac{d\tau}{\theta_2(\tau)}.$$
The first inequality follows from the comparison between $\theta_2$ and $\widetilde{\theta}_2$, the second one from \eqref{chacha}, and the equality is the definition of $m_2$ (with obvious notation).
From $$ \int_{\frac{1}{v(\sqrt{4\tilde{c}t})}}^{+\infty}\frac{d\tau}{\theta_2(\tau)}\le  \int_{m_2(t)}^{+\infty}\frac{d\tau}{\theta_2(\tau)}$$
it follows that  $$m_2(t)\le \frac{1}{v(\sqrt{4\tilde{c}t})},$$
thus $$\frac{1}{w_2(\sqrt{t})}=A^2\,m_2(t/2)\le \frac{A^2}{v(\sqrt{2\tilde{c}t})}.$$

Now for $\theta_1$.
We are going to use a trick from \cite{BPS}.
Consider $\tilde{v}(r)=\frac{2}{r}\int_{r/2}^rv(s)\,ds$. Clearly, $$v(r/2)\le \tilde{v}(r)\le v(r),$$ and by $(D_v)$,
$\tilde{v}(r)$ and $v(r)$ are  within multiplicative constants. One can check by a simple calculation (see \cite[Lemma 2.1]{BPS}) that $\tilde{v}$ is also one-to-one.
Define $\widetilde{\theta}_1(\tau):=\frac{\tau}{\left[\tilde{v}^{-1}\left(\frac{1}{\tau}\right)\right]^2}$.  Again, $\widetilde{\theta}_1$ and $\theta_1$ are uniformly comparable.

 By the change of variables $\tau=\frac{1}{\tilde{v}(r)}$,
 $$\int_{\frac{1}{\tilde{v}(\sqrt{t})}}^{+\infty}\frac{d\tau}{\widetilde{\theta}_1(\tau)}=\int_0^{\sqrt{t}}\frac{\tilde{v}'(r)r^2}{\tilde{v}(r)}\,dr.$$
But again by calculations similar to the ones in \cite[Lemma 2.1]{BPS}, $$\frac{\tilde{v}'(r)}{\tilde{v}(r)}\le \frac{C}{r},$$
therefore $$\int_{\frac{1}{\tilde{v}(\sqrt{t})}}^{+\infty}\frac{d\tau}{\widetilde{\theta}_1(\tau)}\le \frac{Ct}{2},$$
hence
$$\int_{\frac{1}{\tilde{v}(\sqrt{t})}}^{+\infty}\frac{d\tau}{\theta_1(\tau)}\le C't,$$
$$\int_{\frac{1}{v(\sqrt{2t/C'})}}^{+\infty}\frac{d\tau}{\theta_1(\tau)}\le 2t= \int_{m_1(t)}^{+\infty}\frac{d\tau}{\theta_1(\tau)},$$
thus $$m_1(t)\le \frac{1}{v(\sqrt{2t/C'})},$$
and $$\frac{1}{w_1(\sqrt{t})}=A^2 m_1(t/2)\le \frac{A^2}{v(\sqrt{t/C'})}\le \frac{A'}{v(\sqrt{t})},$$
where we use $(D_v)$ in the last inequality.

\end{proof}

Lemmas  \ref{nacou} and \ref{rod} together yield Proposition \ref{extrapolationuna}.

\bigskip

Consider now $(GN_q^v)$ for some $q>2$ and assume $(D_v)$. When $v$ does not depend on $x$, 
$(GN_q^v)$ reads:
\begin{equation*}
v^{1-\frac{2}{q}}(r) \|f \|_{q}^2\leq
C\left(\|f\|_2^2+r^2\mathcal{E}(f)\right),\ \forall\,f\in \mathcal{F},\ \forall\,r>0,
\end{equation*}
that is
\begin{equation*}
\frac{\|f \|_{q}^2}{\|f \|_{2}^2}\leq
C\left(\frac{1}{v^{1-\frac{2}{q}}(r)}+\frac{r^2\mathcal{E}(f)}{v^{1-\frac{2}{q}}(r)\|f \|_{2}^2}\right),\ \forall\,f\in \mathcal{F}\setminus\{0\},\ \forall\,r>0,
\end{equation*}
or
\begin{equation*}
\frac{\|f \|_{q}^2}{\|f \|_{2}^2}\leq
K\left(\frac{\mathcal{E}(f)}{\|f \|_{2}^2}\right),\ \forall\,f\in \mathcal{F}\setminus\{0\},
\end{equation*}
where $$K(s)=C\inf_{r>0}\frac{1+r^2s}{v^{1-\frac{2}{q}}(r)}.$$
Note that if $v$ satisfies  \eqref{dv}  then $v(r)\le C\,v(1)r^{\kappa_v}$, $r\ge 1$, and if  $q$ is such that
 $\frac{q-2}{q}\kappa_v<2$, $\frac{r^2}{v^{1-\frac{2}{q}}(r)}\to+\infty$ as $r\to+\infty$, and $K(s)$ is positive and finite for every $s>0$.
 One checks easily that $K$ is one-to-one from $\R_+$ into itself.
Finally $(GN_q^v)$
 can be written in the more concise form
\begin{equation*}
\|f \|_{2}^2\,\eta\left(\frac{\|f \|_{q}^2}{\|f \|_{2}^2}\right)\leq
\mathcal{E}(f), \ \forall\,f\in \mathcal{F}\setminus\{0\},
\end{equation*}
 where
 $\eta(\tau):=K^{-1}(\tau)$.
  If $v$ is doubling, then choosing $\tau=r^2$ yields
$$\eta(\tau)\ge\frac{c}{\left(v^{-1}\left(\frac{C}{\tau^{q/(q-2)}}\right)\right)^2}.$$
This yields a more general version of the inequalities in \cite{Cou2}.

\bigskip

To go back from $(GN_q^v)$  to $(DU\!E^v)$,  we will use the equivalence between $(GN_q^v)$ and $(vE_{2,q})$, which does not require the above optimisation, 
and we will extrapolate from $(vE_{2,q})$ to $(vE_{2,\infty})$, which will require a uniform boundedness assumption on $L^1$.
Indeed, the  extrapolation lemma  \cite[Lemma 1]{Cou1} can be extended  to the situation where the decay of the semigroup is governed by a doubling function of time instead of a  power function.

\begin{proposition}[\cite{CM}, Lemma 1.3]\label{extrapol} Assume that $(e^{-tL})_{t>0}$ is uniformly bounded on $L^{1}(M,\mu)$. Let $w$ be a  non-decreasing positive  function on $(0,+\infty)$ satisfying the doubling condition $(D_w)$. If there exist  $1\le p < q \le +\infty$ such that:
\begin{equation}\label{varphi}
\|e^{-tL}\| _{p\to q}  \le \frac{1}{w (t)},\, \forall\,t>0,
\end{equation}
 then there exists a constant $C$ such that
\begin{equation*}
\|e^{-tL}\|_{1\to \infty}\le \frac{C}{w^{\alpha} (t)},\, \forall\, t>0,\quad \text{with} \,\,\alpha=1/(1/p - 1/q).
\end{equation*}
\end{proposition}

\begin{proposition}\label{extrapolationu}  Assume that $v:\R_+\to \R_+$ satisfies $(A)$ and $(D_v)$ but does not depend on $x\in M$ and that  $(e^{-tL})_{t>0}$ is uniformly bounded on $L^{1}(M,\mu)$. Then, for all $q$ such that $2<q \le +\infty$,
$(GN_q^v)$ implies $(DU\!E^v)$.
\end{proposition}
\begin{proof}
By Proposition \ref{equivalence}, $(GN_{q}^v)$ implies $(vE_{2,q})$. Now, since $v$ does not depend on $x$, $(vE_{2,q})$ is equivalent to \eqref{varphi} with $p=2$  and $w(t)=cv^{\frac{1}{2}-\frac{1}{q}}(\sqrt{t})$, $c>0$. Note that $w$ satisfies $(D_w)$ since $v$ satisfies  $(D_v)$.  Proposition \ref{extrapol} then yields
\begin{equation*}
\|e^{-tL}\|_{1\to \infty}\le \frac{C'}{v (\sqrt{t})},\, \forall\, t>0,
\end{equation*}
which is obviously equivalent to $(DU\!E^v)$.
\end{proof}

In Proposition \ref{realconverse} below, we shall be able to drop the assumption of independence on $x$. And this will rely on an adapted  extrapolation result, namely Proposition \ref{extrapolation} below, which will require the use of new ingredients.


\section{Local and global inequalities }\label{LG0}

The section will be devoted to a closer study of the relationship between on the one hand global inequalities such as $(GN_q^v)$ and $(N^v)$ and on the other hand local inequalities like  $(K\!G\!N_q^v)$,  $(LS_q^v)$,
   $(K\!N^v)$, and   $(LN^v)$, and also various forms of relative Faber-Krahn inequalities. More precisely, in Sections \ref{GNN} and \ref{LG}, we are going to see that  conditions $(GN_q^v)$,   $(K\!G\!N_q^v)$, and $(LS_q^v)$,
(resp. $(N^v)$,   $(K\!N^v)$, and   $(LN^v)$) are equivalent and that the first group implies the second one.
In Section \ref{R}, we shall  establish   the  link with various versions of Faber-Krahn inequalities.
In Section \ref{KLR}, we shall see in a systematic way that in the case where $(M,d,\mu)$ is non-compact and connected, the so-called reverse doubling property enables one to get rid of (or not to introduce) of a certain local term in  Nash and Faber-Krahn type inequalities.

From Section \ref{LG} on, we shall work in the setting of a metric measure space endowed with a  strongly local regular Dirichlet form  together with a proper distance, as described for instance in \cite{GS} or \cite{ST}.

\bigskip

\subsection{Gagliardo-Nirenberg implies Nash and global implies local}\label{GNN}

We will start by  showing that the implications in the following  diagram hold:
 \[
\xymatrix @R-10pt @C-10pt {
&(GN_q^v) \ar@{=>}[r]   \ar@{=>}[d]   &  (K\!G\!N_q^v) \ar@{=>}[r]\ar@{=>}[d]   &  (LS_q^v)\ar@{=>}[d] &&    
\\
&(N^v)  \ar@{=>}[r] & (K\!N^v)  \ar@{=>}[r]& (LN^v)&&
}
\]
The inequalities in the two first columns can be formulated on any  measure space $(M,\mu)$
endowed with a  nonnegative self-adjoint operator $L$. The ones in the last column require in addition $M$ to be endowed with a distance $d$.
Let us first consider  the vertical implications. Note that they do not require $(D_v)$ or \eqref{D2}.

\begin{proposition} \label{gn}
Let $(M,\mu)$  be a measure space, $L$ a  non-negative self-adjoint operator on $L^2(M,\mu)$ and  let $v:M\times \R_+\to \R_+$ satisfy $(A)$. For any $q>2$, $(GN_{q}^v)$ implies $(N^v)$ and $(K\!G\!N_{q}^v)$ implies
$(K\!N^v)$.  If in addition  $M$ is endowed with a metric, $(LS_q^v)$ implies $(LN^v)$.
\end{proposition}

\begin{proof}
 Let $q>2$. Let $\theta \in
]0,1[$ be such that $\frac{1}{2}=\frac{\theta}{q}+(1-\theta).$
By H\"older's inequality,
$$\|f\|_{2}\le \|fv_{r}^{\frac{1}{2}-\frac{1}{q}}\|_{q}^{\theta}\|fv_{r}^{-1/2}\|_{1}^{1-\theta}.$$
Hence  $(GN_q^v)$ yields
\begin{eqnarray*}
\|f\|_{2}^2&\le& C\left(\|f\|_{2}^2+r^2\mathcal{E}(f)\right)^{\theta}\|fv_{r}^{-1/2}\|_{1}^{2(1-\theta)}\\
&\le& C \varepsilon\left(\|f\|_{2}^2+r^2\mathcal{E}(f)\right)+C\varepsilon^{-\frac{\theta}{1-\theta}}\|fv_{r}^{-1/2}\|_{1}^2,
\end{eqnarray*}
for all  $r,\epsilon>0$, $f\in\mathcal{F}$.
Choosing $\varepsilon=\frac{1}{2C}$  proves the first assertion of the proposition. The second one can be proved in a similar way. The last one again follows directly from H\"older's inequality.
\end{proof}

\begin{rem}\label{BS} According to \cite{BCLS}, $(LS_q^v )$ and $(LN^v)$ are actually equivalent if the quadratic form associated with $L$ is a Dirichlet form. It will follow from Propositions $\ref{Carron}$ and $\ref{Carronbis}$ that
$(GN_q^v)$ and $(N^v)$ are also equivalent, at least in the setting of metric measure spaces endowed with a strongly local regular Dirichlet form and a proper intrinsic distance. See also Section \ref{NG} for a slightly more general setting where all these inequalities happen to be equivalent.
\end{rem}

\begin{rem} The same argument as in Proposition $\ref{gn}$ shows more generally that,   for all  $1\le \tilde{p}<p< q<+\infty$ and $\beta>0$,  \eqref{GNpqbeta} (see Proposition $\ref{resolvent2}$)  implies $(N_{\tilde{p},p,\beta})$ (see Remark $\ref{generalised Nash}$).
\end{rem}


Now for the horizontal implications in the above diagram. On the top line, we have already noticed that both implications were obvious, and on the bottom line, that the first one was obvious. From the first column to the second one, we need not assume $(D_v)$ or \eqref{D2}. Again, $(M,\mu)$ need not be endowed with a metric. To formulate the localised inequalities $(LS_q^v)$ and  $(LN^v)$, one does need a metric $d$. Then the implication from $(K\!G\!N_q^v)$ to $(LS_q^v)$ is obvious if one assumes  \eqref{D2}.  To complete the above diagram, it remains to prove:

\begin{proposition} \label{KL} Let $(M,d,\mu)$  be a metric measure space, $L$ a  non-negative self-adjoint operator on $L^2(M,\mu)$ and  let $v:M\times \R_+\to \R_+$ satisfy $(A)$, $(D_v)$,
 and \eqref{D2}. Then
$(K\!N^v)$  implies $(LN^v_{2/\kappa_v})$,   where $\kappa_v$ is as in \eqref{dv}.
\end{proposition}
\begin{proof}
Write $(K\!N^v)$:
\begin{equation}\label{knvs}
\|f \|_{2}^2\leq
C\left(\frac{\|f\|_{1}^2}{\inf\limits_{z\in\text{supp}(f)}v_s(z)}+s^2\mathcal{E}(f)\right), \quad
\forall\, s>0, \,f\in \mathcal{F}.
\end{equation}
Now consider a ball  $B=B(x,r)$ and let $f \in \mathcal{F}_c(B)$.
Since $\text{supp}(f)\subset B$,  \eqref{D2} implies $$v_r(x)\le C\inf\limits_{z\in \text{supp}(f)}v_r(z),\ \forall\,r>0.$$
For $r\ge s$,   \eqref{dv}  implies $$\inf\limits_{z\in\text{supp}(f)}v_r(z)\le C\left(\frac{r}{s}\right)^{\kappa_v}\inf\limits_{z\in\text{supp}(f)}v_s(z).$$  Gathering these two estimates yields,
for $x\in M$, $r\ge s>0$, $f \in
\mathcal{F}_c(B(x,r)),$
$$\frac{1}{\inf\limits_{z\in\text{supp}(f)}v_s(z)}\le C' \frac{(r/s)^{\kappa_v}}{v_r(x)}.$$
Thus  \eqref{knvs}  implies
\begin{equation*}
\int_B|f|^{2}\,d\mu\le
C\left(\frac{(r/s)^{\kappa_v}}{v_r(x)}\left(\int_B |f|\,d\mu \right)^2
+s^2\mathcal{E}(f) \right)
\end{equation*}
if $r\ge s>0$. 
In order to obtain an inequality which is also valid for $s\ge r>0$, it enough to add a term $\frac{s^2}{r^2}\int_B | f|^2 \,d\mu$ in the RHS :
\begin{equation*}
\int_B|f|^{2}\,d\mu\le
C\left(\frac{(r/s)^{\kappa_v}}{v_r(x)}\left(\int_B |f|\,d\mu \right)^2
+s^2\left( \mathcal{E}(f)+r^{-2}\int_B | f|^2 \,d\mu \right)\right).
\end{equation*}
Now taking the infimum in $s>0$ yields
\begin{equation*}
\left(\int_B|f|^{2}\,d\mu\right)^{\frac{2}{\kappa_v}+1}\le
\frac{Cr^2}{v_r^{2/\kappa_v}(x)}\left(\int_B |f|\,d\mu \right)^{4/\kappa_v}
\left( \mathcal{E}(f)+r^{-2}\int_B | f|^2 \,d\mu \right),
\end{equation*}
 for all $x\in M$,
 $r>0$, $f \in
\mathcal{F}_c(B(x,r))$,
that is, $(LN^v_\alpha)$ with $\alpha=2/\kappa_v$.
\end{proof}

\bigskip

One may wonder why the implication from $(K\!G\!N_q^v)$ to $(LS_q^v)$ is direct, as we have seen already in Section \ref{FMR}, whereas the one from $(K\!N^v)$  to $(LN^v_\alpha)$ requires first the consideration of two different values $r$ and $s$ respectively for the radius of the ball and the parameter in the inequality, then  the use of $(D_v)$ and an optimisation. There are two answers to this question and both are interesting.

The first one is that we could perform a similar optimisation on $(K\!G\!N_q^v)$. With the notation of Section \ref{KLR}, assume $(RD_v)$ and write
\begin{equation*}
v_r^{1-\frac{2}{q}}(x)\left(w(r,s)\right)^{\kappa_v(\frac{2}{q}-1)}\left(\int_B|f|^{q}\,d\mu\right)^{\frac{2}{q}}\le
C\left(\int_B | f|^2 \,d\mu +s^2 \mathcal{E}(f)\right),
\end{equation*}
 for all $x\in M$,
 $r,s>0$, $f \in
\mathcal{F}_c(B)$, $B=B(x,r)$. Then choose 
$s$ such that $$s^2 \mathcal{E}(f)=\int_B | f|^2 \,d\mu.$$
One obtains a  formally stronger form of $(LS_q)$ which should be in fact equivalent to $(LS_q)$ by the methods of \cite{BCLS}. We leave the details to the reader.

The second answer is that one could also perform a similar optimisation already  at the level of
$(N^v)$ by writing
$$\|fv_{s}^{-1/2}\|_{1}^2+s^2\mathcal{E}(f)\le C\left(\left(\frac{r}{s}\right)^{\kappa_v}\|fv_{r}^{-1/2}\|_{1}^2+s^2\left(\mathcal{E}(f)+r^{-2}\|f\|_2^2\right)\right),$$ and improve this inequality into
\begin{equation*}\tag{$\bar{N}_\alpha^v$}\label{na}
\|f \|_{2}^{2\left(1+\alpha\right)}\le C\|fv_{r}^{-1/2} \|_{1}^{2\alpha}\left(\|f\|^2_2+r^2\mathcal{E}(f)\right), \quad
\forall\, r>0, \,f\in \mathcal{F}
\end{equation*}
for some $\alpha>0$ depending on the constant in $(V\!D)$.
The implication from the latter inequality  to $(LN^v_\alpha)$ is then obvious by restricting oneself to functions supported in balls and using \eqref{D2}.

\subsection{From local to global }\label{LG}

In this section we show that, in the setting of  a doubling measure space endowed with a strongly local and regular Dirichlet form and a proper distance, one can go back from the family of localised inequalities $(LN^v)$  (resp. $(LS_q^v)$)  to the global inequality $(N^v)$ (resp. $(GN_q^v)$).  We are grateful to Gilles Carron (\cite{Ca}) for this  observation.

Recall that in this framework, it is well-known that $(LN)$, or  $(LS_q)$,  implies $(DUE)$   (see \cite{ST}, or  \cite[Section 5.2]{SA}). Together with Proposition \ref{dg}, the results in the present section therefore give a short-cut  to Theorem \ref{mainDG} in  the case $v=V$. Remember however that one of our main goals is precisely to give an alternative and more general approach to the  above equivalences from \cite{ST} or \cite{SA}.

Indeed, later in Section \ref{DGE},  we will see, in a slightly more general setting, first that  the strongest global inequality $(GN_q^v)$ is equivalent to $(DU\!E^v)$, second that 
the weakest of the local inequalities, namely  $(LN^v)$, implies back $(GN_q^v)$.  In particular, the local and the global inequalities are all equivalent. The current section is nevertheless important for clarity, since we will see this equivalence directly without going through the machinery of Section \ref{DGE}.


We will use the setting introduced for instance in \cite[2.2]{GS} (for more information see also  \cite{ST1, ST, HR, AH}, and \cite[Section 3]{SF}). We shall only recall the basic notions and notations. 
Let $(M,\mu)$  be  a  locally compact separable  measure space endowed with a Borel measure $\mu$ which is finite on compact sets and strictly positive on non-empty open sets. Let $L$ be a non-negative self-adjoint operator
on $L^2(M,\mu)$ and $\mathcal{E}$ the associated quadratic form with domain $\mathcal{F}$. 
Assume that $\mathcal{E}$ is a strongly local and regular Dirichlet form (see \cite{FOT} for definitions).  In particular, $(e^{-tL})_{t>0}$ is a submarkovian semigroup,
that is
$0\le e^{-tL}f\le 1$ if  $0\le f\le 1$.
Let $d\Gamma$ be the energy measure associated with $\mathcal{E}$,
that is
$$\mathcal{E}(f,g)=\int_M d\Gamma(f,g)$$
for all $f,g\in\mathcal{F}$.


Then $d\Gamma$ satisfies a Leibniz rule (see \cite[Lemma 3.2.5]{FOT} or \cite{AB}, p.?), which yields the following inequality between measures
\begin{equation}\label{leib}
d\Gamma(fg,fg)\le 2\left(f^2{d\Gamma}(g,g)+g^2{d\Gamma}(f,f)\right)
\end{equation}
for all $f,g\in\mathcal{F}\cap L^\infty(M,\mu)$, or rather their quasi-continuous versions (see for instance \cite[Lemma 2.5]{GS}).

Define now the intrinsic quasi-metric:
 \begin{equation}\label{defd}
d(x,y)=\sup\{f(x)-f(y) ; f\in \mathcal{F}\cap \mathcal{C}_0(M) \mbox{ s.t. }d\Gamma(f,f)\le d\mu\}.
\end{equation}
Here $\mathcal{C}_0(M)$ denotes the space of continuous functions on $M$ which vanish at infinity. 
 Assume that $d$ is finite everywhere, separates points, and defines the original topology on $M$; assume also that the metric space $(M,d)$ is complete.
 According to \cite[Lemma 1']{ST1} (see also \cite[Theorem 2.11]{GS}), $\forall\,x\in M$, $d_x\in \mathcal{F}_{loc}$, where $d_x(y):=d(x,y)$, (see \cite[Definition 2.3]{GS} for a precise definition of  $\mathcal{F}_{loc}$) and
$$d\Gamma(d_x,d_x)\le d\mu.$$ In fact, definition \eqref{defd} is not essential, as long as one has the latter properties.
In other words, we could consider any distance $d$ on $(M,\mu)$ defining the original topology, such that $(M,d)$ is complete  and
\begin{equation}\label{dede}
d_x\in \mathcal{F}_{loc}, \ d\Gamma(d_x,d_x)\le d\mu,\ \forall\,x\in M.
\end{equation}
It follows that the balls in $M$ are relatively compact (see \cite[footnote 4, p. 1215]{GT}.

In the sequel, let us say that $(M,d,\mu,L)$ satisfies $(H)$ if $(M,d,\mu)$  is  a  locally compact separable and complete metric  measure space endowed with a Borel measure $\mu$ which is finite on compact sets and strictly positive on non-empty open sets, $L$ is a non-negative self-adjoint operator
on $L^2(M,\mu)$ and the associated quadratic form  $\mathcal{E}$ with domain $\mathcal{F}$ is Dirichlet, strongly local and regular, and if  $d$ satisfies
\eqref{dede}, where $d\Gamma$ is the energy measure associated with $\mathcal{E}$.

\begin{lemma}\label{su} Assume that $(M,d,\mu,L)$ satisfies $(H)$. For $x\in M$ and  $r>0$, define $$\rho(y)=\rho^{r,\epsilon}_x(y):=\left(1-\frac{d(y,B(x,r-2\epsilon))}{\epsilon}\right)_+.$$
Then for all $g\in\mathcal{F}$, $g\rho\in \mathcal{F}_c(B(x,r))$ and
\begin{equation*}
\mathcal{E}(g\rho)
\le \frac{2}{\epsilon^2}\int_{B(x,r)}g^2\,d\mu+2\int_{B(x,r)}{d\Gamma}(g,g).
\end{equation*}
\end{lemma}

\begin{proof} Let us first observe that $\rho$ is supported in $\overline{B(x,r-\epsilon)}$ and that $\rho\equiv1$ on $B(x,r-2\epsilon)$.
According to \cite[Lemma 1.9]{ST} (see also \cite[Theorem 2.11]{GS}), $\rho\in \mathcal{F}$ and
$$d\Gamma(\rho,\rho)\le  \frac{1}{\epsilon^2}d\mu,$$
and in fact, due to the local character of $\mathcal{E}$ (see \cite[Corollary 3.2.1, p.115]{FOT}),
$$d\Gamma(\rho,\rho)\le  \frac{1}{\epsilon^2}\chi_{B(x,r)}d\mu.$$ 
 
 Using \eqref{leib}, $g\rho\in\mathcal{F}$ and
  \begin{eqnarray*}
\mathcal{E}(g\rho)=\int_M d\Gamma(g\rho,g\rho)&\le& 2\left(\int_Mg^2{d\Gamma}(\rho,\rho)+\int_M\rho^2{d\Gamma}(g,g)\right)\\
&\le& 2\left(\frac{1}{\epsilon^2}\int_{B(x,r)} g^2\,d\mu+\int_{B(x,r)}{d\Gamma}(g,g)\right).\end{eqnarray*}

\end{proof}

\begin{proposition} \label{Carron} Assume that $(M,d,\mu,L)$ satisfies $(H)$ and   $(V\!D)$ and that $v$ satisfies \eqref{D2}.
Then conditions  $(N^v)$,  $(K\!N^v)$, and $(LN^v)$ are equivalent.
\end{proposition}

\begin{proof} 
Given the considerations in Section \ref{GNN}, it only remains to prove that $(LN^v)$ implies $(N^v)$.
Assume  $(LN^v)$, that is
$$
\|f \|_{2}^{2(1+\alpha)}\le \frac{C}{v_r^\alpha(x)}\|f \|_{1}^{2\alpha}\left( r^2\mathcal{E}(f)+\|f\|_2^2\right),
$$
for every ball $B=B(x,r)$, for every $f\in \mathcal{F}_c(B)$,  and for some $\alpha, C>0$.
Using \eqref{D2}, this can be rewritten as
$$
\|f \|_{2}^{2(1+\alpha)}\le C\left\|f v_r^{-1/2}\right\|_{1}^{2\alpha}\left(r^2\mathcal{E}(f)+\|f\|_2^2\right),
$$
hence, for all $\varepsilon>0$, 
\begin{equation}\label{fini}
\|f \|_{2}^{2}\le C\varepsilon^{-1/\alpha}\left\|f v_r^{-1/2}\right\|_{1}^{2}+\varepsilon \left(r^2\mathcal{E}(f)+\|f\|_2^2\right).
\end{equation}
Invoking $(BC\!P)$,  consider a  covering of $M$ by balls $B(x_i,r/2)$, $i\in I$,
such that the balls $B(x_i,r/4)$ are pairwise disjoint. Recall that $K(x)=\#\{i\in I,\,x\in B(x_i,r)\}\le K_0$, where $K_0$  only depends on the constant in $(V\!D)$.
Define  cut-off functions $\rho_i$ by
$$\rho_i(x):=\left(1-\frac{4d(x,B(x_i,r/2))}{r}\right)_+,$$ that is, in the notation of Lemma \ref{su},  $\rho_i=\rho^{r,r/4}_{x_i}$. Let $g\in \mathcal{F}$. Since  $\rho_i\equiv 1$ on $B(x_i,r/2)$, one can write
$$
\|g\|_{2}^{2}\le \sum_i \|g\rho_i\|_{2}^{2}.$$
Since $g\rho_i\in\mathcal{F}_c(B(x_i,r))$ by Lemma \ref{su}, one can apply \eqref{fini} to each $g\rho_i$ in $B(x_i,r)$. 
It follows 
\begin{equation}\label{fofo}
\|g\|_{2}^{2}\le C\varepsilon^{-1/\alpha}\sum_i\left\|g\rho_i v_r^{-1/2}\right\|_{1}^{2}+\varepsilon r^2\sum_i\mathcal{E}(g\rho_i)+\varepsilon\sum_i\|g\rho_i\|_2^2.
\end{equation}
Now 
\begin{equation}\label{nono}
\sum_i\left\|g\rho_i v_r^{-1/2}\right\|_{1}^{2}\le \left(\sum_i\left\|g\rho_i v_r^{-1/2}\right\|_{1}\right)^{2}\le K_0^2 \left\|gv_r^{-1/2}\right\|_{1}^2.
\end{equation}
 On the other hand, by Lemma \ref{su}, $g\rho_i\in\mathcal{F}$ and 
\begin{equation*}
\mathcal{E}(g\rho_i)
\le \frac{32}{r^2}\int_{B(x_i,r)}g^2\,d\mu+2\int_{B(x_i,r)}d\Gamma(g,g),
\end{equation*}
hence
\begin{equation}\label{eg}
\sum_i\mathcal{E}(g\rho_i)\le \frac{32K_0}{r^2} \|g\|_2^2+2K_0\mathcal{E}(g),
\end{equation}
and  finally
\begin{equation}\label{fc}
\sum_i\|g\rho_i\|_2^2\le  K_0\|g\|_2^2.
\end{equation}
Gathering  \eqref{fofo},  \eqref{nono}, \eqref{eg} and \eqref{fc}, one obtains
$$
\|g\|_{2}^{2}\le  C\varepsilon^{-1/\alpha}K_0^2 \left\|g v_r^{-1/2}\right\|_{1}^2+33K_0\varepsilon \|g\|_2^2+2K_0\varepsilon r^2\mathcal{E}(g).
$$
Choosing $\varepsilon=\frac{1}{66K_0}$ yields $(N^v)$.

\end{proof}

\begin{proposition} \label{Carronbis} Assume that $(M,d,\mu,L)$ satisfies $(H)$  and $(V\!D)$ and that $v$ satisfies \eqref{D2}.
For all $q>2$, conditions $(GN_q^v)$, $(K\!G\!N_q^v)$, and $(LS_q^v)$ are equivalent.
\end{proposition} 

\begin{proof} 
It only remains to prove that $(LS_q^v)$ implies $(GN_q^v)$.
Assume  \eqref{Sq}, that is
$$
\|f \|_{q}^2\le \frac{C}{v_r^{1-\frac{2}{q}}(x)}\left(\|f\|^2_2+r^2\mathcal{E}(f)\right),
$$
for every ball $B=B(x,r)$, for every $f\in \mathcal{F}_c(B)$,  and for some $C>0$.
Using \eqref{D2}, this can be rewritten as
\begin{equation}\label{finil}
\|v_r^{\frac{1}{2}-\frac{1}{q}}f \|_{q}^{2}\le C\left(\|f\|^2_2+r^2\mathcal{E}(f)\right).
\end{equation}
Consider the $x_i$ and $\rho_i$ as before.
Let $g\in \mathcal{F}$.
Write
$$
\|v_r^{\frac{1}{2}-\frac{1}{q}}g\|_{q}^{q}\le \sum_i \|v_r^{\frac{1}{2}-\frac{1}{q}}g\rho_i\|_{q}^{q}$$
and apply \eqref{finil} to each $g\rho_i$ in $B(x_i,r)$.
It follows 
$$\|v_r^{\frac{1}{2}-\frac{1}{q}}g\|_{q}^{q}\le C\left(\sum_i\|g\rho_i\|^2_2+r^2\sum_i\mathcal{E}(g\rho_i)\right), 
$$
and one concludes by using \eqref{fc} and \eqref{eg}.

\end{proof}

A similar argument as in Propositions \ref{Carron} and \ref{Carronbis}  can also be applied if one substracts from an operator $L$ satisfying the above assumptions
a strongly  positive potential $\mathcal{V}$. 
Let $(M,\mu,L)$ be  as above, $\mathcal{E}$ the associated Dirichlet form, and let $\mathcal{V}$ be a positive function on $M$.  Following \cite{CZ}, we shall say that 
$L-\mathcal{V}$ is strongly positive (or strongly subcritical in the terminology of \cite{DS}) if there exists $0<\epsilon <1$
such that
\begin{equation}\label{av}
(1-\epsilon)\mathcal{E}(f) \ge  \|\mathcal{V}^{1/2}f\|_2^2.
\end{equation}
It follows that $L-\mathcal{V}$ is well-defined as an operator with dense domain on $L^2(M,\mu)$.
Indeed, according to \eqref{av}, the quadratic form 
$$\mathcal{E}_{\mathcal{V}}(f) :=\langle Lf-\mathcal{V}f,f   \rangle=\mathcal{E}(f) - \|\mathcal{V}^{1/2}f\|_2^2$$
satisfies
\begin{equation}\label{quad}
\epsilon\mathcal{E}\le \mathcal{E}_{\mathcal{V}}\le \mathcal{E}
\end{equation}
and is defined on the domain $\mathcal{F}$ of $\mathcal{E}$.

The  semigroup generated by $L-\mathcal{V}$  is not 
necessarily submarkovian 
and possibly does  not act on the whole range of $L^p$ spaces. As a matter of fact,  $\mathcal{E}_{\mathcal{V}}$ is no more a Dirichlet form in general and even when it is one, it is not strongly local but only local. In any case,   we cannot apply directly Propositions   \ref{Carron} and \ref{Carronbis}.
However, one can consider again the family $\{\rho_i\}_{i=1}^\infty$ 
introduced in the proof of Proposition \ref{Carron}.
Then
\begin{eqnarray*}
\sum_i\mathcal{E}_{\mathcal{V}}(g\rho_i)&=& \sum_i \mathcal{E}(g\rho_i) - \sum_i \|\mathcal{V}^{1/2}g\rho_i\|_2^2\\&\le&
\sum_i \mathcal{E}(g\rho_i) \\&\le& 
\frac{32K_0}{r^2} \|g\|_2^2+2K_0\mathcal{E}(g)\\
&\le& \frac{32K_0}{r^2} \|g\|_2^2+\frac{2}{\epsilon}K_0\mathcal{E}_\mathcal{V}(g).
\end{eqnarray*}
The one before last inequality is \eqref{eg} and the last one follows from the first inequality in \eqref{quad}.
Again the same argument as above can then be used to show that  conditions $(GN_q^v)$, $(K\!G\!N_q^v)$, and $(LS_q^v)$ associated with $\mathcal{E}_{\mathcal{V}}$ are equivalent.

\medskip 

A  cousin of inequality $(LS_q)$ was introduced in \cite[Proposition 2.1 (ii)]{AO} in the setting of a metric measure space $(M,d,\mu)$ endowed with a non-negative self-adjoint operator $L$, namely
\begin{equation}\label{OA}
\|\chi_{B(x,r)}f \|_{q}^2\le \frac{C}{V_r^{1-\frac{2}{q}}(x)}\left(\|f\|^2_2+r^2\mathcal{E}(f)\right),
\end{equation}
for all $x\in M$, $r>0$, $f\in \mathcal{F}$. It is shown there that \eqref{OA} implies a localised version of $(V\!E_{2,q})$.
Note that  restricting \eqref{OA} to $\mathcal{F}_c(B(x,r))$ yields $(LS_q)$. Now Proposition \ref{Carronbis} says that if $\mathcal{E}$ is strongly local and regular and $(M,d,\mu)$ satisfies $(V\!D)$, $(LS_q)$ implies $(GN_q)$ which implies the full $(V\!E_{2,q})$ by Proposition \ref{equivalence}. This yields an improvement of \cite[Proposition 2.1]{AO} in that situation, as well as an extension to the case where $v\not=V$.

\bigskip

\subsection{Nash and Faber-Krahn}\label{R}

In the setting of Riemannian manifolds, $(DU\!E)$ has been characterised by Grigor'yan in \cite{G2} in terms of relative Faber-Krahn inequalities. These methods also work in the setting of strongly local and  regular Dirichlet forms as in Section \ref{LG}.
Our aim in this section is to extend this characterisation to $(DU\!E^v)$ for $v\neq V$. More precisely, we are going to make the connection between
suitable versions of relative Faber-Krahn inequalities and localised $v$-Nash inequalities.  Together with  Proposition \ref{newmainDG} below, this will establish the connection between $(DU\!E^v)$ and these relative Faber-Krahn inequalities.  For any open subset $\Omega$ of a topological measure space $M$ endowed with a closed non-negative quadratic form $\mathcal{E}$, set
\begin{equation}\label{Barlow}
\lambda_1(\Omega):= \inf\left\{\frac{\mathcal{E}(f)}{\|f
\|_{2}^{2}},\,f\in \mathcal{F}_c(\Omega)\,\text{and}\,f\neq 0
\right\}.
\end{equation}
Of course one could replace $\mathcal{F}_c(\Omega)$ with its closure in $\mathcal{F}$ for the norm
$\mathcal{E}(f)+\|f
\|_{2}^{2}$ without changing anything. This definition is the one used for instance in \cite{AB}.
More interestingly, if $(M,d,\mu,L)$ satisfies $(H)$, the above definition is also equivalent to the one in \cite{GH}, namely
\begin{equation}\label{GrigHu}
\lambda_1(\Omega):= \inf\left\{\frac{\mathcal{E}(f)}{\|f
\|_{2}^{2}},\,f\in\mathcal{F}\cap C_c(\Omega)\,\text{and}\,f\neq 0
\right\},
\end{equation}
where  $C_c(\Omega)$ is the space of continuous functions that are compactly supported in $\Omega$ (see Lemma \ref{cnrs} below).
As above, one can replace $\mathcal{F}\cap C_c(\Omega)$ with its closure in $\mathcal{F}$.

The classical Faber-Krahn theorem says that,
for any open set $\Omega$ of $\mathbb{R}^{n}$, we have
\begin{equation*}\label{classical}
\lambda_1(\Omega)\geq c_{n}\mu(\Omega)^{-2/n}.
\end{equation*}
We shall say that $M$ admits the
relative $v$-Faber-Krahn inequality if, for any ball $B(x,r)\subset M$
and any relatively compact open set $\Omega\subset B(x,r)$:
\begin{equation*}\label{fk}
\tag{$FK^v_\alpha$} \lambda_1(\Omega)\geq
\frac{c}{r^{2}}\left(\frac{v_r(x)}{\mu(\Omega)}\right)^{\alpha},
\end{equation*}
where $c$ and $\alpha$  are some positive constants. As usual, we abbreviate $(FK^V_\alpha)$ into $(FK_\alpha)$, $(FK^v)$ means $(FK^v_\alpha)$ for some $\alpha>0$, and  $(FK)$ means $(FK_\alpha)$, that is $(FK^V_\alpha)$, for some $\alpha>0$.
Note that in general, contrary to the case $v=V$, the inclusion $\Omega\subset B(x,r)$ does not  guarantee any more that $\frac{v_r(x)}{\mu(\Omega)}\ge 1$.

 It is known  (see \cite{G2})  that Nash and Faber-Krahn inequalities are equivalent in the setting of Riemannian manifolds.
 Let us make this more precise in the present generality.
 Consider
  the following stronger version of \eqref{LNa},  homogeneous in the sense that it does not display
 the local term $\|f\|_2^2$ in the RHS:
  there exist $\alpha,C>0$  such that for every ball $B=B(x,r)$, for every $f\in \mathcal{F}_c(B)$,
\begin{equation*}\label{LNalo}
\tag{$HLN^v_\alpha$}
\|f \|_{2}^{2(1+\alpha)}\le \frac{Cr^2}{v_r^\alpha(x)}\|f \|_{1}^{2\alpha}\mathcal{E}(f).
\end{equation*}
Conversely, let us introduce a weaker, inhomogeneous, form of the relative $v$-Faber-Krahn inequality,
namely
\begin{equation*}\label{hfk}
\tag{$\widetilde{FK}^v_\alpha$} r^2\lambda_1(\Omega)+1\geq
c\left(\frac{v_r(x)}{\mu(\Omega)}\right)^{\alpha}.
\end{equation*}

The following lemma is probably well-known. We will use it in the proof of Proposition   \ref{LFlo} below.
\begin{lemma}\label{cnrs} Assume that $(M,d,\mu,L)$ satisfies $(H)$. Let $\Omega$ be an open subset of $M$. Then $\mathcal{F}_c(\Omega)\subset\overline{\mathcal{F}\cap C_c(\Omega)}^\mathcal{F}$.
\end{lemma}
\begin{proof}
Let $f\in\mathcal{F}_c(\Omega)$. Choose $\epsilon>0$ so that $f\in\mathcal{F}_c(\Omega_{4\epsilon})$, where
 $$\Omega_{\epsilon}:=\{x\in \Omega,\, d(x,\Omega^c)>\epsilon\}.$$
Select  a finite family of points $x_1, \ldots, x_k$ in $M$ such 
that the balls $B(x_i, \epsilon)$ cover the support of $f$. Consider next the functions 
$\eta_i=\rho_{x_i}^{2\epsilon, \epsilon/2}$ defined in Lemma \ref{su}. Note that   $\eta_i\equiv 1$ on $B(x_i,\epsilon)$ and $\eta_i$ is compactly supported in $B(x_i,2\epsilon)$.
Set 
$$g_0=f \quad \mbox{and} \quad  g_i=\Pi_{j=1}^i (1-\eta_j)f.$$
 Note that $g_k=0$ so if we put $h_i=g_{i-1}-g_i$ then 
$ f=\sum_{i=1}^k h_i$. Note also that $h_i = \eta_ig_{i-1} $ so, by applying several times Lemma \ref{su},
one sees that $h_i  \in\mathcal{F}_c(B(x_i, 2\epsilon))$. Hence to prove Lemma \ref{cnrs} it is enough 
to show that if $h  \in\mathcal{F}_c(B(x, 2\epsilon))$  there exists  
a sequence of continuous functions $\overline{h_n} \in\mathcal{F}_c(B(x, 3\epsilon))$ which 
converges to $h$ in $\mathcal{F}$.

Indeed, since $\mathcal{E}$ is regular, we may approximate $h  \in\mathcal{F}_c(B(x, 2\epsilon))$  by a sequence of continuous functions $\tilde{h}_n\in\mathcal{F}$ such that
$\mathcal{E}(h-\tilde{h}_n)\to 0$ and $\|h-\tilde{h}_n\|_2\to 0$.
Then, again  by Lemma \ref{su},
$$\mathcal{E}(h-\tilde{h}_n\rho_{x}^{3\epsilon, \epsilon/2})=\mathcal{E}((h-\tilde{h}_n)\rho_{x}^{3\epsilon, \epsilon/2})\le  \frac{8}{\epsilon^2}\|h-\tilde{h}_n\|_2^2+2\mathcal{E}(h-\tilde{h}_n).$$
It follows that $\overline{h_n}=\tilde{h}_n\rho_{x}^{3\epsilon, \epsilon/2}$ converges to $h$ in $\mathcal{F}$ as $n\to+\infty$. Moreover, the functions $\overline{h_n}$ are continuous and compactly supported in $B(x, 3\epsilon)$.
\end{proof}

Note that the following statement does not require any doubling assumption on function $v$.
 
 \begin{proposition} \label{LFlo} Assume that $(M,d,\mu,L)$  satisfies $(H)$ and  let $v:M\times \R_+\to \R_+$ satisfy $(A)$. Then 
 $(FK^v_\alpha)$  is equivalent to $(HLN^v_\alpha)$ and $(\widetilde{FK}^v_\alpha)$ is equivalent to $(LN^v_\alpha)$ for all $\alpha>0$.
\end{proposition}

\begin{proof}
Inequality \eqref{LNalo} can be rewritten as
\begin{equation*}\left(\frac{\|f
\|_{2}^2}{\|f
\|_{1}^2}\right)^{\alpha} \le \frac{Cr^2}{v^{\alpha}_r(x)}\frac{\mathcal{E}(f)}{\|f
\|_{2}^2},\quad \forall x\in M,\, r>0,\quad \forall f\in
 \mathcal{F}_c(B(x,r))\setminus\{0\}.
\end{equation*}
Let  $\Omega$ be an open subset of $B(x,r)$. We can restrict the above inequality to  $f \in\mathcal{F}_c(\Omega)\subset  \mathcal{F}_c(B(x,r))$. By H\"older's inequality,  the LHS is larger than
 $\frac{1}{\mu^\alpha(\Omega)}$.   Then taking the infimum in $f \in\mathcal{F}_c(\Omega)$ in the RHS and using definition \eqref{Barlow} of $\lambda_1(\Omega)$   yields
\begin{equation*}\frac{1}{\mu^\alpha(\Omega)} \le \frac{Cr^2}{v^{\alpha}_r(x)} \lambda_1(\Omega),\quad \forall x\in M,\, r>0,\quad \forall\,\Omega\subset B(x,r),
\end{equation*}
that is,  $(FK^v_\alpha)$.  Similarly, $(LN^v_\alpha)$ can be rewritten as
\begin{equation*}\left(\frac{\|f
\|_{2}^2}{\|f
\|_{1}^2}\right)^{\alpha} \le \frac{C}{v^{\alpha}_r(x)}\left(\frac{r^2\mathcal{E}(f)}{\|f
\|_{2}^2}+1\right),\quad \forall x\in M,\, r>0,\quad \forall f\in
 \mathcal{F}_c(B(x,r))\setminus\{0\},
\end{equation*}
and the same argument yields $(\widetilde{FK}^v_\alpha)$.

For the converse, we use a trick introduced by Grigor'yan in \cite{G2}. 
First observe that  by definition \eqref{GrigHu} of $\lambda_1(\Omega)$, $(FK^v_\alpha)$ can be rewritten as
\begin{equation}
\label{FKO}\|g\|_2^2\le C\left(\frac{\mu(\Omega)}{v_r(x)}\right)^{\alpha}r^2\mathcal{E}(g),\end{equation}
$\forall x\in M$, $r>0$, $\Omega\subset B(x,r)$, $g\in
\mathcal{F}\cap C_c(\Omega)$.
 Now, for $f\in L^2(M,\mu)$, $f\ge 0$,
 and $\lambda>0$, write
\begin{eqnarray*}
\|f\|_2^2&=&\int_{f>2\lambda} f^2\,d\mu+\int_{f\le 2\lambda} f^2\,d\mu\\
&\le&4\int_{f>2\lambda} (f-\lambda)^2\,d\mu+2\lambda\int_{f\le 2\lambda} f\,d\mu,
\end{eqnarray*}
hence
\begin{equation}\label{tel}
\|f\|_2^2\le 4\int_M (f-\lambda)_+^2\,d\mu+2\lambda\|f\|_1.
\end{equation}

Let  $x\in M$, $r>0$,  $f\in\mathcal{F}\cap C_c(B(x,r))$. Assume in addition that $f\ge 0$.
Obviously  $\Omega_\lambda=\{f>\lambda\}$ is an open set. Since the semigroup $(e^{-tL})_{t>0}$ is submarkovian,  $(f-\lambda)_+=f-\min{(f,\lambda)}\in \mathcal{F}$.
Now $$\sqrt{\mathcal{E}\left((f-\lambda)_+\right)}\le \sqrt{\mathcal{E}(f)}+\sqrt{\mathcal{E}\left(\min{(f,\lambda)}\right)}$$ and since by the submarkovian property
$\mathcal{E}\left(\min{(f,\lambda)}\right)\le \mathcal{E}(f)$,
 \begin{equation}\label{carlos}
\mathcal{E}\left( (f-\lambda)_+\right)\le 4\mathcal{E} (f).
\end{equation}

Apply now \eqref{FKO} to $\Omega=\Omega_{\lambda/2}$ and $g=(f-\lambda)_+\in \mathcal{F}\cap C_c(\Omega_{\lambda/2})$.
This yields, using Bienaym\'e-Tchebycheff and \eqref{carlos},
\begin{eqnarray*}
\|(f-\lambda)_+\|_2^2&\le& C\left(\frac{\mu(\{f>\lambda/2\})}{v_r(x)}\right)^\alpha r^2\mathcal{E}\left( (f-\lambda)_+\right)\\
&\le&  C'\left(\frac{\|f\|_1}{\lambda v_r(x)} \right)^\alpha r^2\mathcal{E} (f),
\end{eqnarray*}
 therefore, together with \eqref{tel},
 \begin{equation}\label{friday}
\|f\|_2^2\le 4 C'\left(\frac{\|f\|_1}{\lambda v_r(x)} \right)^\alpha r^2\mathcal{E} (f)+2\lambda\|f\|_1.
\end{equation}
The same inequality holds, with a different constant $C''$, for all $f\in\mathcal{F}\cap C_c(B(x,r))$  by applying \eqref{friday} to $f_+$ and $f_-$,
using the fact that $f_+,f_-\in\mathcal{F}\cap C_c(B(x,r))$ because $(e^{-tL})_{t>0}$ is submarkovian,  $$\mathcal{E} (f)=\mathcal{E} (f_+)+\mathcal{E} (f_-)$$ since $\mathcal{E}$ is local, $$\|f\|_2^2=\|f_+\|_2^2+\|f_-\|_2^2$$ and $$\|f_+\|_1,\|f_-\|_1\le  \|f\|_1.$$
Taking $\lambda=\frac{\|f\|_2^2}{4\|f\|_1}$ then yields
  \begin{equation}\label{satu}
\|f\|_2^2\le 2 C''\left(\frac{4\|f\|_1^2}{v_r(x)\|f\|_2^2} \right)^\alpha r^2\mathcal{E} (f),
\end{equation}
for all $x\in M$, $r>0$,  $f\in\mathcal{F}\cap C_c(B(x,r))$. According to Lemma \ref{cnrs}, this is nothing but \eqref{LNalo}.

 If one assumes $(\widetilde{FK}^v_\alpha)$, one starts with
 \begin{equation*}
\|g\|_2^2\le C\left(\frac{\mu(\Omega)}{v_r(x)}\right)^{\alpha}\left(r^2\mathcal{E}(g)+\|g\|_2^2\right),\end{equation*}
and the argument is similar.

\end{proof}
 
 Putting together  Proposition \ref{newmainDG}, which will be proved in Section \ref{NG}, and Proposition \ref{LFlo}, we can establish the link between $(DU\!E^v)$ and $(\widetilde{FK}^v)$. We will see in Section \ref{KLR} under which conditions one can replace 
 $(\widetilde{FK}^v_\alpha)$ with the more classical $(FK^v_\alpha)$. Note that the following statement does apply to doubling compact spaces, in particular to compact Riemannian manifolds, in the case $v=V$. In other terms, considering   $(\widetilde{FK})$ instead of $(FK)$ solves the difficulty raised in
 \cite[comment 5, p.9]{GH}.
 
 \begin{theorem}\label{Nanano} Assume that $(M,d,\mu,L)$  satisfies $(H)$ and that $(M,d,\mu)$ satisfies $(V\!D)$. Let $v:M\times \R_+\to \R_+$  satisfy $(A)$, $(D_v)$, and \eqref{D2}.  Then the upper bound $(DU\!E^v)$ is equivalent to  $(\widetilde{FK}^v)$.
\end{theorem}

Of course, if $(M,d,\mu,L)$  satisfies $(H)$, $\mathcal{E}$ is a Dirichlet form and $(e^{-tL})_{t>0}$ is submarkovian, hence in particular positivity preserving and uniformly bounded on $L^1$. The Davies-Gaffney estimate is known as well for a strongly local and regular Dirichlet form (see \cite[Corollary 1.11]{ST}). This is why we can use Proposition \ref{newmainDG} towards the proof of  Theorem \ref{Nanano}. Moreover, under the assumptions of Theorem \ref{Nanano}, $(\widetilde{FK}^v)$ can be added to the string of equivalences of Theorem \ref{mainDG}.

\subsection{Killing the local term with reverse doubling}\label{KLR}

We will introduce the notion of reverse doubling for a general function $v$. Let us first consider the case   $v=V$. In this case  the notion originates for Riemannian manifolds in \cite[Theorem 1.1]{G1}.  Let $(M,d,\mu)$ be a metric measure space satisfying $(V\!D)$. It is known (see \cite[Proposition 5.2]{GH}),  that, if in addition $M$ is unbounded and connected, one has a  so-called reverse doubling volume property, namely there
exist
 $0<\kappa'<\kappa$ and $c>0$ such that,
for all $r\geq s>0$ and $x,y\in M$ such that $d(x,y)<r+s$,
\begin{equation*}\label{rd}\tag{$RD$}
c\left(\frac{r}{s}\right)^{\kappa'}\leq \frac{V_{r}(y)}{V_{s}(x)}.
\end{equation*}
Together with $(D_\kappa)$, this yields
\begin{equation*}\label{reverse doubling}
\tag{$D_{\kappa,\kappa'}$}c\left(\frac{r}{s}\right)^{\kappa'}\leq \frac{V_{r}(y)}{V_{s}(x)}
\leq C\left(\frac{r}{s}\right)^{\kappa},\quad \forall r\geq s>0,\ d(x,y)<r+s,
\end{equation*}
which can also be written:
\begin{equation}\label{doubling with max}
\frac{V_{r}(y)}{V_{s}(x)} \leq C'w(r,s),\ \forall\,r,s>0,\ x,y\in M\mbox{ such that }d(x,y)< r+s
\end{equation}
where $w(r,s):=\max\{
\left(\frac{r}{s}\right)^{\kappa},\left(\frac{r}{s}\right)^{\kappa'}\}.$

Consider now a measure space $(M,\mu)$ endowed with a function $v:M\times \R_+\to \R_+$  satisfying $(A)$.
We shall say that $(M,\mu,v)$ satisfies weak $(RD_v)$ if there exist
$\kappa',c>0$ such that, for any $x\in M$ and any  $r\geq s>0$,
\begin{equation*}
c\left(\frac{r}{s}\right)^{\kappa'_v}\leq \frac{v_{r}(x)}{v_{s}(x)},
\end{equation*}
and if in addition $(M,\mu)$ is endowed with a metric $d$, that $(M,d,\mu,v)$ satisfies strong $(RD_v)$ or simply  $(RD_v)$ if there exist
$\kappa'_v,c>0$ such that, for any  $r\geq s>0$ and $x,y\in M$ such that $d(x,y)< r+s$,
\begin{equation*}
c\left(\frac{r}{s}\right)^{\kappa'_v}\leq \frac{v_{r}(y)}{v_{s}(x)}.
\end{equation*}
One checks easily that, under $(D_v)$,
 $$(RD_v)\Leftrightarrow \mbox{weak }(RD_v)+(D'_v).$$
As above, the conjunction of  $(D_v)$ and $(RD_v)$  can be rewritten as
\begin{equation}\label{vdoubling with max}
\frac{v_{r}(y)}{v_{s}(x)} \leq C'w_v(r,s),\ \forall\,r,s>0,\ x,y\in M\mbox{ such that }d(x,y)< r+s,
\end{equation}
where $w_v(r,s):=\max\{
\left(\frac{r}{s}\right)^{\kappa_v},\left(\frac{r}{s}\right)^{\kappa'_v}\}.$

\begin{proposition}\label{LF} Assume that $(M,d,\mu,v)$ satisfies  weak $(RD_v)$ and that $v\ge \epsilon V$ for some $\epsilon>0$. Let $L$ be a nonnegative self-adjoint operator on $L^2(M,\mu)$.Then
$(LN_\alpha^v)$  is equivalent to $(HLN_\alpha^v)$ and $(\widetilde{FK}^v_\alpha)$ is equivalent to $(FK^v_\alpha)$.
\end{proposition}

\begin{proof}

It is obvious that $(HLN_\alpha^v)$   implies $(LN_\alpha^v)$ and that $(FK^v_\alpha)$  implies $(\widetilde{FK}^v_\alpha)$. Now for the converses.
Assume $(LN_\alpha^v)$, that is
 \begin{equation}\label{NN}
\|f \|_{2}^{2(\alpha+1)}\le  \frac{C}{v^{\alpha}_r(x)}\|f
\|_{1}^{2\alpha}(r^2\mathcal{E}(f)+\|f \|_{2}^2),
\end{equation}
$\forall\,B=B(x,r),\,
f\in  \mathcal{F}_c(B)$.

Now use a trick from \cite[Proposition 2.3]{CG3}.  Let $A>1$ to be chosen later. Applying \eqref{NN} in the ball $B(x,Ar)$ to $f\in  \mathcal{F}_c(B)\subset \mathcal{F}_c(B(x,Ar))$,  one obtains
\begin{equation}\label{ANN}\|f
\|_{2}^{2(\alpha+1)}\le \frac{C}{v^{\alpha}(x,Ar)}\|f
\|_{1}^{2\alpha}(A^2r^2\mathcal{E}(f)+\|f \|_{2}^2),\quad
\forall\,B=B(x,r),\, f\in  \mathcal{F}_c(B).
\end{equation}
Since $$\|f \|_{1}\le V^{1/2}(x,r)\|f \|_{2}\le C'v^{1/2}(x,r)\|f \|_{2},$$ (the first inequality uses Cauchy-Schwarz inequality and the  second one the assumption that $v\ge \epsilon V$), \eqref{ANN} yields
\begin{equation}\label{NN with Ar}\|f
\|_{2}^{2(\alpha+1)} \le \frac{CA^2r^2}{v^{\alpha}(x,Ar)} \|f
\|_{1}^{2\alpha}\mathcal{E}(f)+
C''\left(\frac{v_r(x)}{v_{Ar}(x)}\right)^{\alpha}\|f
\|_{2}^{2(\alpha+1)}.
\end{equation}
Now, by  weak $(RD_v)$, one has
\begin{equation}\label{syd}
\frac{v_r(x)}{v_{Ar}(x)}\le \frac{1}{cA^{\kappa'_v}}.
\end{equation}
 One can therefore choose $A$
so large  that $C''\left(\frac{v(x,r)}{v(x,Ar)}\right)^\alpha\le 1/2.$ Then \eqref{NN
with Ar} implies
\begin{equation*}\|f
\|_{2}^{2(\alpha+1)} \le \frac{2CA^2r^2}{v_{Ar}^{\alpha}(x)} \|f
\|_{1}^{2\alpha}\mathcal{E}(f),\quad \forall B=B(x,r),\quad \forall f\in
 \mathcal{F}_c(B),
\end{equation*}
that is, using \eqref{syd} once again, $(HLN_\alpha)$.

The statement about the implication from  $(\widetilde{FK}^v_\alpha)$  to $(FK^v_\alpha)$ follows from the one we have just proved through Proposition  \ref{LFlo}  if  $(M,d,\mu,L)$  satisfies $(H)$.

Alternatively, one can rewrite $(\widetilde{FK}^v_\alpha)$ as 
$$c\left(\frac{v_r(x)}{\mu(\Omega)}\right)^{\alpha}\leq
r^2\lambda_1(\Omega)+1,$$
apply it in $B(x,Ar)$, $A\ge 1$, use $(RD_v)$,
obtain
$$cA^{\alpha \kappa'_v}\left(\frac{v_r(x)}{\mu(\Omega)}\right)^{\alpha}\leq c\left(\frac{v_{Ar}(x)}{\mu(\Omega)}\right)^{\alpha}
\leq A^2r^2\lambda_1(\Omega)+1,$$
and choose $A$ so large that $cA^{\alpha \kappa'_v}\ge \frac{2}{\epsilon^\alpha}$. Since $\frac{v_r(x)}{\mu(\Omega)}\ge\epsilon \frac{V_r(x)}{\mu(\Omega)} \ge \epsilon$,
$(FK^v_\alpha)$ follows.
\end{proof}

Note that the role of the local term $\|f\|_2^2$ is different in the case of  $(LS_q^v)$: here it seems one cannot get rid of it except when $v_r(x)\simeq r^n$ and one considers the limit case $q=n$.

 This time, putting together  Theorem \ref{Nanano} and Proposition \ref{LF}, we can establish the link between $(DU\!E^v)$ and $(FK^v)$ under doubling and reverse doubling for $v$. In the case $v=V$ and  $M$ a  Riemannian manifold the following statement gives back Proposition 5.2 from \cite{G2}, see also Theorem 15.21 in \cite{G4}; note the role played by reverse doubling in both instances.

 \begin{theorem}\label{Nanana} Assume that $(M,d,\mu,L)$  satisfies $(H)$ and that $(M,d,\mu)$ satisfies $(V\!D)$. Let $v:M\times \R_+\to \R_+$  satisfy $(A)$, $(D_v)$,    $(RD_v)$, and $v\ge \epsilon V$ for some $\epsilon>0$.  Then the upper bound $(DU\!E^v)$ is equivalent to  $(FK^v)$.
\end{theorem}

In fact, once one has $(RD_v)$,  instead of killing  the local term or non-homogeneous term in $(LN_\alpha^v)$  or $(FK^v_\alpha)$, one may as well avoid to introduce it in the first place.
Let us show how this works by proving directly the following version of Proposition  \ref{KL} (which can also be derived by using Proposition \ref{LFlo}). One could also obtain directly
$(FK^v)$ instead of $(HLN^v)$.

\begin{proposition} \label{KLrd} Let $(M,d,\mu)$  be a metric measure space, $L$ a  non-negative self-adjoint operator on $L^2(M,\mu)$ and  let $v:M\times \R_+\to \R_+$ satisfy $(A)$, $(D_v)$,
   $(RD_v)$, and  $v\ge \epsilon V$ for some $\epsilon>0$. Then
$(K\!N^v)$  implies $(HLN^v_{2/\kappa_v})$,   where $\kappa_v$ is as in \eqref{dv}.
\end{proposition}

\begin{proof} Let $f \in \mathcal{F}_c(B)$, $B=B(x,r)$. By \eqref{vdoubling with
max}, one has
$$v_r(x)\le Cw(r,s)v_s(z),\qquad \forall s>0,\, \forall z\in \text{supp}(f ).$$
Thus $(K\!N^v)$ yields
\begin{equation}\label{LLN}
\int_B|f|^{2}\,d\mu\le
C\left(\frac{w(r,s)}{v_r(x)}\left(\int_B |f|\,d\mu \right)^2
+s^2\mathcal{E}(f)\right),\quad \forall f \in
\mathcal{F}_c(B),
\end{equation}
for all $s,r>0$ and $x\in M$.
Choose  $s_0>0$ such that $$\frac{w(r,s_0)}{v_r(x)}\left(\int_B |f|\,d\mu \right)^2=s_0^2\mathcal{E}(f),$$
which is possible since the function $s\to \frac{w(r,s)}{s^2}$ is continuous, $\lim_{s\to 0+} \frac{w(r,s)}{s^2}=+\infty$ and $\lim_{s\to +\infty} \frac{w(r,s)}{s^2}=0$.

Then \eqref{LLN} yields
$$(i) \quad \int_B|f|^{2}\,d\mu\le 2Cs_0^2\mathcal{E}(f)\quad \text{and} \quad (ii) \quad \int_B|f|^{2}\,d\mu\le \frac{2Cw(r,s_0)}{v_r(x)}\left(\int_B |f|\,d\mu \right)^2.$$
If $r\ge s_0,$ $(ii)$ reads $\int_B|f|^{2}\,d\mu\le \frac{2C r^\kappa}{v_r(x)s_0^\kappa}\left(\int_B |f|\,d\mu \right)^2$,
that is 
$$\frac{c'}{r^2}\left(\frac{v_r(x)\int_B|f|^{2}\,d\mu}{\left(\int_B |f|\,d\mu \right)^2}\right)^{2/\kappa}\le \frac{1}{s_0^2},$$
hence, together with  $(i)$,
\begin{equation}\label{FKK}
\frac{c''}{r^2}\left(\frac{v_r(x)\int_B|f|^{2}\,d\mu}{\left(\int_B |f|\,d\mu \right)^2}\right)^{2/\kappa}\le \frac{\mathcal{E}(f)}{\int_B |f|^2\,d\mu}.
\end{equation}
If $r\le s_0$,  $(ii)$ yields this time
$$\frac{c'}{r^2}\left(\frac{v_r(x)\int_B|f|^{2}\,d\mu}{\left(\int_B |f|\,d\mu \right)^2}\right)^{2/{\kappa}'}\le \frac{1}{s_0^2}.$$ 
Since $v\ge \epsilon V$, H\"older's inequality yields
$$\frac{v_r(x)\int_B|f|^{2}\,d\mu}{\left(\int_B |f|\,d\mu\right)^2}\ge \epsilon.$$
 Since $\kappa'\le \kappa$, $$\frac{c''\epsilon^{2/{\kappa'}-2/\kappa}}{r^2}\left(\frac{v_r(x)\int_B|f|^{2}\,d\mu}{\left(\int_B |f|\,d\mu \right)^2}\right)^{2/{\kappa}}\le \frac{c''}{r^2}\left(\frac{v_r(x)\int_B|f|^{2}\,d\mu}{\left(\int_B |f|\,d\mu \right)^2}\right)^{2/{\kappa}'},$$ hence again \eqref{FKK}, and finally $(HLN^v_{2/\kappa_v})$.
\end{proof}


\bigskip

\section{Converses under Davies-Gaffney estimate}\label{DGE}

Let $(M,d,\mu)$   be a  metric measure space and $L$  a non-negative self-adjoint operator on $L^2(M,\mu)$ with dense domain. We have already introduced the Davies-Gaffney estimate \eqref{DG2}.
In fact, contrary to what we did in  \cite{CS}, we will not use directly \eqref{DG2}  in the present work (except for a technical argument in the end of the proof of Proposition \ref{extrapolation}), but rather an equivalent form, namely the finite propagation speed  for the wave equation.
By the way, it is known that one can also attack the problems treated in \cite{CS}  (the implication from on-diagonal to off-diagonal bounds) by using the latter property instead of the former (see for instance \cite{Sik1}). It would be interesting to know
whether conversely one could use exclusively \eqref{DG2} to prove the results in the present section. This would probably help to get $(DU\!E^v)$ and $(U\!E^v)$ in one go from $(N^v)$ or $(GN_q^v)$, instead of going through two steps: first the present article, then \cite{CS}.

Following  \cite{CS},  we say that $(M,d,\mu,L)$, or in short $L$, satisfies the finite propagation  speed property for
solutions of the corresponding wave equation if
\begin{equation}  \label{fs11}
\langle \cos(r\sqrt L) f_1, f_2 \rangle = 0
\end{equation}
for all $f_i\in L^2(B_i,\mu)$, $i=1,2$, where $B_i $ are  open balls in $M$ such that  $d(B_1,B_2)>r>0$.
Here and in the sequel, if $\Omega$ is a measurable subset of $M$,  $L^2(\Omega,\mu)$ will mean $L^2(\Omega,\restr{\mu}{\Omega})$.

We shall use the following notational convention.  For $r>0$, set
\begin{equation*}
D_r=\{ (x,\, y)\in M\times M: {d}(x,\, y) \le r \}.
\end{equation*}
Denote by $L^p_c(M,\mu)$, $p\ge 1$, the vector space of functions in $L^p(M,\mu)$ with support in a ball.
Given a linear operator $T$ from $L^p_c(M,\mu)$  to $L^q_{loc}(M,\mu)$, for some $1\le p,q\le +\infty$,   write
\begin{equation}\label{e2.1}
\text{supp}\,T  \subseteq D_r
\end{equation}
 if $\langle T f_1, f_2 \rangle = 0$ whenever  $f_1\in L^p(B_1,\mu)$, $f_2\in L^{q'}(B_2,\mu)$, and $B_1,B_2$ are balls such that
 $d(B_1,B_2)> r$.   Note that  if $T$ is an integral operator with   kernel
$K_T$, then (\ref{e2.1}) coincides with the  standard meaning of
$\text{supp} \, K_{T}  \subseteq D_r$,
 that is $K_T(x, \, y)=0$ for all $(x, \, y) \notin D_r$. Now we can state the finite  propagation speed
 property \eqref{fs11} in the following way
 $$
 \text{supp} \, \cos(r\sqrt L) \subseteq D_r,\quad\forall\,r>0.
 $$

The  proof of the above-mentioned equivalence can be found in \cite{Si1} or \cite{CS}.

\begin{theorem}
\label{fspro} Let $(M,d,\mu)$ be a metric measure space and $L$ be a non-negative self-adjoint operator acting on $%
L^2(M,\mu)$. Then the finite  propagation speed property  \eqref{fs11}   and
the Davies-Gaffney estimate \eqref{DG2}   are equivalent properties for $(M,d,\mu,L)$. 
\end{theorem}

\subsection{From   Gagliardo-Nirenberg to  heat kernel upper bounds}\label{DG}

Our starting point will be the fact that, on a doubling metric space, support properties like  \eqref{e2.1} enable one to commute 
the operator with multiplication by doubling weights in $L^p-L^q$ estimates.

\begin{proposition}\label{fala1} Let $(M,d,\mu)$ be a doubling metric measure space and let $v:M\times \R_+\to \R_+$  satisfy $(A)$, $(D_v)$, and \eqref{D2}.   For all  $\gamma \in \R$, there exists $C_\gamma>0$ only depending on the constants in $(V\!D)$,  $(D_v)$, and \eqref{D2} such that,   for all  $p,q$, $1\le p\le q\le \infty$, and for every family of operators  $T_r$, $r>0$, mapping continuously
$\lp$ to  $\lqq$ and  satisfying \begin{equation}\label{dgr}
\text{\em supp} \,T_r  \subseteq D_r,
\end{equation}
one has
$$
\| {v_{{r}}^\gamma}  T_r {v_{{r}}^{-\gamma}} \|_{p \to q}\le
C_\gamma \|   T_r  \|_{p \to q},$$
uniformly in $r>0$.
\end{proposition}
Proposition \ref{fala1} relies on
ideas that already appeared in  \cite{CS, GHS}. If  $\Omega$ is a subset of $M$, denote by $\chi_\Omega$ its characteristic function. We will deduce Proposition \ref{fala1} from the following statement which does not involve $v$.

\begin{lemma}\label{fala2} Let $(M,d,\mu)$ be a doubling metric measure space.   There exists $C>0$ only depending on the doubling constant  such that,   for all  $p,q$, $1\le p\le q\le \infty$, and for every family of operators  $T_r$, $r>0$ from $L^p_c(M,\mu)$  to $L^q(M,\mu)$  satisfying \eqref{dgr},
one has
$$
\|T_r\|_{p \to q} \le C \sup_{x\in M} \| T_r\chi_{B(x,  r) }\|_{p\to q},$$
uniformly in $r>0$.
\end{lemma}
Note that the reverse inequality is obvious.
Let us check that Proposition \ref{fala1}  follows from Lemma \ref{fala2}.
Indeed,  the operator $S_r:={v_{{r}}^\gamma}  T_r {v_{{r}}^{-\gamma}}$ clearly also satisfies  \eqref{dgr}. 
Lemma \ref{fala2} applied to $S_r$ thus yields
$$
\| {v_{{r}}^\gamma}  T_r {v_{{r}}^{-\gamma}} \|_{p \to q}\le
C \sup_{x\in M} \|  {v_{{r}}^\gamma}  T_r {v_{{r}}^{-\gamma}} \chi_{B(x,  r) }  \|_{p \to q},$$
uniformly in $r>0$.
Now, for $f\in L^p(M,\mu)$, 
\begin{eqnarray*}
\|  {v_{{r}}^\gamma}  T_r {v_{{r}}^{-\gamma}} \chi_{B(x,  r) } f \|_{q}&\le& \|  {v_{{r}}^\gamma}  T_r  \chi_{B(x,  r) }  \|_{p \to q}\|   {v_{{r}}^{-\gamma}} \chi_{B(x,  r) } f \|_p\\
&\le& \|  {v_{{r}}^\gamma}  T_r  \chi_{B(x,  r) }  \|_{p \to q}\|   {v_{{r}}^{-\gamma}} \chi_{B(x,  r) }\|_\infty \| f \|_p.
\end{eqnarray*}
Since $v$ satisfies \eqref{D2}, the function ${v_{{r}}^{-\gamma}} \chi_{B(x,  r)}$ is dominated by
$C^{|\gamma|}v_r^{-\gamma}(x)\chi_{B(x,  r)}$, where $C$ is the constant in \eqref{D2}, hence
$$\|  {v_{{r}}^\gamma}  T_r {v_{{r}}^{-\gamma}} \chi_{B(x,  r) } f \|_{q}\le C^{|\gamma|}v_r^{-\gamma}(x)\|  {v_{{r}}^\gamma}  T_r  \chi_{B(x,  r) }  \|_{p \to q} \| f \|_p$$
Now, since $T_r$ satisfies \eqref{dgr},
$$T_r\chi_{B(x,  r) } =\chi_{B(x,  2r) } T_r\chi_{B(x,  r) },$$
and
\begin{eqnarray*}
v_r^{-\gamma}(x)\|  {v_{{r}}^\gamma}  T_r  \chi_{B(x,  r) }  \|_{p \to q}&=&v_r^{-\gamma}(x)\|  {v_{{r}}^\gamma}  \chi_{B(x,  2r) }T_r  \chi_{B(x,  r) }  \|_{p \to q}\\
&\le&v_r^{-\gamma}(x)\|  {v_{{r}}^\gamma} \chi_{B(x,  2r) } \|_\infty\|T_r  \chi_{B(x,  r) }  \|_{p \to q}\\
&\le&C'\|T_r  \chi_{B(x,  r) }  \|_{p \to q},
\end{eqnarray*}
where $C'_\gamma$ only depends on the constants in $(D_v)$,  \eqref{D2} and on $\gamma$.
Finally
$$\|{v_{{r}}^\gamma}  T_r {v_{{r}}^{-\gamma}}\chi_{B(x, r) } \|_{p \to q}\le  C^{|\gamma|}C'_\gamma\|  T_r \chi_{B(x, r) } \|_{p \to q}$$
and$$
\| {v_{{r}}^\gamma}  T_r {v_{{r}}^{-\gamma}} \|_{p \to q}\le  CC'_\gamma C^{|\gamma|} \sup_{x\in M} \|  T_r \chi_{B(x, r) } \|_{p \to q}\le CC'_\gamma C^{|\gamma|}\|  T_r  \|_{p \to q},$$
which proves Proposition \ref{fala1}. 

\begin{proof}[{ Proof of Lemma $\ref{fala2}$}]
Fix $r>0$. Apply $(BC\!P)$:  there exists a sequence $x_i \in M$ such that
$d(x_i,x_j)> {r/2}$ for $i\neq j$ and $\sup_{x\in M}\inf_i d(x,x_i)
\le {r/2}$. 
Define $\widetilde{B_i}$ by the formula
\begin{equation}\label{e3.2}
\widetilde{B_i}=B\left(x_i,r\right)\setminus
\left(\bigcup_{j<i}B\left(x_j,r\right)\right),
\end{equation}
 so that $\left(\widetilde{B_i}\right)_{i}$ is a denumerable partition of $M$.
For $f\in\lp$,  write
$$
{ T_r }f =\sum_{i,j} \chi_{\widetilde B_i}{  T_r}\chi_{\widetilde
B_j}f.
$$
Now, if ${d}(x_i,x_j)>3r$, $d(\widetilde B_i, \widetilde B_j)> r$, hence, using \eqref{dgr},
$$\sum_{i,j:\, {d}(x_i,x_j)>3r} \chi_{\widetilde B_i}{  T_r \chi_{\widetilde
B_j}}f=0,
$$
thus
$$
{  T_r }f =\sum_{i,j:\, {d}(x_i,x_j)\le 3r} \chi_{\widetilde B_i} T_r\chi_{\widetilde
B_j}f.
$$
Assume $q<+\infty$. Obvious modifications yield the case $q=+\infty$ . Write
\begin{eqnarray*}
\| T_r  f\|_{q}^q&=&\|\sum_{i,j:\, {d}(x_i,x_j)\le 3r} \chi_{\widetilde B_i}T_r \chi_{\widetilde B_j}f\|_{q}^q \\
&=&\sum_i \|\sum_{j:\,{d}(x_i,x_j)\le
3r} \chi_{\widetilde B_i}   T_r  \chi_{\widetilde B_j}f\|_{q}^q,
\end{eqnarray*}
the last equality using the fact that the ${\widetilde B_i}$'s are disjoint. According to $(BC\!P)$, $$K_0=\sup_i\#\{j:\;d(x_i,x_j)\le  3r\}$$
 is a finite integer which only depends  on the doubling constant of $(M,d,\mu)$.
It follows that 
\begin{eqnarray*}  
&&\|\sum_{j:\,{d}(x_i,x_j)\le
3r} \chi_{\widetilde B_i}    T_r  \chi_{\widetilde B_j}f\|_{q}^q
\\
&&\left(\sum_{j:\,{d}(x_i,x_j)\le
3r}\| \chi_{\widetilde B_i}    T_r  \chi_{\widetilde B_j}f\|_{q}\right)^q
\\
&\le& K_0^{q-1}    \sum_{j:\,{d}(x_i,x_j)\le 3r}
\left\| \chi_{\widetilde
B_i}  T_r \chi_{\widetilde B_j}f\right\|_{q}^q.
\end{eqnarray*}  
The last line uses convexity together with the fact that there are at most  $K_0$ terms in the summation.

Gathering the above inequalities,
one obtains
\begin{eqnarray*}
\|  T_r  f\|_{q}^q&\le&K_0^{q-1}  \sum_i   \sum_{j:\,{d}(x_i,x_j)\le 3r} \left\| \chi_{\widetilde
B_i}T_r\chi_{\widetilde B_j}\right\|_{p\to q}^q\left\|\chi_{\widetilde B_j}f\right\|_p^q\\
&\le&K_0^{q-1}  \sum_i   \sum_{j:\,{d}(x_i,x_j)\le 3r} \left\| T_r\chi_{\widetilde B_j}\right\|_{p\to q}^q\left\|\chi_{\widetilde B_j}f\right\|_p^q\\
&=&K_0^{q-1}  \sum_j   \sum_{i:\,{d}(x_i,x_j)\le 3r} \left\| T_r\chi_{\widetilde B_j}\right\|_{p\to q}^q\left\|\chi_{\widetilde B_j}f\right\|_p^q\\
&\le&K_0^{q-1}  \sum_j   \#\{i:\,{d}(x_i,x_j)\le 3r\} \left\| T_r\chi_{\widetilde B_j}\right\|_{p\to q}^q\left\|\chi_{\widetilde B_j}f\right\|_p^q\\
&\le&K_0^q  \sum_j \left\| T_r\chi_{\widetilde B_j}\right\|_{p\to q}^q\left\|\chi_{\widetilde B_j}f\right\|_p^q
\\
&\le&K_0^q \sup_\ell\left\| T_r\chi_{\widetilde B_\ell}\right\|_{p\to q}^q \sum_j \left\|\chi_{\widetilde B_j}f\right\|_p^q
\\&\le&  K_0^{q}   \sup_\ell\left\| T_r\chi_{\widetilde B_\ell}\right\|_{p\to q}^q  \left(\sum_j\|\chi_{\widetilde
B_j}f\|_{p}^p\right)^{\frac{q}{p}}\\
&=&K_0^{q}   \sup_\ell\left\| T_r\chi_{\widetilde B_\ell}\right\|_{p\to q}^q  \|f\|_{p}^q.
\end{eqnarray*}
The one before last inequality uses the fact that $p\le q$. 
Finally,
$$\|  T_r  \|_{p\to q}\le K_0  \sup_\ell\left\| T_r\chi_{\widetilde B_\ell}\right\|_{p\to q},$$
hence the claim. 
\end{proof}

In the sequel we will use the following straightforward observation in order to build functions of $L$ that satisfy support properties of the type \eqref{e2.1}. Next our task will be to relate these operators to $(e^{-tL})_{t>0}$.

  \begin{lemma}\label{nos} Assume that 
 $(M,d,\mu,L)$ satisfies 
 \eqref{DG2}.   Let $\Phi\in L^1(\mathbb{R})$ be  even and such that 
  $$\mbox{\em supp}\,\widehat\Phi \subset [-1,1].$$ Then
  \begin{equation}
 \mbox{\em supp} \, {\Phi(r\sqrt L)} \subseteq D_r
\label{noseq}
\end{equation}
 for all $r>0$. 
  \end{lemma}
  \begin{proof}
The claim follows from Theorem \ref{fspro}, \eqref{fs11} and the formula
\begin{equation}\label{formu}
\Phi(r\sqrt L) =\frac{1}{2\pi}\int_{-\infty}^{+\infty}%
  \widehat{\Phi}(s) \cos(rs\sqrt L) \;ds.
\end{equation}
\end{proof}

\begin{lemma}\label{pocz3a} Let $(M,d,\mu,v)$ be as in Proposition  $\ref{fala1}$.  Assume that 
 $(M,d,\mu,L)$ satisfies 
 \eqref{DG2}.  
Let  $1 \le p\le 2 $, $\gamma, \delta \in \R$.   Suppose that the function $\Phi$ on $\R$ is even and satisfies $\mbox{ supp}\, \widehat\Phi \subset [-1,1] $ as well as
$$\sup_\lambda |(1+\lambda^2)^{N+1} \Phi(\lambda)| <  \infty,$$ for some $N\in\N$.
Then
$$
\|{v_{{\sqrt{t}}}^\gamma}\Phi(\sqrt {tL}) {v_{{\sqrt{t}}}^\delta} \|_{p \to 2} \le
C\| {v_{{\sqrt{t}}}^\gamma}  (I+t L)^{-N} {v_{{\sqrt{t}}}^\delta} \|_{p \to 2}
$$
uniformly in  $t>0$. 
\end{lemma}
\begin{proof}
Set $\Psi(\lambda)=(1+\lambda^2)^N \Phi(\lambda).$
Function $\Psi$ is even, bounded, belongs to $L^1(\R)$, and satisfies $\mbox{ supp}\,\widehat\Psi \subset [-1,1].$
By Lemma   \ref{nos}, $$
 \mbox{ supp} \, {\Psi(\sqrt {tL})} \subseteq D_{\sqrt{t}}.
$$
Thus by Proposition \ref{fala1}
$$\|{v_{{\sqrt{t}}}^{\gamma}} \Psi(\sqrt{t L}) {v_{{\sqrt{t}}}^{-\gamma}} \|_{2 \to 2}\le C_\gamma\| \Psi(\sqrt{t L}) \|_{2 \to 2}
$$
and by spectral theory
$$\|\Psi(\sqrt{t L}) \|_{2 \to 2} \le C'_\gamma, \ \forall\,t>0.
$$
Hence
\begin{eqnarray*}
&&\|{v_{{\sqrt{t}}}^\gamma}\Phi(\sqrt {tL}) {v_{{\sqrt{t}}}^\delta} \|_{p \to 2}\\\
&\le&\|{v_{{\sqrt{t}}}^{\gamma}} \Psi(t\sqrt L) {v_{{\sqrt{t}}}^{-\gamma}} \|_{2 \to 2}
\| {v_{{\sqrt{t}}}^\gamma}  (I+t L)^{-N} {v_{{\sqrt{t}}}^\delta} \|_{p \to 2}\\
&\le& C_\gamma C'_\gamma \| {v_{{\sqrt{t}}}^\gamma}  (I+tL)^{-N} {v_{{\sqrt{t}}}^\delta} \|_{p \to 2}.
\end{eqnarray*}
\end{proof}

  \begin{lemma}\label{pocz2a} Let $(M,\mu,L)$ be a   measure space endowed with a non-negative self-adjoint operator and let $v:M\times \R_+\to \R_+$  satisfy $(A)$ and $(D_v)$. Let $1 \le p\le q\le \infty $,  $\gamma, \delta \in \R $, and  $N> \kappa_v(|\delta|+|\gamma|)/2$, where $\kappa_v$ is the exponent in \eqref{dv}.
  Then  there exists $C>0$ such that
  $$
 \sup_{t>0} \| {v_{{\sqrt{t}}}^\gamma}  (I+t L)^{-N} {v_{{\sqrt{t}}}^\delta} \|_{p \to q}
 \le C \sup_{t>0}
 \| {v_{{\sqrt{t}}}^\gamma} e^{-tL}\,{v_{{\sqrt{t}}}^\delta} \|_{p \to q}.
  $$
\end{lemma}

\begin{proof}
Recall the following  standard   integral representation
$$
{(I+tL)^{-N}}
= \frac{1}{\Gamma(N)}\int_0^\infty e^{-s} \, s^{N-1}\exp(-s(tL))
\,ds.
$$
It yields
\begin{equation*}
\| {v_{{\sqrt{t}}}^\gamma}  (I+t L)^{-N} {v_{{\sqrt{t}}}^\delta} \|_{p \to q} \le
 C\int_0^\infty e^{-s} \, s^{N-1}
\|  {v_{{\sqrt{t}}}^\gamma}\exp(-s(tL)){v_{{\sqrt{t}}}^\delta} \|_{p \to q}  \,ds,
\end{equation*}
hence,  using $(D_v)$,
\begin{eqnarray*}
&&\| {v_{{\sqrt{t}}}^\gamma}  (I+t L)^{-N} {v_{{\sqrt{t}}}^\delta} \|_{p \to q}\\
&\le&  C\int_0^\infty e^{-s} \, s^{N-1}
 \left(\sqrt{s}+\frac{1}{\sqrt{s}}\right)^{\kappa_v(|\delta|+|\gamma|)}\|  {v_{{\sqrt{st}}}^\gamma} \exp(-s(tL)){v_{\sqrt{st}}^\delta} \|_{p \to q}  \,ds
\\ &\le& C\left(\int_0^\infty e^{-s} \, s^{N-1}
 \left(\sqrt{s}+\frac{1}{\sqrt{s}}\right)^{\kappa_v(|\delta|+|\gamma|)}  \,ds\right)\sup_{t>0}\|  {v_{{\sqrt{t}}}^\gamma}\exp(-tL){v_{{\sqrt{t}}}^\delta} \|_{p \to q},
\end{eqnarray*}
which proves the claim.
\end{proof}

\begin{proposition}\label{cieplo1} Let $(M,d,\mu,v)$ be as in Proposition  $\ref{fala1}$.  Assume in addition that 
 $(M,d,\mu,L)$ satisfies 
 \eqref{DG2}.  
Let $1\le p \le 2$.
  Then for any $\gamma_1, \gamma_2,\delta_1,\delta_2\in \R $ such that   $\gamma_1+ \delta_1=\gamma_2+\delta_2=1/p-1/2$,  there exists a constant $C$ such that
  $$
\sup_t \|{v_{{\sqrt{t}}}^{\gamma_1}}   e^{-tL}\,  {v_{{\sqrt{t}}}^{\delta_1}} \|_{p \to 2}
 \le
C \sup_t \| {v_{{\sqrt{t}}}^{\gamma_2}}   e^{-tL}\,  {v_{{\sqrt{t}}}^{\delta_2}}\|_{p \to 2}.
$$
As a consequence,  for fixed $p$, $1\le p<2$, all the conditions  $(vEv_{p,2,\gamma})$  for  $\gamma \in \R$
are equivalent.
\end{proposition}

\begin{proof}
A direct calculation shows that for all $a>0$, $x\in \R$,
$$
\frac{1}{\Gamma(a)}\int_0^\infty(s-x^2)^a_+e^{-s}\,ds=e^{-x^2},
$$
where 
$$
(t)_+=t \quad \mbox{if} \quad t \ge 0 \quad \quad \mbox{and}\quad  (t)_+=0 \quad  \mbox{if} \quad 
t<0.
$$
Hence
$$
C_a\int_0^\infty\left(1-\frac{x^2}{s}\right)^a_+e^{-s/4}s^a \,ds=e^{-x^2/4},
$$
for some suitable $C_a>0$.
Taking the Fourier transform of both sides of the above inequality yields
$$
\int_0^\infty F_a(\sqrt s\lambda)  s^{a+\frac{1}{2}}e^{-s/4}ds=e^{-\lambda^2},
$$
where $F_a$ is the Fourier transform of the function $t \to (1-{t^2})^a_+$ multiplied
by the appropriate constant. Hence, by spectral theory,
$$
\int_0^\infty F_a(\sqrt{stL})  s^{a+\frac{1}{2}}e^{-s/4}ds=e^{-tL}
$$
(this is a version of the well-known transmutation formula).
Write now 
\begin{eqnarray*}
 &&\|{v_{{\sqrt{t}}}^{\gamma_1}}   e^{-tL}\,  {v_{{\sqrt{t}}}^{\delta_1}} \|_{p \to 2}
 \le \int_0^\infty \|{v_{{\sqrt{t}}}^{\gamma_1}}   F_a(\sqrt{tsL})    {v_{{\sqrt{t}}}^{\delta_1}} \|_{p \to 2} s^{a+\frac{1}{2}}e^{-s/4}ds\\
 &\le& C\int_0^\infty \|{v_{{\sqrt{st}}}^{\gamma_1}}   F_a(\sqrt{tsL})    {v_{{\sqrt{st}}}^{\delta_1}} \|_{p \to 2} \Big(\sqrt{s}+\frac{1}{\sqrt{s}}\Big)^{\kappa_v(|\delta_1|+|\gamma_1|)}s^{a+\frac{1}{2}}e^{-s/4}ds,
  \end{eqnarray*}
  hence, for $a$ large enough,
  $$
 \sup_{t>0} \|{v_{{\sqrt{t}}}^{\gamma_1}}   e^{-tL}\,  {v_{{\sqrt{t}}}^{\delta_1}} \|_{p \to 2}
 \le C'\sup_{t>0}\|{v_{{\sqrt{t}}}^{\gamma_1}}   F_a(\sqrt{tL})    {v_{{\sqrt{t}}}^{\delta_1}} \|_{p \to 2}.$$
On the other hand,  $\Phi=F_a$ satisfies the assumptions of Lemma \ref{nos}, thus
\begin{equation}
 \mbox{
  supp} \, {F_a(r\sqrt {L})} \subseteq D_r,\ \forall
 \,r>0.
 \end{equation}
 Setting $T_r={v_{r}^{\gamma_2}}   F_a(r\sqrt{L})    {v_{r}^{\delta_2}} $,
 it follows  that
 \begin{equation}
 \mbox{ supp} \, T_{r}\subseteq D_{r},\ \forall
 \,r>0.
 \end{equation}
 By Proposition \ref{fala1} with $\gamma=\gamma_1-\gamma_2=\delta_2-\delta_1$,
 $$
  \sup_{t>0}\|{v_{{\sqrt{t}}}^{\gamma_1}}   F_a(\sqrt{tL})    {v_{{\sqrt{t}}}^{\delta_1}} \|_{p \to 2}
 \le C\sup_{t>0} \|{v_{{\sqrt{t}}}^{\gamma_2}}   F_a(\sqrt{tL})    {v_{{\sqrt{t}}}^{\delta_2}} \|_{p \to 2}.
 $$
We can now apply Lemma \ref{pocz3a} with $\Phi=F_a$. One checks easily that the assumptions are  satisfied as soon as $2N+1\le a$, in which case
$$\sup_{t>0} \|{v_{{\sqrt{t}}}^{\gamma_2}}   F_a(\sqrt{tL})    {v_{{\sqrt{t}}}^{\delta_2}} \|_{p \to 2}\le
 C'\sup_{t>0} \|  {v_{{\sqrt{t}}}^{\gamma_2}}  (I+t L)^{-N}{v_{{\sqrt{t}}}^{\delta_2}} \|_{p \to 2}.
 $$
 By Lemma  \ref{pocz2a}, since $a$ hence $N$ can be chosen arbitrarily large,
 $$
      \sup_{t>0}\|  {v_{{\sqrt{t}}}^{\gamma_2}} (I+t L)^{-N}{v_{{\sqrt{t}}}^{\delta_2}} \|_{p \to 2}
 \le C'' \sup_{t>0}\|  {v_{{\sqrt{t}}}^{\gamma_2}} e^{-tL}\,{v_{{\sqrt{t}}}^{\delta_2}} \|_{p \to 2}.
$$
This finishes the proof of the proposition.
\end{proof}


\begin{rem} In  \cite[Proposition 2.1]{BK} and  \cite[Proposition 2.16,  p.40, see also Remark  2.17,  p.42]{K}, it is proved that a commutation phenomenon similar to the one in Proposition $\ref{cieplo1}$ holds in presence of generalised Gaussian estimates.  In \cite[Theorem 4.15 and Remarks $a)$ and $b)$]{CS}, it is shown that such estimates follow from $\eqref{DG2}$ and $(Ev_{p,2})$ estimates. This provides an alternative approach to Proposition $\ref{cieplo1}$, at least in the case $\delta_2=0$.
\end{rem} 

A first consequence of Proposition \ref{cieplo1} is the following result, which is in some sense dual to Corollary \ref{55}. 

\begin{proposition}\label{extrap} Let $(M,d,\mu,v)$ be as in Proposition $\ref{fala1}$.   Assume in addition
 that 
 $(M,d,\mu,L)$ satisfies 
 \eqref{DG2} and  that $(e^{-tL})_{t>0}$  is bounded analytic on $L^{1}(M,\mu)$. Then $(DU\!E^v)$ is equivalent to
\begin{equation*}\label{bost}
\tag{$GN^v_{1,2,\beta}$}\|f\sqrt{v_{r}} \|_{2}\leq
C_\beta(\|f\|_{1}+r^\beta\|L^{\beta/2}f\|_{1}), \quad \forall\, r>0,\ \forall\,f\in\mathcal{D}_1(L^{\beta/2}),
\end{equation*}
for all (or some) $\beta>\kappa_v/2$.
\end{proposition}

\begin{proof} That  $(DU\!E^v)$ implies \eqref{bost}  is a particular case of Theorem \ref{GN23}.
Conversely, substituting  $e^{-r^2L}f$ to $f$  in \eqref{bost} yields $(vE_{1,2})$, hence $(Ev_{1,2})$ by Proposition \ref{cieplo1}, hence   $(DU\!E^v)$ by Corollary \ref{implication}.
\end{proof}

\medskip

Now we can extend  the  extrapolation lemma from \cite[Lemma 1]{Cou1} or \cite[Lemma 1.3] {CM} to our setting. 

\begin{proposition}\label{extrapolation} Let $(M,d,\mu,v)$ be as in Proposition  $\ref{fala1}$.   Assume in addition that 
 $(M,d,\mu,L)$ satisfies 
 \eqref{DG2} and  that $(e^{-tL})_{t>0}$ is  uniformly bounded on $L^{{p_0}}(M,\mu)$ for some $1\le p_0<2$.  Then
$(Ev_{p_1,2})$ for some $p_1$,  $p_0\le p_1<2$,
implies
  $(vEv_{{{p}}, 2, \gamma})$,
for  all $\gamma \in \R$, and in particular $(Ev_{{{p}}, 2})$, for all ${p}$, $p_0\le {p}\le 2$.  
\end{proposition}
\begin{proof} Observe first that by Proposition \ref{interpolation}, $(Ev_{p_1,2})$ implies  $(Ev_{p,2})$ for $p_1\le p<2$.  Then under our assumptions Proposition \ref{cieplo1} yields  $(vEv_{{{p}}, 2, \gamma})$,
for  all $\gamma \in \R$ and $p_1\le p<2$. It remains to treat the case where $p_0\le p<p_1$. 
Now by interpolation $(e^{-tL})_{t>0}$ is also  uniformly bounded on $L^{{{p}}}(M,\mu)$ if $p_0\le {p}<p_1<2$.  We can therefore assume without loss of generality that ${p}=p_0$ and assume $(Ev_{p_1,2})$.

Fix $x\in M$ and $r>0$, and let    $f \in L^2(B(x,r),\mu) \cap L^{{p_0}}(B(x,r),\mu)$.  Then
$$
\|{v_{\sqrt{t}}^{\frac{1}{{p_0}}-\frac{1}{{2}}}}e^{-tL} f\|_{2}
\le \|{v_{\sqrt{t}}^{\frac{1}{{p_0}}-\frac{1}{{2}}}}e^{-(t/2)L}\, {v_{\sqrt{t}}^{\frac{1}{p_1}-\frac{1}{{p_0}}}}\|_{p_1\to {2}}
 \|   {v_{\sqrt{t}}^{\frac{1}{{p_0}}-\frac{1}{p_1}}}  e^{-(t/2)L} f\|_{p_1}.
$$
   
Proposition \ref{cieplo1}  yields  in particular $(vEv_{p_1,2,\gamma})$ with $\gamma=\frac{1}{{p_0}}-\frac{1}{{2}}$, hence, by $(D_v)$,
$$\sup_{t>0} \|{v_{\sqrt{t}}^{\frac{1}{{p_0}}-\frac{1}{{2}}}}e^{-tL}\, {v_{\sqrt{t}}^{\frac{1}{p_1}-\frac{1}{{p_0}}}} \|_{p_1 \to {2}} \le C,$$
therefore
\begin{equation}\label{ineq}
\|{v_{\sqrt{t}}^{\frac{1}{{p_0}}-\frac{1}{{2}}}}e^{-tL} f\|_{2}
\le C
 \|   {v_{\sqrt{t}}^{\frac{1}{{p_0}}-\frac{1}{p_1}}}  e^{-(t/2)L} f\|_{p_1}.
\end{equation}
Next, we choose $\theta$ such that
$$
\frac{1}{p_1}= \frac{\theta}{{p_0}}+\frac{1-\theta}{{2}}.
$$
H\"older's inequality  yields
$$
\|   v_{\sqrt{t}}^{\frac{1}{{p_0}}-\frac{1}{p_1}}  g\|_{p_1} \leq  \|g\|_{{p_0}}
^\theta
 \|   v_{\sqrt{t}}^{\frac{1}{{p_0}}-\frac{1}{{2}}} g\|_{2}^{1-\theta}.
$$
Taking $g=e^{-(t/2)L}f$, we obtain
$$\|{v_{\sqrt{t}}^{\frac{1}{{p_0}}-\frac{1}{p_1}}}e^{-(t/2)L}f\|_{p_1}\leq  
\|e^{-(t/2)L}f\|_{{p_0}}^\theta
 \|   {v_{\sqrt{t}}^{\frac{1}{{p_0}}-\frac{1}{{2}}}}e^{-(t/2)L}f\|_{2}^{1-\theta},$$
 therefore, since $(e^{-tL})_{t>0}$  is  uniformly bounded on $L^{{p_0}}(M,\mu)$,
 \begin{equation}\label{toto}\|{v_{\sqrt{t}}^{\frac{1}{{p_0}}-\frac{1}{p_1}}}e^{-(t/2)L}f\|_{p_1}\leq  C'
\|f\|_{{p_0}}^\theta
 \|{v_{\sqrt{t}}^{\frac{1}{{p_0}}-\frac{1}{{2}}}}e^{-(t/2)L}f\|_{2}^{1-\theta}.
 \end{equation} 
It follows from $(D_v)$ that
$$ \|{v_{\sqrt{t}}^{\frac{1}{{p_0}}-\frac{1}{{2}}}}e^{-(t/2)L}f\|_{2}\le C  \|{v_{\sqrt{t/2}}^{\frac{1}{{p_0}}-\frac{1}{{2}}}}e^{-(t/2)L}f\|_{2}.$$
Thus \eqref{ineq} and \eqref{toto} yield
\begin{equation}\label{Ktheta}
\|{v_{\sqrt{t}}^{\frac{1}{{p_0}}-\frac{1}{{2}}}}e^{-tL} f\|_{2}
\le C
\|f\|_{{p_0}}^\theta
 \|{v_{\sqrt{t/2}}^{\frac{1}{{p_0}}-\frac{1}{{2}}}}e^{-(t/2)L}f\|_{2}^{1-\theta}.
\end{equation}
For $T>0$, define
$$K(f,T):=\sup_{0< t\le T}\|{v_{\sqrt{t}}^{\frac{1}{{p_0}}-\frac{1}{{2}}}}e^{-tL} f\|_{2},$$
which is a finite quantity.
Indeed, write
\begin{eqnarray*}
K^2(f,T)&=&\sup_{0\le t\le T}\|{v_{\sqrt{t}}^{\frac{1}{{p_0}}-\frac{1}{{2}}}}e^{-tL} f\|_{2}^2
 \\&=&\sup_{0\le t\le T} \sum_{k=0}^\infty\|{v_{\sqrt{t}}^{\frac{1}{{p_0}}-\frac{1}{{2}}}}e^{-tL} f\|^2_{L^2(B(x,(k+1)r) \setminus B(x,kr),\mu)}\\ &\le& 
 \sup_{0\le t\le T} \sum_{k=0}^\infty\|{v_{\sqrt{T}}^{\frac{1}{{p_0}}-\frac{1}{{2}}}}e^{-tL} f\|^2_{L^2(B(x,(k+1)r) \setminus B(x,kr),\mu)}.
 \end{eqnarray*}
 Now note that
 \begin{equation}\label{vv}
v(y,r) \le C\left(1+\frac{d(x,y)}{r}\right)^{\kappa_v} v(x,r), \ \forall\,r>0,\,\mbox{ for  }\mu-\mbox{a.e.  }x,y\in M.
\end{equation}
 Indeed, since $v$ is non-decreasing in $r$,
  $$v(y,r) \le  v(y,r+d(x,y)).$$
 By \eqref{D2}
$$v(y,r+d(x,y))\le C\,v(x,r+d(x,y))$$
  and by \eqref{dv}
$$v(x,r+d(x,y)) \le C\left(1+\frac{d(x,y)}{r}\right)^{\kappa_v} v(x,r).$$

Therefore, using $(DG)$ and the fact  that $f \in L^2(B(x,r),\mu)$,
\begin{eqnarray*}
&&K^2(f,T)\le C\left[v(x,\sqrt{T}) \right]^{\frac{2}{{p_0}}-1}\|f\|_{2}^2 \\&&\sup_{0\le t\le T}\left(\sum_{k=1}^\infty 
 \left(1+\frac{r(k+1)}{\sqrt{T}} \right)^{\kappa_v\left(\frac{1}{{p_0}}-\frac{1}{{2}}\right)}
 \exp\left(-\frac{(k-1)^2r^2}{4t}\right)+\left(1+\frac{r}{\sqrt{T}} \right)^{\kappa_v\left(\frac{1}{{p_0}}-\frac{1}{{2}}\right)}
 \right)\\&&\le C\left[v(x,\sqrt{T}) \right]^{\frac{2}{{p_0}}-1} \|f\|_{2}^2 \\
 &&\left(\sum_{k=1}^\infty 
 \left(1+\frac{r(k+1)}{\sqrt{T}} \right)^{\kappa_v\left(\frac{1}{{p_0}}-\frac{1}{{2}}\right)}
 \exp\left(-\frac{(k-1)^2r^2}{4T}\right)+\left(1+\frac{r}{\sqrt{T}} \right)^{\kappa_v\left(\frac{1}{{p_0}}-\frac{1}{{2}}\right)}\right)  < +\infty.
\end{eqnarray*}
 
Taking the supremum for $t\in [0,T]$ in \eqref{Ktheta} yields
$$K(f,T)
\le C
\|f\|_r ^\theta
 K(f,T)^{1-\theta},\ \forall\,T>0,$$
 hence, since $K(f,T)$ is finite,
 $$K(f,T)
\le C'
\|f\|_r,\ \forall\,T>0.$$
It follows that
$$\sup_{t>0}\|{v_{\sqrt{t}}^{\frac{1}{{p_0}}-\frac{1}{{2}}}}e^{-tL} f\|_{2}\le C'\|f\|_{{p_0}},$$
 for all $x\in M$, $r>0$, and $f \in L^2(B(x,r),\mu) \cap L^{{p_0}}(B(x,r),\mu)$,
but this estimate does not depend on $x$ and $r$. Therefore
$(vE_{{p_0},{2}})$ holds and   by Proposition \ref{cieplo1} we obtain all estimates
 $(vEv_{{p_0},2, \gamma})$, $\gamma\in\R$, and in particular $(Ev_{{p_0},2})$.

\end{proof}

Proposition~\ref{extrapolation} together with Corollary \ref{implication} yields: 

\begin{proposition}\label{converse} Let $(M,d,\mu,v)$ be as in Proposition  $\ref{fala1}$.  Assume in addition
 that 
 $(M,d,\mu,L)$ satisfies 
 \eqref{DG2},  
and    that $(e^{-tL})_{t>0}$  is uniformly bounded on $L^{1}(M,\mu)$. Let $q>2$. Then
$(vE_{2,q})$ implies  {\rm$(DU\!E^v)$}.
\end{proposition}
\begin{proof}
By duality,  $(vE_{2,q})$ implies $(Ev_{q',2})$.
Proposition \ref{extrapolation} with ${p_0}=1$ and $p=q'$  yields
 $(Ev_{1,2})$, hence   {\rm$(DU\!E^v)$}
by Corollary \ref{implication}.
\end{proof}

As a consequence of  Propositions \ref{equivalence} and
\ref{converse}, we can at last state a converse to Proposition  \ref{dg}. By using Proposition \ref{resolvent2} instead of Proposition \ref{equivalence}, one could replace  in  the following $(GN^v_q)$ by any $(GN^v_{2,q,\beta})$ for $\beta>(\frac{1}{2}-\frac{1}{q})\kappa_v$. One can also combine with Proposition \ref{interpolation} to obtain more results.

\begin{proposition}\label{realconverse} Let $(M,d,\mu)$ be a doubling metric measure space and let $v:M\times \R_+\to \R_+$  satisfy $(A)$, $(D_v)$, and \eqref{D2}.  Assume that 
 $(M,d,\mu,L)$ satisfies 
 \eqref{DG2} and   that $(e^{-tL})_{t>0}$  is uniformly bounded on $L^{1}(M,\mu)$. Let $q>2$. Then
$(GN^v_q)$ implies  {\rm$(DU\!E^v)$}.
\end{proposition}

The first assertion of Theorem \ref{mainDG} follows from Propositions  \ref{dg} and  \ref{realconverse}. 

\bigskip

In the case where $v=V$, one can  use Proposition~\ref{extrapolation} together with \cite[Corollary 4.16]{CS} to obtain $L^p$ uniform boundedness results for semigroups which are not necessary uniformly 
bounded  on $L^{1}(M,\mu)$ or possibly do not even act on this space.  

\begin{proposition}\label{nonmarkov}  Let $(M,d,\mu)$ be a doubling metric measure space.  Assume that 
 $(M,d,\mu,L)$ satisfies 
 \eqref{DG2}.  
Assume further that $(M,\mu,L)$ satisfies   $(GN_{q})$ for some $q$ such that $2<q\le +\infty$ and $\frac{q-2}{q}\kappa<2$, where $\kappa$ is as in \eqref{d}.
	Then $(e^{-tL})_{t>0}$ satisfies  $(V\!E_{2,q})$ and   is uniformly bounded on $L^p(M,\mu)$ for $q'\le p\le q$. For $p$ outside this interval,  $(e^{-tL})_{t>0}$ is uniformly bounded on $L^p(M,\mu)$    if and only if it satisfies  $(V\!E_{2,\widetilde{q}})$, where $\widetilde{q}=\max(p,p')$. 

\end{proposition}

\begin{proof}
According to Proposition \ref{equivalence},  $(GN_{q})$  implies $(V\!E_{2,q})$ and, by duality, $(V\!E_{2,q})$ implies $(EV_{q',2})$. Now  \cite[Corollary 4.16]{CS} yields in particular the uniform boundedness of 
$(e^{-tL})_{t>0}$ on $L^p(M,\mu)$ for $q'\le p\le q$.  Next, for $1\le p <q'$,   if   $(e^{-tL})_{t>0}$ is uniformly bounded on $L^p(M,\mu)$, then Proposition \ref{extrapolation}   yields
 $(EV_{p,2})$ hence $(V\!E_{2,p'})$  and again  \cite[Corollary 4.16]{CS} yields the converse. The case $q<p\le+\infty$ is treated by duality.
\end{proof}

As an application of the above, let us present a result on Schr\"odinger semigroups which applies in particular to the case of negative inverse square potentials (see for instance \cite[p.539]{CS}). 
Compare with \cite[Theorem 11]{DS}.

\begin{theorem}\label{main3}  Let $(M,d,\mu,L)$ be as in Proposition $\ref{nonmarkov}$.  Assume in addition that $L-\mathcal{V}$ is strongly positive in the sense of  $\eqref{av}$.
 Then the Schr\"odinger semigroup $(e^{-t(L-\mathcal{V})})_{t>0}$ also satisfies  estimates $(V\!E_{2,q})$ and is uniformly bounded on $L^p(M,\mu)$ for $q'\le p\le q$. For  $p$ outside this interval,
  $(e^{-t(L-\mathcal{V})})_{t>0}$  is uniformly bounded on $L^p(M,\mu)$ if and only if it satisfies  estimates $(V\!E_{2,\widetilde{q}})$, where $\widetilde{q}=\max(p,p')$.
\end{theorem}
\begin{proof}
Since $(M,d,\mu,L)$ satisfies $(GN_q)$,   we can write, using the first inequality in \eqref{quad},
$$
\|fV_{r}^{\frac{1}{2}-\frac{1}{q}} \|_{q}^2\leq
C(\|f\|_{2}^2+r^2\mathcal{E}(f)) \le 
\frac{C}{\epsilon}(\|f\|_{2}^2+r^2\mathcal{E}_{\mathcal{V}}(f)).
$$
This shows  that  $(M,d,\mu,L-\mathcal{V})$ satisfies estimates $(GN_{q})$ too and  Theorem~\ref{main3}
follows from Proposition~\ref{nonmarkov}, since $(e^{-t(L-\mathcal{V})})_{t>0}$ satisfies \eqref{DG2}, see  \cite[Theorem 3.3]{CS}.
\end{proof}

Let us finish this section by an application of Theorem  \ref{main3} to the Hodge Laplacian.

\begin{theorem}\label{hodge}  Let $M$ be a complete non-compact Riemannian manifold satisfying the doubling volume property $(V\!D)$ and the upper estimate $(DU\!E)$ for the heat kernel on functions.  Let $\mathcal{V}(x)$ be the negative part of a lower bound on the Ricci curvature at $x\in M$. Assume that $\Delta-\mathcal{V}$ is strongly positive.
 Then the heat semigroup on $1$-forms $(e^{-t\vec{\Delta}})_{t>0}$
   is uniformly bounded on $L^p(M,\mu)$  for  $p\in [1,+\infty]$ if $\kappa<2$ and for $p\in (\frac{2\kappa}{\kappa+2},\frac{2\kappa}{\kappa-2})$ if $\kappa\ge 2$ , where $\kappa$ is as in \eqref{d}.
\end{theorem}

\begin{proof}
This is a straightforward consequence of Proposition \ref{dg}, Theorem  \ref{main3}, and of the well-known domination property
$$|\vec{p}_t(x,y)|\le p^\mathcal{V}_t(x,y), \forall x,y\in M,$$
where  $\vec{p}_t$ (resp. $p^\mathcal{V}_t$) is the kernel of $e^{-t\vec{\Delta}}$ (resp. $e^{-t(\Delta-\mathcal{V})}$)
(see for instance \cite{HSU}).
\end{proof}

\begin{rem} Compare Theorem $\ref{hodge}$ with the results in  \cite{Dev}, where, under stronger assumptions, one obtains a stronger conclusion, namely Gaussian estimates on $\vec{p}_t$,  therefore the boundedness of the Riesz transform on $L^p(M,\mu)$ for all $p\in(1,+\infty)$.
Note that, if $\kappa<2$, the above uniform  boundedness of $(e^{-t\vec{\Delta}})_{t>0}$ on $L^p(M,\mu)$ yields such Gaussian estimates.
\end{rem}

\bigskip

\subsection{From local Nash to Gagliardo-Nirenberg}\label{NG}

One can skip this section in a first reading. Indeed, the second assertion from Theorem \ref{mainDG} follows from the first one and Remark \ref{BS} in the case where $\mathcal{E}$ is a strongly local regular Dirichlet form.
And we have seen in Section \ref{LG} that in the same setting all the inequalities $(N^v)$,  $(K\!N^v)$, $(L\!N^v)$ and $(GN^v_q)$, $(K\!GN^v_q)$, $(LS^v_q)$ are equivalent. Our aim here is to prove all this under the  more general assumptions of
Theorem \ref{mainDG}, meaning that we replace some properties of Dirichlet forms by the finite  propagation speed  of the wave equation.

Let $(M,d,\mu)$   be a separable  locally compact  metric measure space and $L$  a  non-negative self-adjoint  operator on $L^2(M,\mu)$  with associated quadratic form $\mathcal{E}$.  Let $\Omega$ be an open subset of $M$. There is a classical notion of a restriction of $L$ to $\Omega$ with Dirichlet boundary conditions  in the case where $\mathcal{E}$ is a strongly local and regular Dirichlet form (see for instance  \cite[section 2.4.1]{GS}).  We are going to start this section by showing that the latter assumption can be replaced with $(M,d,\mu,L)$ satisfying  \eqref{DG2}.
 We initially define the Dirichlet operator $L_\Omega$ on the set $\mathcal{D}_{c}(\Omega)$ of
 all functions $f\in L^2(\Omega,\mu) \subset L^2(M, \mu) $ such that $\mbox{ supp} (f)\subset \Omega$ is compact
  and  $f$
 is in the domain $\mathcal{D}$ of the operator $L$.  Thanks to the following Lemma we will be able to  consider the Friedrichs extension of  $L_\Omega$. With some abuse of notation  we will still denote the resulting self-adjoint operator by $L_\Omega$ and name it the  Dirichlet restriction of  the operator $L$ to the open set $\Omega$. 

 \begin{lemma}\label{dom} Assume that $(M,d,\mu,L)$ satisfies  \eqref{DG2} and let $\Omega$ be an open subset of $M$.
The set $\mathcal{D}_{c}(\Omega)$ is dense in $L^2(\Omega,\mu)$. In addition, the quadratic form $\mathcal{E}$ restricted 
 to $\mathcal{D}_{c}(\Omega)$ is closable in $L^2(\Omega,\mu)$ and the domain of its closure contains the set $\mathcal{F}_c(\Omega)$.
 \end{lemma}
 
 \begin{proof}
  The set of all  functions $f\in L^2(\Omega,\mu) $ with compact support is dense in  $L^2(\Omega,\mu)$,  so that it is enough to show that any such function 
  is in the closure of $\mathcal{D}_{c}(\Omega)$. 
 
 To do this consider the family 
 $$\frac{\sin^2 r\sqrt L}{r^2L}f$$ 
 for $r>0$.
 By spectral theory, the above expression converges to $f$ in $L^2(M,\mu)$ when $r$ goes to zero. 
 Next, note that the Fourier transform of the function $\lambda \to \frac{\sin^2 \lambda}{\lambda^2}$
 is supported in the interval $[-2,2]$. Hence, if $d(\mbox{supp}( f), \Omega^c)=\varepsilon >0$ and $2r \le \varepsilon$, it follows from  Lemma \ref{nos} that 
 $$\mbox{supp} \,\frac{\sin^2 r\sqrt L}{r^2L}f \subset (\mbox{supp} (f))_\varepsilon\subset\Omega.$$
 In particular,  $$\mbox{supp}\, \frac{\sin^2 r\sqrt L}{r^2L}f $$ is compact. Also, by spectral theory,
 $$L\frac{\sin^2 r\sqrt L}{r^2L}f $$ is well-defined as a function in  $L^2(M,\mu)$ , that is
 $\frac{\sin^2 r\sqrt L}{r^2L}\in\mathcal{D}$.
 Therefore $ \frac{\sin^2 r\sqrt L}{r^2L}f \in \mathcal{D}_{c}(\Omega)$. This proves that $\mathcal{D}_{c}(\Omega)$
 is dense in $L^2(\Omega)$. Now consider the operator $L_\Omega$ which is the restriction of  $L$ 
 to the set $\mathcal{D}_{c}(\Omega)$. Note that $L_\Omega$ is symmetric. 
 By Friedrichs's theorem, the quadratic form corresponding to the operator $L_\Omega$ is closable.
\medskip 

Let now $f\in\mathcal{F}_c(\Omega)$.  
Since
$$\mathcal{E}\left(\frac{\sin^2 r\sqrt L}{r^2L}f\right)= \left\langle
 L\frac{\sin^2 r\sqrt L}{r^2L}f ,\frac{\sin^2 r\sqrt L}{r^2L} f \right\rangle = \left\|\frac{\sin^2 r\sqrt L}{r^2L}L^{1/2}f\right\|_2^2,$$
one sees that
$$ \lim_{r\to 0^+} \mathcal{E}\left(\frac{\sin^2 r\sqrt L}{r^2L}f\right) =  \mathcal{E}(f)$$
and  the same argument 
which we used in the above paragraph to show that $\mathcal{D}_{c}(\Omega)$ is dense in $L^2(\Omega,\mu)$
can be used to prove  that the closure of $\mathcal{D}_{c}(\Omega)$ with respect to the norm corresponding to $\mathcal{E}$
contains the set $ \mathcal{F}_c(\Omega) $.
This shows that the closures of $\mathcal{E}$ restricted to $\mathcal{D}_{c}(\Omega)$  and $ \mathcal{F}_c(\Omega) $
coincide. 
\end{proof}

In the next lemma we discuss the relation between the wave propagators for the operator $L$ and
for its Dirichlet restriction $L_{B(x,r_0)}$ to an open ball $B(x,r_0)$.

\begin{lemma}\label{wediri} Assume that $(M,d,\mu,L)$ satisfies  \eqref{DG2}. Let $x\in M$, $r_0>0$, and $0<\epsilon<r_0$.
Then 
$$
\cos (r \sqrt{L}) \chi_{B(x,\epsilon)}= \cos (r \sqrt{L_{B(x,r_0)}} ) \chi_{B(x,\epsilon)}
$$
for all   $r$ such that $0<r < r_0-\epsilon$.
\end{lemma}
\begin{proof}  It is enough to prove that $$
\cos (r \sqrt{L}) f= \cos (r \sqrt{L_{B(x,r_0)}} ) f
$$ for $f\in L^2(B(x,\epsilon),\mu)$. According to Lemma \ref{dom}, one can assume in addition $f\in \mathcal{D}$.
Since  $\mbox{supp} (f )\subset B(x,\epsilon)$, the finite propagation speed for the wave equation associated with $L$ yields
$$\mbox{supp}\, \cos ( r \sqrt{L}) f \subset B(x,r+\epsilon)\subset B(x,r_0)$$
for all $r$ such that $0<r < r_0-\epsilon$.
 Hence, for $0<r < r_0-\epsilon$, 
$$
L\cos ( r \sqrt{L})f=L_{B(x,r_0)}\cos (r \sqrt{L})f=-\partial_r^2 \cos (r \sqrt{L})f,
$$
i.e. the function $F(\cdot,r)=\cos (r \sqrt{L})f $ is a solution of the wave equation 
$$L_{B(x,r_0)}F=-\partial_r^2F.$$
The standard argument shows that for  any solution of the above equation the energy function $E(r) = |F_r|^2+\langle L_{B(x,r_0)}F, F\rangle $ is 
conserved. The energy conservation implies the  uniqueness of the solutions of the wave equation
which in turn implies the claim.
\end{proof}

In the following result we show that under an $L^1$ uniform boundedness assumption on the semigroup $(e^{-tL})_{t>0}$,  $(LN^v_\alpha)$ implies 
$(GN^v_q)$ for $q>2$  small enough in terms of $\alpha$.  Together with the results of Sections \ref{GNN} and \ref{LG}, this yields
the full equivalence, under the following assumptions, of $(N^v)$,  $(K\!N^v)$, $(L\!N^v)$, and, for $q>2$  small enough, of $(GN^v_q)$, $(K\!GN^v_q)$, $(LS^v_q)$.

\begin{proposition}\label{ngn} Let $(M,d,\mu)$ be a doubling metric measure space, $v:M\times \R_+\to \R_+$ satisfy $(A)$, $(D_v)$, and \eqref{D2}, and $L$ a  non-negative self-adjoint operator on $L^2(M,\mu)$ such  that $(M,d,\mu,L)$ satisfies  \eqref{DG2}. Assume  that the semigroup $\exp(-t L_{B(x,r)})$ is bounded 
on $L^1(B(x,r),\mu)$ uniformly in $t>0$, $x\in M$, and $r>0$. Then $(LN^v_\alpha)$ implies $(GN^v_q)$ for all $q$ such that $2<q<+\infty$ and $\frac{q-2}{q} < \alpha$.
\end{proposition}
\begin{proof}
Condition  $(LN^v_\alpha)$
can be stated in the following way
\begin{equation}\label{nana}
\|f\|_{2}^{2(1+\alpha)}\leq   \frac{Cr^2}{v^{\alpha}_{r}(x) }\|f\|_{1}^{2\alpha}
\mathcal{E}_{L_{B(x,3r)}+r^{-2}I}(f),\ \forall\, \,x\in M,\,r>0,\,
 f\in  \mathcal{F}_c(B(x,3r)).
\end{equation}
Now since by assumption the semigroup $(\exp(-s L_{B(x,3r)}))_{s>0}$ is uniformly bounded on $L^1(B(x,3r),\mu)$, this also holds for $(\exp(-s( L_{B(x,3r)+r^{-2}I}))_{s>0}$.
Then by Nash's classical argument, \eqref{nana} implies
$$\|\exp(-s( L_{B(x,3r)+r^{-2}I}))\|_{1\to\infty}\le \frac{Cr^{2/\alpha}}{ v_{r}(x)} s^{-1/\alpha}, \ \forall\,x\in M,\,r,s>0,$$
hence
$$\|\exp(-s( L_{B(x,3r)}+r^{-2}I))\|_{2\to q}\le \left( \frac{C'r^{2/\alpha}}{ v_{r}(x)}\right)^{\frac{1}{2}-\frac{1}{q}} s^{-\frac{1}{\alpha}\left(\frac{1}{2}-\frac{1}{q}\right)}, \ \forall\,x\in M,\,r,s >0.$$
In the last two inequalities as well as in the sequel, the $L^p$ norms have to be understood on $L^p(B(x,3r),\mu)$.
Let $\lambda>0$.  By integrating in $s>0$ the function $s\to e^{-s}e^{-s\lambda H}$, with $$H= L_{B(x,3r)}+r^{-2}I,$$
we obtain, for $\frac{q-2}{q} < \alpha$,
$$\|\left( I+\lambda( L_{B(x,3r)}+r^{-2}I)\right)^{-1}\|_{2\to q}\le \left( \frac{C''r^{2/\alpha}}{ v_{r}(x)}\right)^{\frac{1}{2}-\frac{1}{q}}  \lambda^{-\frac{1}{\alpha}\left(\frac{1}{2}-\frac{1}{q}\right)} , \ \forall\,\lambda>0,$$
where $C''$ does not depend on $t>0$ or $x\in M$.

Taking $\lambda=2r^2$ yields
$$\|\left( I+ r^2L_{B(x,3r)}\right)^{-1}\|_{2\to q}\le  \frac{C}{ v^{\frac{1}{2}-\frac{1}{q}}_{r}(x)}, \ \forall\,r>0,\,x\in M.$$

Set $\Psi(\lambda)= \frac{\lambda -\sin \lambda}{\lambda^3}$.
Note that by the spectral theorem, the function $(1+\lambda^2)\Psi(\lambda)$ and its inverse being bounded,
 \begin{equation}\label{psi} \|(I+r^2L_{B(x,3r)})^{-1}\|_{2 \to q} \simeq
 \|\Psi\left(r\sqrt{L_{B(x,3r)}}\right)\|_{2 \to q},
  \end{equation}
  uniformly in $r>0$ and $x\in M$.
  Therefore
  $$\|\Psi\left(r\sqrt{L_{B(x,3r)}}\right)\|_{2 \to q}\le  \frac{C}{ v^{\frac{1}{2}-\frac{1}{q}}_{r}(x)}$$
  and also
  $$\|\Psi\left(r\sqrt{L_{B(x,3r)}}\right)\chi_{B(x,r)}\|_{2 \to q}\le \frac{C}{ v^{\frac{1}{2}-\frac{1}{q}}_{r}(x)},\ \forall\,x\in M,\,t>0.$$
 Now Lemma \ref{wediri} with $\epsilon=r$, $r_0=3r$ yields
  $$\cos(r\sqrt{L})\chi_{B(x,r)}=\cos\left(r\sqrt{L_{B(x,3r)}}\right)\chi_{B(x,r)}.$$
 Since the Fourier transform of $\Psi$ is supported in the interval $[-1,1]$, one can use  
 formula \eqref{formu} and conclude
 $$\Psi(r\sqrt{L})\chi_{B(x,r)}=\Psi\left(r\sqrt{L_{B(x,3r)}}\right)\chi_{B(x,r)}.$$
Hence
\begin{equation}\label{shahe}
\|\Psi(r\sqrt{L})\chi_{B(x,r)}\|_{2 \to q} \le  \frac{C}{ v^{\frac{1}{2}-\frac{1}{q}}_{r}(x)}, \ \forall\,x\in M,\,r>0.
\end{equation}
Since by Lemma \ref{nos},  
\begin{equation*}
 \mbox{supp}\, {\Phi(r\sqrt{L})} \subseteq D_{r},
\end{equation*}
one may write
$$\Psi(r\sqrt{L})\chi_{B(x,r)}=\chi_{B(x,2r)}\Psi(r\sqrt{L})\chi_{B(x,r)}$$
and, by  \eqref{D2},
\begin{eqnarray*}
\|v^{\frac{1}{q}-\frac{1}{2}}_{r}\Psi(r\sqrt{L})\chi_{B(x,r)}\|_{2 \to q}&=&\|v^{\frac{1}{q}-\frac{1}{2}}_{r}\chi_{B(x,2r)}\Psi(r\sqrt{L})\chi_{B(x,r)}\|_{2 \to q}\\
&\le& v^{\frac{1}{q}-\frac{1}{2}}_{r}(x) \|\chi_{B(x,2r)}\Psi(r\sqrt{L})\chi_{B(x,r)}\|_{2 \to q}\\
&=& v^{\frac{1}{q}-\frac{1}{2}}_{r}(x) \|\Psi(r\sqrt{L})\chi_{B(x,r)}\|_{2 \to q}.
\end{eqnarray*}
Thus it follows from \eqref{shahe} that
$$\|v^{\frac{1}{q}-\frac{1}{2}}_{r}\Psi(r\sqrt{L})\chi_{B(x,r)}\|_{2 \to q} \le  C, \ \forall\,x\in M,\,r>0.$$
Lemma~\ref{fala2} yields
$$\|{v_{r}^{\frac{1}{q}-\frac{1}{{2}}}}\Psi(r\sqrt{L})\|_{2 \to q} \le C.$$
Using  the $L^2$-boundedness  of
$\Psi^{-1}(r\sqrt{L})(I+r^2L)^{-1}$,
we obtain
$$\|{v_{r}^{\frac{1}{q}-\frac{1}{{2}}}}(I+r^2L)^{-1}\|_{2 \to q} \le C,$$
that is, (\ref{VE2q}).
Since $v$ satisfies $(D_v)$, we can use Proposition \ref{resolvent2} and  $(GN^v_q)$ follows.

\end{proof}

 Propositions \ref{KL}, \ref{ngn}, and \ref{realconverse}   yield the following. Together with  Proposition \ref{KL}, this gives finally as a by-product a converse to Proposition \ref{Nash1}.

\begin{proposition}\label{Nana} Let $(M,d,\mu)$ be a doubling metric measure space, $v:M\times \R_+\to \R_+$  satisfy $(A)$, $(D_v)$, and \eqref{D2}, and $L$ a  non-negative self-adjoint operator on $L^2(M,\mu)$ such that $(M,d,\mu,L)$ satisfies $(DG)$.
Assume that the semigroup $\exp(-t L_{B(x,r)})$ is bounded 
on $L^1(B(x,r),\mu)$ uniformly in $t>0$, $x\in M$, and $r>0$. Then $(LN^v)$ implies  $(DU\!E^v)$.
\end{proposition}

Verifying that the semigroup $\exp(-t L_{B(x,r)})$ is bounded 
on $L^1(B(x,r),\mu)$ uniformly in $t>0$, $x\in M$, and $r>0$ is possibly not always an easy task.
The final result of this section shows that the situation simplifies if
the semigroup $\exp(-t L)$ is positivity preserving. 

\begin{proposition}\label{posi} Suppose that   $(M,d,\mu,L)$ satisfies $(DG)$, that the semigroup $\exp(-t L)$ is positivity preserving and 
uniformly bounded on  $L^p(M,\mu)$ 
 for some $1 \le p \le \infty$.  
Then  the semigroups $\exp(-t L_{B(x,r)})$ are bounded 
on $L^p(B(x,r),\mu)$ uniformly in $t>0$, $x\in M$, and $r>0$. 
\end{proposition}

\begin{proof} We observe that the semigroup $\exp(-t L_{B(x,r)})$ is also positivity preserving and  dominated by the original semigroup 
 $\exp(-t L)$. That is if $f\in L^2(\Omega,\mu)$ and $f\ge 0 $ a.e. then 
 $$
 0\le \exp(-t L_{B(x,r)})f\le \exp(-t L)f, \ \mu-\rm{a.e.}.
 $$
 Indeed it is not difficult to check that the proof of the similar property described in 
 \cite[(4.6), Theorem 4.2.1] {Are} or \cite[Proposition 2.1]{ER}  
 can be adapted to our setting (see also \cite[Proposition 4.23]{O}). Now the uniform boundedness of 
 the semigroup $\exp(-t L_{B(x,r)})$ is a straightforward consequence of the domination property described
 above, positivity and uniform boundedness on $L^p(M,\mu)$ of the initial semigroup $\exp(-t L)$.
\end{proof}

We can finally state:

\begin{proposition}\label{newmainDG} Let $(M,d,\mu)$ be a doubling metric measure space, $v:M\times \R_+\to \R_+$  satisfy $(A)$, $(D_v)$, and \eqref{D2},  and $L$ a  non-negative self-adjoint operator on $L^2(M,\mu)$.  Assume that 
 $(M,d,\mu,L)$ satisfies 
the Davies-Gaffney
estimate \eqref{DG2}  
and   that the semigroup $(e^{-tL})_{t>0}$ is  uniformly bounded on $L^{1}(M,\mu)$ and positivity preserving. Then $(LN^v)$ implies the upper bound
$(DU\!E^v)$.
\end{proposition}

\bigskip

 The second assertion of Theorem \ref{mainDG} follows from Propositions \ref{Nash22} and \ref{newmainDG}.
 
 \bigskip

Let us conclude by a remark on on-diagonal sub-Gaussian heat kernel estimates. 
They are typically of the type 
$$p_t(x,x)\le \frac{C}{V(x,t^{1/m})},$$
for $m>2$, and they can be characterised by Faber-Krahn  type inequalities together with exit time estimates (see
for instance \cite{kig},  \cite{GH}, and the references therein). One could certainly replace  these Faber-Krahn inequalities by inequalities similar to the ones we have considered in this paper,
simply by taking $v(x,r)=V(x,r^{2/m})$. But if we try to get rid of the exit time estimates and make our theory fully work as in the Gaussian case, we  encounter two obstacles: we lose  
 \eqref{D2}, and more importantly our main tool,  namely the finite speed propagation of the wave equation.
The only hope would be to exploit the existence of  specific cut-off functions related to the exponent $m$ as in \cite{AB}.

\bigskip

{\bf Acknowledgements: } The authors thank Alexander Bendikov, Gilles Carron, Alexander Grigor'yan, El Maati Ouhabaz and Laurent Saloff-Coste for useful remarks or suggestions on this paper. Thanks are also due to Li Huaiqian for reading carefully the manuscript and noticing some typos.


\end{document}